\documentclass[11pt,a4paper]{amsart}

\usepackage[utf8]{inputenc}
\usepackage{amssymb,amsmath,amsbsy,amscd,amsfonts,amsthm,latexsym}
\usepackage[mathcal]{eucal}
\usepackage{times, euler}
\usepackage{eurosym}

\usepackage{stackrel, enumerate}
\usepackage{mathabx}
\usepackage[usenames]{color}
 
\usepackage[top=3cm,bottom=3cm,left=2cm,right=2cm,footskip=17mm]{geometry}

\usepackage{hyperref}

\hypersetup{
    colorlinks=true,
    linkcolor=blue,
    filecolor=cyan,      
    urlcolor=blue,
    citecolor=magenta
}

\definecolor{darkgreen}{rgb}{0,0.6,0}

\newcommand{\uri}[1]{{\color{black}{#1}}}

\usepackage{xy}
\xyoption{all}

\numberwithin{equation}{section}

\theoremstyle{plain}
    \newtheorem{theorem}[equation]{Theorem}
    \newtheorem{lemma}[equation]{Lemma}
    \newtheorem{lemma-definition}[equation]{Lemma-Definition}
    \newtheorem{corollary}[equation]{Corollary}
    \newtheorem{proposition}[equation]{Proposition}
    \newtheorem{conjecture}[equation]{Conjecture}

\theoremstyle{definition}
    \newtheorem{definition}[equation]{Definition}

    \newtheorem{remark}[equation]{Remark}

   \newtheorem{thmABC}{Theorem}

\newcommand{\C}{\mathbb{C}}
\newcommand{\N}{\mathbb{N}}
\newcommand{\Z}{\mathbb{Z}}

\newcommand{\F}{\mathbb{F}}

\renewcommand{\phi}{\varphi}
\renewcommand{\epsilon}{\varepsilon}

\newcommand{\dd}{d}  
\newcommand{\restrict}{\big{\vert}}
\newcommand{\argument}{\hspace{2pt}\underbar{\phantom{g}}\hspace{2pt}}

\newcommand{\germ}{\mathfrak}

\newcommand{\ol}{\overline}

\newcommand{\ot}{\otimes}
\newcommand{\ti}{\widetilde}

\DeclareMathOperator{\Irr}{Irr}
\DeclareMathOperator{\Hom}{Hom}
\DeclareMathOperator{\End}{End}
  
\DeclareMathOperator{\Ker}{Ker}
\DeclareMathOperator{\Ima}{Im}

\DeclareMathOperator{\Ad}{Ad}
\DeclareMathOperator{\ad}{ad}

\DeclareMathOperator{\Ind}{Ind}
\DeclareMathOperator{\ind}{ind}

\DeclareMathOperator{\GL}{GL}

\DeclareMathOperator{\Sp}{Sp}
\DeclareMathOperator{\Rep}{\mathrm{Rep}}

\DeclareMathOperator{\infl}{inf}

\DeclareMathOperator{\Span}{Span}

\DeclareMathOperator{\pind}{i}
\DeclareMathOperator{\pres}{r}

\DeclareMathOperator{\pInd}{I}
\DeclareMathOperator{\pRes}{R}

\DeclareMathOperator{\diag}{diag}

\DeclareMathOperator{\Gal}{Gal}

\newcommand{\Mat}{\mathrm{M}}     % the full algebra of matrices

\newcommand{\cA}{\mathcal{A}}  	% subalgebra
\newcommand{\gA}{A}                       % the corresponding subgroup

\newcommand{\cZ}{\mathcal{Z}}  	% trace complement of A 
\newcommand{\gZ}{Z}                        % the corresponding subgroup

\newcommand{\cL}{\mathcal{L}}  	% Lagrangian

\newcommand{\charpoly}{\operatorname{char}} % characteristic polynomial

\renewcommand{\aa}{a}  	
\newcommand{\bb}{b}  	
\newcommand{\cc}{c}  	
\newcommand{\xx}{x}  	
\newcommand{\yy}{y}
\newcommand{\zz}{z}  	
\newcommand{\MM}{\xi}  	
  	
\renewcommand{\o}{\mathfrak o}	   %  (local) ring of integers
\renewcommand{\O}{\mathfrak O} 	   %  (local) Unramified extension of local ring of integers

\newcommand{\p}{\mathfrak p}	           %  (local) ring of integers
\renewcommand{\l}{\ell}                         % level  	

 \newcommand{\kk}{\mathbf{k}}  	

\newcommand{\nn}{n}  	
\newcommand{\mm}{m}  	

\DeclareMathOperator{\s}{\sigma}  	
\DeclareMathOperator{\n}{\nu}

\DeclareMathOperator{\tr}{Tr}  	 % Trace

\DeclareMathOperator{\GR}{\mathcal{R}}  	
\DeclareMathOperator{\GS}{\mathcal{S}}  	
\DeclareMathOperator{\GN}{\mathcal{N}}

\newcommand{\modd}{\, \mathrm{mod} \, } 

\newcommand{\nmult}{\circ}

\newcommand{\Ow}{\mathfrak o}	  %  (local) ring of integers
\newcommand{\T}{\mathfrak O} 	   %  (local) Unramified extension of local ring of integers

\newcommand{\cF}{\mathcal{F}}     % functor
\newcommand{\LL}{S}  	

\newcommand{\ffbar}[1]{\ol{\phi}({#1})}     
\newcommand{\wt}{\widetilde}
\newcommand{\TL}{Q}


\begin{document}

\begin{abstract}  In his seminal Lecture Notes in Mathematics published in 1981, Andrey Zelevinsky introduced a new family of Hopf algebras which he called {\em PSH-algebras}. These algebras were designed to capture the representation theory of the symmetric groups and of classical groups over finite fields. The gist of this construction is to translate representation-theoretic operations such as induction and restriction and their parabolic variants to algebra and coalgebra operations such as multiplication and comultiplication. The Mackey formula, for example, is then reincarnated as the Hopf axiom on the algebra side. In this paper we take substantial steps to adapt these ideas for general linear groups over compact discrete valuation rings. We construct an analogous \uri{bialgebra} that contains a large PSH-algebra that extends Zelevinsky's algebra for the case of general linear groups over finite fields. We prove several base change results relating algebras over extensions of discrete valuation rings.           
\end{abstract}

\title[An inductive approach to representations of general linear groups]{An inductive approach to representations of general linear groups \\ over compact discrete valuation rings}

\author{Tyrone Crisp}
\address{Department of Mathematics and Statistics, University of Maine, Orono, ME 04469-5752, USA}
  \email{tyrone.crisp@maine.edu}

\author{Ehud Meir}
\address{Institute of Mathematics, University of Aberdeen, Fraser Noble Building, Aberdeen AB24 3UE, UK}
  \email{meirehud@gmail.com}

\author{Uri Onn}
\address{Mathematical Sciences Institute, The Australian National University, Canberra, Australia}
  \email{uri.onn@anu.edu.au}

\date{\today}
\maketitle

\setcounter{tocdepth}{1}
\tableofcontents
\thispagestyle{empty}

%\newpage

\section{Introduction}\label{sec:intro}

\subsection{Overview} Studying the evolution of representations across families of groups under induction and restriction functors has been a recurrent theme in representation theory.  Such an approach proved to be extremely fruitful for studying representations of the symmetric groups, for example, and many other natural infinite families of groups and algebras. Harish-\uri{C}handra functors form a variation on this theme adapted to reductive groups over finite or local fields. While usual induction and restriction functors are inefficient for keeping track of representations of the latter, Harish-Chandra functors appear to be powerful for organising and studying these categories. More precisely, let $G$ be a finite group and let $\Rep(G)$ denote the category of its finite dimensional complex representations. Let $L < G$ be a subgroup and let $\C L \subset \C G$ be the corresponding complex group algebras. Induction from $L$ to $G$ is realised by tensoring with $\C G$ viewed as a  $\C G$-$\C L$ bimomodule, and similarly for restriction with the roles of $L$ and $G$ interchanged. Harish-Chandra induction in this abstract context involves another subgroup $U < G$, which is normalised by $L$, that is used to \lq slow down\rq~ and therefore gain a better control on the effect of these functors. Concretely it is defined by tensoring with the $\C G$-$\C L$ bimodule $\C G e_U$ where $e_U=|U|^{-1}\sum_{u \in U} u$, the idempotent associated with $U$. The adjoint functor is given by tensoring with the $\C L$-$\C G$ bimodule $e_U\C G$.  Applying these functors to categories of representations of reductive groups over finite fields gives rise to an elegant description of the relations between the functors (Mackey formula) and of those representations which cannot be obtained by induction (cuspidal representations). 

A variant of these functors was introduced in \cite{CMO1} with a view toward reductive groups over finite rings. The idea is to consider another subgroup $V < G$ which is normalised by $L$ and consider the functors of tensoring with the $\C G$-$\C L$ bimodule $\C G e_U e_V$ and $e_U e_V \C G$. See \cite{CMO1} for general properties of these functors and \cite{CMO2} for analysis of principal series representations of $\GL_n$ over finite rings. The present paper utilises these functors to give a deeper insight to general linear groups over local rings, as we now describe.     

Suppose that $\{G_n \mid \nn \in \N\}$ is a sequence of finite groups. Given functors between the categories of complex representations of these groups that can be realised by tensoring with bimodues 
\[
\begin{split}
&\pind_{\mm,\nn}: \Rep(G_{\mm}) \times \Rep(G_{\nn}) \to \Rep(G_{\mm+\nn}), 
\\ &\pres_{\mm,\nn} : \Rep(G_{\mm+\nn}) \to \Rep(G_{\mm}) \times \Rep(G_{\nn}), 
\end{split}
\]
one can define a ring and coring structure on $\GR=\oplus_{\nn \ge 0} \GR_\nn$, where $\GR_\nn$ is the Grothendieck group of $\Rep(G_{\nn})$ and $R_0:=\Z$. The multiplication $\nmult:\GR_\mm \otimes \GR_{\nn} \to \GR_{\mm+\nn}$ is defined on representatives by $\rho \otimes \sigma \mapsto \pind_{\mm,\nn}(\rho \times \sigma)$, and the comultiplication $\Delta:  \GR_\nn \rightarrow \oplus \GR_{\mm} \otimes \GR_{\nn-\mm}$ by $\Delta(\chi)=\oplus \pres_{\mm,\nn-\mm}(\chi)$, where the sum ranges over all $0 \le \mm \le \nn$; both maps are then extended linearly.  In~\cite{Zelevinsky}, Andrei Zelevinsky shows that when the groups~$G_\nn$ are the symmetric groups and the functors $\pind$ and $\pres$ are the usual induction and restriction with respect to the coherent embeddings of the symmetric groups in the natural way, one gets a Hopf $\Z$-algebra with a very rigid structure which he called {\em PSH-algebra}, axiomatised in \cite[\S1.4]{Zelevinsky}. The specific PSH-algebra $\mathcal{S}$ arising from representations of the symmetric groups contains a unique primitive irreducible element and it turns out that there is a unique PSH-algebra with this property up to two isomorphisms; see \cite[\S 3.1(g)]{Zelevinsky}. We call this algebra the {\em elementary} PSH-algebra. It transpires that the elementary PSH-algebra completely captures the representation theory of the symmetric groups; see~\cite[\S 6]{Zelevinsky}. 

Zelevinsky also considers the groups $G_\nn=\GL_\nn(\F_q)$, with $\F_q$ a finite field. In that case the functors~$\pind$ and~$\pres$ are interpreted as Harish-Chandra induction and restriction, respectively, with respect to the natural coherent embeddings of the groups $L=G_{\mm} \times G_{\nn}$ in $G=G_{\mm+\nn}$,  and appropriate unipotent subgroups $U_{\mm,\nn}$. He shows that the resulting structure is a PSH-algebra as well. Moreover, this algebra is an infinite tensor product of copies of the elementary PSH-algebra, indexed by cuspidal representations of the~$G_\nn$'s. The latter are in bijection with Frobenius orbits in the direct limit $\lim_n{\left(\F_{q^n}^\times\right)^\vee}$, the limit taken over the Pontryagin duals of the multiplicative groups of finite field extensions of $\F_q$ with respect to the transpose of the norm maps~ \cite[\S1]{Macdonald}.    

The main goal  of the present paper is to generalise this construction for the groups $G_\nn=\GL_\nn(\o)$, where $\o$ is a compact discrete valuation ring, with respect to the above-mentioned variant of Harish-Chandra functors, which we now explain in detail.  

\subsection{Main results}\label{subsec:main.results} 

Let $\o$ be a complete discrete valuation ring with maximal ideal $\p=(\pi)$ and finite residue field~$\kk$, and let $\o_\l$ stand for the finite quotient $\o/\p^\l$. For every composition $(\nn_1,\ldots,\nn_r)$ of a natural number $\nn$ we fix an embedding $\GL_{n_1} \times \cdots \times \GL_{\nn_r} \hookrightarrow \GL_\nn$ as the subgroup of invertible block-diagonal matrices in a natural way. Let $U_{\nn_1,\ldots,\nn_r}$ denote the upper uni-triangular subgroup with block sizes $\nn_i \times \nn_j$ and let $V_{\nn_1,\ldots,\nn_r}$ be its transpose. 
%Whenever the composition is clear from the context we omit it from the notation, writing simply $U$ and $V$. 
Let $\Rep\left(\GL_n(\o)\right)$ denote the category of finite dimensional smooth complex representations of $\GL_n(\o)$. The natural quotient maps induce inclusions $\Rep\left(\GL_n(\o_\l)\right) \hookrightarrow \Rep\left(\GL_n(\o_{\l+1})\right)$ and 
%we have an equality 
$\Rep\left(\GL_n(\o)\right)$ is the direct limit of this directed system.
For every $\nn \ge 1$ let $\GR_\nn^\l=\GR_n^{\o,\l}$ be the Grothendieck group of $\Rep(\GL_\nn(\o_\l))$ and set $\GR_0^{\l}=\Z$. Let  
\[
\GR^\o= \bigoplus_{\nn \ge 0} \GR_n^\o= \lim_{\substack{\longrightarrow \\ \l}} \bigoplus_{\nn \ge 0}  \GR_\nn^{\o,\l}.  
\]

We define $\nmult: \GR^\o \otimes \GR^\o \rightarrow \GR^\o$ and $\Delta: \GR^\o \rightarrow \GR^\o \otimes \GR^\o$ as follows. Let $G_n^\l=\GL_\nn(\o_\l)$. For the graded summands of level $\l$ let 
\[
\begin{split}
\nmult_\l: &\GR^{\o,\l}_{\mm} \otimes \GR^{\o,\l}_{\nn} \longrightarrow \GR^{\o,\l}_{\mm+\nn}, \qquad \rho \nmult_\l \sigma=
\C G^\l_{\mm+\nn} e_{U^\l_{\mm,\nn}}e_{V^\l_{\mm,\nn}}\otimes_{\C (G^\l_{\mm} \times G^\l_{\nn})}(\rho\times \sigma) \\
\Delta_\l: &\GR^{\o,\l}_{\nn}  \longrightarrow  \GR^{\o,\l}_{\mm} \otimes \GR^{\o,\l}_{\nn-\mm}, \qquad \Delta_\l(\chi)=\bigoplus_{0 \le \mm \le \nn} e_{U^\l_{\mm,\nn}}e_{V^\l_{\mm,\nn}}\C G^\l_{\mm+\nn} \otimes_{\C G^\l_{\mm+\nn}} (\chi).
\end{split}
\] 
By Proposition~\ref{prop:GR-changing-l} the multiplications $\nmult_\l$ for $\l \in \N$ are compatible with the reduction maps (and similarly for the comultiplications $\Delta_\l$), and therefore give rise to a well defined associative (and coassociative) structures on~$\GR^\o$. {These operations satisfy all of the PSH-algebra axioms except, possibly, the compatibility of the multiplication and the comultiplication (i.e. the Hopf axiom). We make the following conjecture.
\begin{conjecture}\label{conjecture:main}
The above operations equip $\GR^{\o}$ with a structure of a PSH-algebra.
\end{conjecture}
If this conjecture holds it would impose a very rigid structure on $\GR^{\o}$, thereby reducing the study of smooth representations of $\GL_n(\o)$ to study of the cuspidal representations (with respect to these induction functors) and of the (well-understood) unique elementary PSH-algebra. When $\l=1$ the local ring $\o_1$ is a finite field and these induction and restriction functors are isomorphic to the usual Harish-Chandra functors \cite[Theorem 2.4]{Howlett-Lehrer_HC}. So in this case the subalgebra $\GR^{\o,1}$ coincides with the PSH-algebra defined by Zelevinsky. }

Let $K^\l_n=\Ker \left(\GL_\nn(\o) \to \GL_\nn(\o_\l)\right)$ denote the $\l$-th principal congruence subgroup. We write $K^{\l}$ for~$K^{\l}_n$ whenever the index $n$ is clear from the context. 
A continuous (or smooth) irreducible representation of $\GL_\nn(\o)$ always factors through a finite quotient $\GL_n(\o)/K^{\l+1}$ for some $\l \geq 0$, the minimal such $\l$ is the {\em level}  of the representation.  
For every $\ell \ge 1$ the quotient $K^{\l}/K^{\l+1}$ is isomorphic to the additive group \uri{$\Mat_\nn(\kk)$. The latter is self dual, namely, $\Mat_\nn(\kk) \cong \Mat_\nn(\kk)^\vee=\Hom_\Z(\Mat_{\nn}(\kk),\C^\times)$ and this isomorphism is equivariant; see \S\ref{subsec:Pontryagin.Dual}. We therefore identify $\Mat_\nn(\kk)$ with its Pontryagin dual.}
%\footnote{TC: I propose that we fix a level $\l$, always work over $\o_\l$, and then define $K_n^j$ to be the quotient of the $j$th congruence subgroup by the $\l$th, for $j=1,\ldots,l-1$. } 
Using these identifications we write $\GR^{\o,\l}_{[\ol{\MM}]}$ for the subgroup of $\GR^{\o,\l}$ generated by those irreducible representations whose restriction to \uri{$K^{\l}/K^{\l+1}$} is an orbit of characters that correspond to the similarity class of~$\ol{\MM} \in \Mat_n(\kk)$. 

\smallskip

Let $\ol{f}\in \o_1[t]$ be an irreducible polynomial of degree $d$. 
We write 
$$\GR^{\o,\l}(\ol{f}) = \bigoplus_{m\geq 0}\bigoplus_{\mathrm{char}[\ol{\MM}]=\ol{f}^m}\GR^{\o,\l}_{[\ol{\MM}]},$$ 
where $\mathrm{char}$ denotes the characteristic polynomial.
\begin{thmABC}[Primary decomposition] \label{thm:primary}
The multiplication $\nmult$ induces an isomorphism of algebras and coalgebras
\[
\bigotimes_{\ol{f}  \in \o_1[t]~\text{irreducible}} \GR^{\o,\l}(\ol{f})\to \GR^{\o,\l}.
\]
This isomorphism gives a one-to-one correspondence between tensor products of
irreducible elements on the left and irreducible elements in $\GR^{\o,\l}$ on the right.
\end{thmABC}

Since the tensor product of PSH-algebras is again a PSH-algebra, the above theorem has the following corollary.
\begin{corollary}\label{cor:hopf.axiom.primary} 
Let $\ol{f_1},\cdots, \ol{f_n}$ be distinct irreducible polynomials in $\o_1[t]$. Let $\nu_i\in \GR^{\o,\l}(\ol{f_i})$ for $i=1,\ldots,n$. Then $$\Delta(\nu_1\circ\cdots\circ\nu_n) = \Delta(\nu_1)\circ\cdots\circ \Delta(\nu_n).$$ In particular, Conjecture \ref{conjecture:main} is valid in $\GR^{\o}$ if and only if the Hopf axiom is valid in $\GR^{\o,\l}(\ol{f})$ for every $\l$ and $\ol{f}$. 
\end{corollary}

Our second theorem is a \lq base change\rq~result. Lift $\ol{f}$ to a monic polynomial $f\in \o[t]$. Let $x\in \Mat_d(\o)$ be a root of $f$, and define $\O=\o[x]$. Then $\O$ is a discrete valuation ring which is an unramified extension of $\o$.
\begin{thmABC}[Base change]  \label{thm:base.change.iso}
We have an (explicit) isomorphism of rings and corings $$\GR^{\o,\l}(\ol{f})\cong \GR^{\O,\l}(t-\ol{x}).$$
\end{thmABC}

Since our operations are compatible with translation by one dimensional characters (see the discussion at the beginning of Section \ref{sec:base.change}), we also have an isomorphism $\GR^{\O,\l}(t-\ol{x})\cong \GR^{\O,\l}(t)$ of algebras and coalgebras. Theorems  \ref{thm:primary} and \ref{thm:base.change.iso} thus have the following corollary.
\begin{corollary}
The Hopf axiom is valid in $\GR^{\o}$ if and only if it is valid in $\GR^{\O,\l}(t)$ for every unramified extension~$\O$ of $\o$. 
The algebra and colagebra $\GR^{\O,\l}(t)$ is spanned by irreducible representations whose associated matrix is nilpotent.
\end{corollary}

Theorem~\ref{thm:primary} and Theorem~\ref{thm:base.change.iso}, in conjunction with the complete analysis of the principal series representations of~$\GL_n(\o)$ given in~\cite{CMO2}, give rise to a large PSH-subalgebra inside $\GR^{\o}$ that contains Zelevinsky's PSH algebra~$\GR^{\o,1}$.  To describe this PSH-algebra we require a couple of definitions.

\begin{definition}
For a smooth irreducible representation $\rho$ of  $\GL_\nn(\o)$ we let $\GS(\rho)$ denote the $\Z$-submodule of $\GR^\o$ spanned by all the irreducible representations that are contained in $\rho^{i}=\underbrace{\rho \nmult \rho \nmult \cdots \nmult \rho}_{i~\text{times}}$ for all $i \ge 0$.  
\end{definition}
Following Carayol \cite[Section 4]{Carayol} we give the following definition. 
\begin{definition} A continuous irreducible representation of  $\GL_\nn(\o)$ of level $\l$ is called {\em strongly cuspidal} 
%(tr\`{e}s cuspidale in \cite{Carayol}) 
if either 
\begin{enumerate}
\item $\l=0$ and $\rho$ is inflated from a cuspidal representation of $\GL_n(\kk)$; or 
\item $\l \ge 1$ and the restriction of $\rho$ to $K^{\l}/K^{\l+1} \cong \Mat_n(\kk)$ consists of \uri{characters represented by irreducible matrices}.    
\end{enumerate}
\end{definition}

\begin{thmABC}[Local PSH]\label{thm:atom.psh.subalgebra} Let $\rho$ be a strongly cuspidal representation of $\GL_n(\o)$ for some $n \in \N$. Then up to a grading shift $\GS(\rho)$ is isomorphic to the unique universal PSH-algebra. {In particular, it satisfies the Hopf axiom.} 
\end{thmABC}

Cuspidal representations of $\GL_n(\kk)$ are parameterised by Galois orbits of {\em norm-primitive} characters of the multiplicative group of a degree $n$-extension $\kk_n$ of $\kk$. That is, $\Gal(\kk_n/\kk)$-orbits of characters $\theta: \kk_n^\times \to \C^\times$ which cannot be written as $\theta' \circ N_{\kk_n^\times \mid \kk_m^\times}$ where $N_{\kk_n^\times \mid \kk_m^\times}$ is the norm map and $\theta'$ a character of~$\kk_m^\times$.    
An analogous description of the strongly cuspidal representations of $\GL_\nn(\o)$ of level $\l > 0$ is given in \cite[Theorem B]{AOPS}. It is shown that there exists a bijection between strongly cuspidal representations of $\GL_\nn(\o_\l)$ and Galois orbits of strongly primitive characters of $(\O_\l^{\mathrm{ur}, n})^\times$, with $\O^{\mathrm{ur}, n}$ an unramified extension of $\o$ of degree $\nn$.  A character $\theta: (\O_\l^{\mathrm{ur}, n})^\times \to \C^\times$ is called {\em strongly primitive} if its restriction to the smallest congruence subgroup $1+\pi^{\l-1}\O^{\mathrm{ur}, n}_\l \cong (\kk_n,+)$ does not factor through the trace map $\tr_{\kk_n \mid \kk_m}$. Let $\Pi(\O^{\mathrm{un},\times})$ denote the disjoint union of the strongly primitive characters of  $(\O_\l^{\mathrm{ur}, n})^\times$ going over all $n,\l \in \N$.   

\begin{thmABC}[Global PSH]\label{thm:large.psh.subalgebra} Let $\GS^\o$ be the product of all the subrings $\GS(\rho)$ inside $\GR^{\o}$, where $\rho$ ranges over all strongly cuspidal representations of $\GL_n(\o_{\l})$ for all $n,\l \in \N$.
Then the multiplication and comultiplication on $\GR^{\o}$ induce a PSH-algebra structure on $\GS^{\o}$. 
\end{thmABC}

\begin{remark} In fact, $\GR^\o$ contains a much larger PSH-algebra. As mentioned in and after Theorem \ref{thm:base.change.iso}, if $f$ is an irreducible polynomial, then we have an isomorphism of rings and corings $$\GR^{\o,\l}(\ol{f})\cong \GR^{\O,\l}(t-\ol{x})\cong \GR^{\O,\l}(t),$$ where $\O = \o[x]$. The ring and coring $\GR^{\O,\l}(t)$ contains $\GR^{\O,\l-1}$ in a natural way, by considering representations of level $\l-1$ as representations of level $\l$. 
Using the isomorphism above iteratively, we see that for every chain of unramified ring extensions $\o\subseteq \O^{(1)}\subseteq\cdots\subseteq \O^{(\l)}$ we get a chain of inclusion\uri{s} of rings and corings $$\GR^{\O^{(\l)},0}\subseteq \GR^{\O^{\uri{(\l-1)}},1}\subseteq\cdots\subseteq \GR^{\o,\l}.$$ We define a \textit{strongly semisimple} irreducible representation to be one that is contained in the \uri{union of all such chains}. The subgroup generated by the strongly semisimple representations is then a subring and a subcoring of $\GR^{\o,\l}$ which is also a PSH algebra. This can be seen by similar methods to those we use in the proof of Theorem~\ref{thm:large.psh.subalgebra}. The algebra of strongly semisimple representations properly contains $\GS^\o$.
\end{remark}

\subsection{Related and future work} 

In a manuscript by Howe and Moy \cite{Howe-Moy} and a paper by Hill \cite{Hill_Jord}, the authors prove isomorphisms between certain Hecke algebras closely related to this paper. The motivation of the former is to study representations of general linear groups over local fields, while that of the latter is to study a Jordan decomposition for representations of general linear groups over their ring of integers. In the language of the present paper Theorem 2.1 in \cite[Chapter 1]{Howe-Moy} translates to the conjunction of the categorical equivalences given in Theorem \ref{thm:primary.decomposition.of.reps} and in~\S\ref{subsec:even} and~\S\ref{subsec:odd} below. The approach taken in the present paper is not based on the study of Hecke algebras, and is different in tools and scope. Notably, the thrust of this paper is to establish the compatibility between the above equivalences and the operations $\circ$ and $\Delta$. 
 
 \smallskip
 
There are two fundamental open problems emerging from this paper. The first is to clarify whether $\GR$ is a PSH-algebra. This is directly related to the existence of a Mackey formula of the type seen for reductive groups over finite fields \cite[Theorem 5.1]{DM} or for reductive groups over local fields \cite[Theorem 5.2]{BZ_reductive}. While there are several positive indications in this direction, for instance the principal series for general linear groups (see \S\ref{sec:principal.series}) or $\Sp_4(\o_2)$ and its Siegel Levi subgroup (see \cite[\S5]{CMO2}), the general case is wide open.   The second problem is to describe the cuspidal representations with respect to the induction and restriction functors of the present paper.

\subsection{Organisation of the paper} The paper is organised as follows. In \S\ref{sec:preliminaries} we give the necessary preliminaries. In particular we review the construction of the induction functors in \S\ref{subsec:VHC} and in \S\ref{subsec:clifford} we describe the main tools developed in \cite{CMO1} for working with these functors. In \S\ref{fieldext} we gather results from linear algebra over fields and local rings that are constantly used throughout. In lack of appropriate reference we also give proofs.    

In \S\ref{sec:GR} we describe the ring and coring structure of $\GR^{\o}$.  In \S\ref{sec:primary.decomposition} we prove Theorem~\ref{thm:primary} and in \S\ref{sec:base.change} we prove Theorem~\ref{thm:base.change.iso} where the results from \S\ref{sec:preliminaries} and   \S\ref{sec:GR} are heavily used. Section \S\ref{sec:base.change} has two parts: \S\ref{subsec:even} deals with even level and \S\ref{subsec:odd}  deals with odd level and is far more involved. In particular, we establish a coherence of trivialisations of certain two-cocycles across all linear groups simultaneously.   

In \S\ref{sec:principal.series} we explain the analysis of the principal series representations as developed in \cite{CMO2} and prove Theorems~\ref{thm:atom.psh.subalgebra} and~\ref{thm:large.psh.subalgebra}. Finally, in \S\ref{sec:nilpotent} we study primary subalgebra of $\GR^{\o,2}$ associated to certain nilpotent matrices and show that it is a PSH-subalgebra as well. 

\subsection{Notation}

Throughout this paper $\o$ denotes a complete discrete valuation ring with finite residue field $\kk$ and $\pi$ is a uniformiser. We let $\O$ denote  an unramified extension, and let $\o_\l$ and $\O_\l$ denote finite quotients modulo~$\pi^\l$. The reduction modulo $\pi$ is denoted $x \mapsto \ol{x}$. The letters $m,n,i,j,\l$ represent natural numbers and $a,b,c,x,y,z$ are elements in algebras or modules over local rings. The ring of matrices is denoted $\Mat_n$. The stabiliser in a group $G$ of an element $x$ under a group action is denoted $G(x)$. \uri{For a finite group $G$ we let $\Rep(G)$ denote the category of its finite dimensional complex representations.} The following list collects further key notation, together with the number of the section where it is introduced.

\smallskip

\begin{center}
\begin{tabular}{r||l|l}
  $\cA, \cZ$ & algebra and its complement \uri{with respect to the trace form} & \S\ref{fieldext} \\ 
 $K^i=\gZ^i \gA^i$ & congruence kernels $\gZ^i=1+\cZ^i, \gA^i=1+\cA^i$, $\cA^i=\pi^i\cA$,  $\cZ^i=\pi^i\cZ$& \S\ref{sec:base.change} \\ 
 $\GR_n^{\o,\l}$, $\GR_n^{\o}$ & Grothendieck rings of $\GL_n(\o_\l)$ and $\GL_n(\o)$& \S\ref{sec:GR}\\
$\GR^\o$ & the sum of Grothendieck rings $\oplus_{n \ge 0} \GR_n^{\o}$ & \S\ref{sec:GR}\\
$\GR^{\o,\l}_{[\ol{\MM}]}$ & the submodule of $\GR^{\o,\l}$ with associated matrix $\ol{\MM}$ &\S\ref{sec:GR-gradings}\\
$\GR^{\o,\l}(\ol{f})$ & the subring associated to an irreducible polynomial $\ol{f}$ & \S\ref{sec:GR-gradings} \\
$\nmult$, $\Delta$   & multiplication and comultiplication & \S\ref{subsec:main.results} \\
$\phi: \mathfrak{f} \to \C^\times$ & a fixed additive character of $\mathfrak{f}=\mathrm{frac}(\o)$ & \S\ref{subsec:Pontryagin.Dual}\\
      $\xi \mapsto \phi_\xi$  & the duality $\Mat_n(\o_\l) \to \Mat_n(\o_\l)^\vee$ & \S\ref{subsec:Pontryagin.Dual}\\
      $[\Mat_n]=[\Mat_n^{\o,i}]$  & the set of similarity classes in $\Mat_n(\o_i)$ & \S\ref{sec:GR-gradings} \\
      $[\Mat]$ & the disjoint union $\bigsqcup_{n\geq 0} [\Mat_n]$ & \S\ref{sec:GR-gradings}  \\
      $\charpoly:[\Mat]\to \o_1[t]$ & the map sending a matrix to its characteristic polynomial &\S\ref{sec:GR-gradings}  \\
  \end{tabular}
\end{center}

%%%%%%%%%%%%%%%%%%%%%%%

\section{Preliminaries}\label{sec:preliminaries} 

\subsection{A variant of Harish-Chandra functors.}\label{subsec:VHC}  
%In \cite{CMO1}, the study of the following variant of Harish-Chandra functors was initiated. 
The following is a summary of results from~\cite{CMO1} that are essential for the present paper. The focus in {\em loc. cit.} is on profinite groups but here we restrict attention to finite groups.      
\begin{definition} Let $G$ be a finite group, and suppose that $U$, $L$ and $V$ are subgroups of $G$ such that $L$ normalises~$U$ and $V$. We say that $(U,L,V)$ is a {\em virtual Iwahori decomposition} of $G$ if the map
\[
U\times L\times V \hookrightarrow G
\]
given by multiplication in $G$ is injective. If that map is a bijection we call the triple $(U,L,V)$ an {\em Iwahori decomposition}. 
\end{definition}

Let $e_U$ and $e_V$ denote the idempotents in the complex group algebra~$\C G$ associated to the trivial representations of $U$ and $V$, and consider the $\C G$-$\C L$ bimodule~$\C G e_U e_V$. Let $\pind_{U,V}$ be the functor from the category of complex representations of $L$ to the category of complex representations of $G$ defined by tensoring with this bimodule:
\[
\pind_{U,V}:\Rep(L) \to \Rep(G), \qquad X \mapsto \C Ge_U e_V \otimes_{\C L}X.
\]
Similarly, define 
\[
\pres_{U,V}:\Rep(G) \to \Rep(L), \qquad Y \mapsto e_Ue_V \C G \otimes_{\C G}Y.
\]

In \cite{CMO1} we prove, among other things, that whenever $(U,L,V)$ is a virtual Iwahori decomposition of $G$ the functors $\pind_{U,V}$ and $\pres_{U,V}$ are biadjoints and that $\pind_{U,V} \cong \pind_{V,U}$ and similarly $\pres_{U,V} \cong \pres_{V,U}$. The latter isomorphisms follow from the $\C L$-$\C G$ bimodule isomorphism 
\begin{equation}\label{commutativity}
e_Ue_V\C G \to e_Ve_U\C G, \qquad f \mapsto e_V f,
\end{equation} 
see \cite[Theorem 2.15]{CMO1}. Essential to the present paper are the compatibility with Clifford theory proved in \cite[Section 3]{CMO1} and summarised in \ref{subsec:clifford} below.

\subsection{Compatibility with Clifford theory and natural operations}\label{subsec:clifford} 
Let $(U,L,V)$ be a virtual Iwahori decomposition of a finite group $G$ and let $G_0 \lhd G$ be a normal subgroup such that $G_0=U_0 L_0V_0$, where we write $H_0=H \cap G_0$ for every $H \le G$. 

Let $\psi$ be an irreducible representation of $L_0$, and let $\uri{{\eta}}=\pind_{U_0,V_0}(\psi)$ be the corresponding (irreducible) induced representation of $G_0$. Let $L(\psi)$ and $G(\uri{{\eta}})$ denote the inertia groups of $\psi$ and $\uri{{\eta}}$. Then the following diagram for induction is commutative (and there is a similar diagram for restriction):
\[
 \xymatrix@C=80pt{ \Rep(L)_\psi \ar[r]^-{\pind_{U,V}}  & \Rep(G)_{\uri{\eta}} \\
 \Rep (L(\psi))_\psi \ar[r]^-{\pind_{U(\uri{\eta}),V({\uri{\eta}})}} \ar[u]_{\ind}^{\cong}   & \Rep(G({\uri{\eta}}))_{\uri{{\eta}}}  \ar[u]_{\ind}^{\cong} \\
 \Rep^{\gamma } (L(\psi)/L_0) \ar[r]^-{\pind_{U({\uri{\eta}})/U_0, V({\uri{\eta}})/V_0}} \ar[u]^-{\cong}   & \Rep^{\gamma }(G(\uri{\eta})/G_0)  \ar[u]^-{\cong} 
 }
\] 
where $\Rep(H)_\theta$ stands for the representations of $H$ whose restriction to $H_0\lhd H$ consists entirely of $H$-conjugates of the irreducible representation $\theta$; and $\Rep^{\gamma}(H(\theta)/H_0)$ stands for \uri{the category of projective representations of $H(\theta)/H$ with the cocycle $\gamma \in Z^2(H(\theta)/H_0, \C^\times)$ associated to $\theta$; see \cite[Chapter 11]{Isaacs}.}       

We recall here Theorem 2.18 (5) from \cite{CMO1}. It will be used in the course of the proof of Theorem \ref{thm:base.change.iso} in \S\ref{sec:base.change}. 
\begin{theorem}\label{thm:inflation.coherence}
Let $G$ be a finite group with a virtual Iwahori decomposition $(U,L,V)$.
Let $\gZ$ be a normal subgroup of $G$ such that $(U_\gZ,L_\gZ,V_\gZ):=(U\cap \gZ, L\cap \gZ, V\cap \gZ)$ is an Iwahori decomposition of $\gZ$. 
The diagrams
\[
\xymatrix@C=60pt{
\Rep(L) \ar[r]^-{\pind_{U,V}} & \Rep(G) \\
\Rep(L/L_\gZ) \ar[u]^-{\infl} \ar[r]^-{\pind_{U/U_\gZ,V/V_\gZ}} & \Rep(G/ \gZ) \ar[u]_-{\infl} 
}
\qquad\text{and}\qquad
 \xymatrix@C=60pt{
\Rep(G) \ar[r]^-{\pres_{U,V}} & \Rep(L) \\
 \Rep(G/\gZ) \ar[u]^-{\infl} \ar[r]^-{\pres_{U/U_\gZ,V/V_\gZ}} & \Rep(L/L_\gZ)\ar[u]_-{\infl} 
}
\]
commute up to natural isomorphism. (Here $\infl$ denotes inflation.)
\end{theorem}

\subsection{Setup for general linear groups}The virtual Iwahori decompositions occurring in this paper are the following coherent family of subgroups of $\GL_n$ for all $\nn$. Whenever $(\nn_1,\ldots,\nn_r)$ is a composition of $\nn$ we take $L=\GL_{n_1} \times \cdots \times \GL_{\nn_r}$ as the subgroup of invertible block-diagonal matrices in a natural way. The subgroup $U_{\nn_1,\ldots,\nn_r}$ is the upper uni-triangular subgroup with block sizes $\nn_i \times \nn_j$ and $V_{\nn_1,\ldots,\nn_r}$ is its transpose. Whenever the composition is clear from the context we omit it from the notation, writing simply $U$ and $V$.

\subsection{Pontryagin duals}\label{subsec:Pontryagin.Dual}We fix an additive character ${\phi}:\mathfrak{f}=\mathrm{frac}(\o) \to \C^\times$ with ${\phi}_{\mid \o}\equiv1$ and ${\phi}_{\mid \pi^{-1}\o} \notequiv 1$, as follows. We first define $\phi':\pi^{-1}\o/\o\to \C^{\times}$ as the composition $\pi^{-1}\o/\o\cong \o/\pi\o\cong \kk\stackrel{\text{Tr}}{\to}\F_p\to \C^{\times}$ where $\text{Tr}$ is the trace of the Galois extension $\kk/\F_p$ and $\F_p\to \C^{\times}$ is given by choosing a primitive $p$-th root of unity. We then extend $\phi'$ to a homomorphism $\phi:\mathfrak{f}\to \C^{\times}$ by using the fact that $\C^{\times}$ is an injective abelian group.

For every $\l, n \in \N$ we define  a $\GL_n(\o_\l)$-equivariant  isomorphism $\xi \mapsto \phi_\xi$ from  $\Mat_n(\o_\l)$ to its Pontryagin dual  $\Mat_n(\o_\l)^\vee=\Hom_\Z\left(\Mat_n(\o_\l), \C^\times\right)$, defined as follows: fix a set-theoretical \uri{section} $\sigma:\o_\l\to\o$ and define 
\[
\phi_\xi(a)=\phi(\pi^{-\l}\tr(\sigma(\xi a))), \quad a \in \Mat_n(\o_\ell).
\]
Note that since $\phi$ vanishes on $\o$ this is well defined and does not depend on the particular set-theoretical \uri{section}~$\sigma$. The fact that this map is an isomorphism is a consequence of the fact that $\o$ is a local ring. The $\GL_n(\o_\l)$-equivariance is immediate.
 
In our applications the level $\l$ and the dimension $n$ will be clear from the context and they are omitted from the definition of the isomorphism. Whenever $i \ge \l$ we have an isomorphism $\Mat_n(\o_\l) \xrightarrow{\cong} K^i/K^{i+\l}$  given by $a \mapsto 1+\pi^i a$, which allows us to identify $\Mat_n(\o_\l)$ with the Pontryagin dual of $K^i/K^{i+\l}$ as well by $\xi \mapsto \phi^\star_\xi$, where
\begin{equation}\label{eq:phi.star}
\phi^\star_\xi(1+\pi^i a)=\phi_\xi(a).
\end{equation}

\subsection{Some linear algebra}\label{fieldext}
\subsubsection{Linear algebra over fields}
Let $F$ be a field and  let $f(t)\in F[t]$ be an irreducible polynomial of degree $d$. Let $\xx \in \Mat_{d}(F)$ be a $d \times d$ matrix with characteristic polynomial $f(t)$; this condition characterizes $\xx$ up to conjugation. Let $E=F[\xx]$ be the $F$-subalgebra of $\Mat_{d}(F)$ generated by $\xx$. Since $E\simeq F[t]/(f(t))$, it is a field extension of $F$ of degree  $d$ inside $\Mat_d(F)$. The trace induces a non-degenerate symmetric bilinear form on $\Mat_d(F) \times \Mat_d(F)$, namely, $(\aa, \bb) \mapsto \tr (\aa\bb)$. Let 
\[
E^\perp=\{\bb \in \Mat_{d}(F) \mid \forall \aa\in E, \, \tr(\aa \bb)=0\, \}.
\]
\begin{lemma}\label{directsum} Let $F$, $E$ and $\xx$ be as above. Then $\Mat_{d}(F) = E^\perp \oplus E$, and the map $\ad_\xx:\Mat_{d}(F)\rightarrow \Mat_{d}(F)$ given by $\ad_\xx(\yy) = \xx\yy-\yy\xx$ satisfies $\Ker(\ad_\xx)=E$ and $\Ima(\ad_\xx) = E^\perp$. In particular, the restriction of $\ad_\xx$ to $E^\perp$ induces an isomorphism $\ad_\xx:E^\perp\rightarrow E^\perp$.
\end{lemma}
\begin{proof}
Since $E/F$ is a finite field extension, we know that the trace form is non-degenerate on $E$ and therefore $E\cap E^\perp=0$. 
As $\dim_{F}\!E + \dim_{F}\! E^\perp= \dim_{F}\! \Mat_{d}(F)$ we have a direct sum decomposition.
A matrix is in the kernel of $\ad_\xx$ if and only if it commutes with $\xx$. 
If a matrix $\yy$ commutes with $\xx$, we can view $\yy$ as a linear transformation of $F^{d}$ viewed as a vector space over $E$.
But since $\xx$ is irreducible, the only endomorphisms of this space are those in $E$, so $\Ker(\ad_\xx)=E$.
On the other hand, since $\tr(\aa \bb)=\tr(\bb\aa)$, $E^\perp$ is an $E$-bimodue, therefore, $\Ima(\ad_\xx)\subseteq E^\perp$. 
Since their dimensions agree, we have an equality.
\end{proof}

We consider next matrices whose characteristic polynomial is $f^m$ for some $m$.
Our goal is to show that for such matrices a weaker version of Lemma~\ref{directsum} holds (where $m=1$). 
Let $\MM$ be an $n\times n$ matrix with characteristic polynomial $f^m$ (where $n=md$). Let $\cA=\Mat_m(E)$ and $\cZ=\Mat_m(E^\perp)$.
By passing to a conjugate matrix if necessary, we may assume that $\MM$ lies in~$\cA$,
and has the form $\MM=\sigma+\nu$, where $\s$ is the semisimple part of~$\MM$, specifically the block diagonal matrix with copies of $\xx$ on the diagonal, 
$\n \in \cA$ is the nilpotent part of $\MM$, and $\s$ and $\n$ commute.
\begin{lemma}\label{kernelQ}
Let $m, f, \cA=\Mat_m(E)$ and $\cZ=\Mat_m(E^\perp)$ be as above. Assume that $\MM \in A$ has characteristic polynomial~$f^m$ and rational semisimple-nilpotent decomposition $\MM=\s+\n$. Then the centralizer of $\MM$ in $\Mat_n(F)$ is contained in the subalgebra $\cA$, the image of the map $\ad_\MM$ contains the subspace $\cZ$, and there is a direct sum decomposition $\Mat_n(F)=\cA \oplus \cZ$.
\end{lemma}
\begin{proof}
We have $\ad_\MM = \ad_{\s} + \ad_{\n}$. The operators $\ad_{\s}$ and $\ad_{\n}$ commute, $\ad_{\s}$ is semisimple and $\ad_{\n}$ is nilpotent.
Therefore, $\ad_\MM = \ad_{\s} + \ad_{\n}$ is the semisimple-nilpotent decomposition of $\ad_\MM$, considered as a linear operator on $\Mat_n(F)$.
Using the geometric expansion, and the fact that $\ad_{\n}$ is nilpotent of nilpotancy degree at most~$2m$, we see that the kernel of $\ad_\MM$ is contained in the kernel of $\ad_{\s}^{2m} - \ad_{\n}^{2m} = \ad_{\s}^{2m}$.
But the operator $\ad_{\s}$ is semisimple. So the kernel of $\ad_\MM$ (which is the centraliser of $\MM$) is contained in the kernel 
of $\ad_{\s}^{2m}$ which is equal to the kernel of $\ad_{\s}$ (by semisimplicity) which is $\cA$.

In a similar fashion, the \uri{i}mage of $\ad_\MM$ contains the image of $\ad_{\s}^{2m} - \ad_{\n}^{2m} = \ad_{\s}^{2m}$.
By semisimplicity of $\ad_{\s}$ this is the same as the image of $\ad_{\s}$, which is $\cZ$. 

Finally, it is clear that $\cZ \subset \cA^\perp$ and equality follows from dimension consideration.
\end{proof}

\begin{definition}\label{def:coprime} We say that two matrices $\MM_1 \in \Mat_{\nn_1}(F)$ and $\MM_2 \in \Mat_{\nn_2}(F)$ are \textit{coprime} if their characteristic polynomials are coprime. 
\end{definition}

For $\MM_1\in \Mat_{\nn_1}(F)$ and $\MM_2 \in \Mat_{\nn_2}(F)$ we write $\MM_1\oplus \MM_2$ for the matrix 
\[
\left(\begin{matrix} \MM_1 & 0 \\ 0& \MM_2 \end{matrix}\right) \in \Mat_{\nn_1+\nn_2}(F). 
\]
\begin{lemma}\label{lem:coprime} Let  $\MM_1\in \Mat_{\nn_1}(F)$ and $\MM_2\in \Mat_{\nn_2}(F)$ be coprime. Then the only solution $\yy \in \Mat_{\nn_1 \times \nn_2}(F)$ to the equation $\MM_1 \yy=\yy \MM_2$ is $\yy=0$. In particular, the centraliser in $\Mat_{\nn_1+\nn_2}(F)$ of $\MM_1 \oplus \MM_2$ consists of block diagonal matrices with blocks sizes $\nn_1 \times \nn_1$ and $\nn_2 \times \nn_2$.  
\end{lemma}

\begin{proof}
Let $p_i$ be the characteristic polynomial of $\MM_i$ and assume that $\MM_1\yy=\yy \MM_2$. Using the latter relation repeatedly and taking linear combinations we deduce that for every polynomial $g \in F[t]$ we have $g(\MM_1)\yy=\yy g(\MM_2)$. In particular, for $g=p_1$ we get that $0=p_1(\MM_1)\yy=\yy p_1(\MM_2)$. Thus the claim would follow if we show that $p_1(\MM_2)$ is invertible. But this follows by our assumption that the $p_1$ and $p_2$ are coprime. Similarly, the only solution $\zz \in \Mat_{\nn_2 \times \nn_1}(F)$ to  $\zz \MM_1=\MM_2 \zz$ is $\zz=0$. It follows that  
\[
\left(\begin{matrix} \yy_1 & \yy_3 \\ \yy_4 & \yy_2 \end{matrix}\right)\left(\begin{matrix} \MM_1 & 0 \\ 0& \MM_2 \end{matrix}\right) =\left(\begin{matrix} \MM_1 & 0 \\ 0& \MM_2 \end{matrix}\right)\left(\begin{matrix} \yy_1 & \yy_3 \\ \yy_4 & \yy_2 \end{matrix}\right) 
\]
 implies that $\yy_1 \oplus \yy_2$ centralizes $\MM_1\oplus \MM_2$ and that $\yy_3$ and $\yy_4$ are zero.
\end{proof}

\begin{lemma}\label{lem:duality.between.U.and.V.over.F} Let $\MM_1 \in \Mat_{\nn_1}(F)$ and $\MM_2 \in \Mat_{\nn_2}(F)$ be coprime and set $\MM=\MM_1 \oplus \MM_2$. Then the bilinear map
\begin{equation}\label{eq:perfect.pairing.U.and.V.over.F}
\Mat_{\nn_1 \times \nn_2}(F) \times \Mat_{\nn_2 \times \nn_1}(F) \to F, \quad (\aa,\bb) \mapsto \tr(\MM_1\aa\bb)-\tr(\bb\aa\MM_2) 
\end{equation}
is a perfect pairing. 
\end{lemma}

\begin{proof} We have that
\[
\tr(\MM_1\aa\bb)-\tr(\bb\aa\MM_2)=\tr\left(\aa(\bb \MM_1-\MM_2 \bb)\right).
\]
Consequently, if the form vanishes for every $\aa \in \Mat_{\nn_1 \times \nn_2}(F)$ we must have $\bb \MM_1-\MM_2 \bb=0$ by the non-degeneracy of the trace form. Lemma~\ref{lem:coprime} now implies that $\bb=0$, and the assertion follows.   
\end{proof}

\subsubsection{Linear algebra over discrete valuation rings}\label{subsec:lin.alg.rings} Let $\o$ be a complete discrete valuation ring with finite residue field. Let $\pi$ be a uniformiser in $\o$ and let $\o_\l = \o/(\pi^\l)$.  Let $f(t) \in \o[t]$ be a monic polynomial of degree~$d$ whose \uri{reduction} modulo $\pi$, which we denote by $\ol{f}$, is a monic irreducible polynomial over the finite field $\o_1$. 
%It is possible to lift $f$ to an irreducible monic polynomial $\ti{f}$ of degree $d$ over $\Ow$ 
%(just lift all the coefficients while guaranteeing that the leading coefficient is still 1). 
Let $x\in \Mat_d(\o)$ be a root of the polynomial $f$. 
Then $\O = \o[x]$ is an unramified extension of $\o$ with uniformiser $\pi$ and we set $\O_\l=\O/(\pi^\l)$ for $\l \in \N$.
We write $\ol{x}$ for the image of $x$ in $\O_1$, so that $\O_1 = \o_1(\ol{x})$. We may view $\Mat_m(\O_\l)$ as a subalgebra of $\Mat_{md}(\o_\l)$ for $\l \in \N \cup {\infty}$ in a natural way. 
 The following proposition deals with lifting matrices in $\Mat_n(\o_1)$ with characteristic polynomial $\ol{f}^m$ to matrices over $\Mat_n(\o_\l)$.
This proposition will be essential in the proof of Theorem~\ref{thm:base.change}.
\begin{proposition}\label{intoT}
Let $\MM \in \Mat_n(\o_\l)$ be a matrix where $n=md$. Assume that $\ol{\MM}$, the reduction of $\MM$ modulo $\pi$, is contained in $\Mat_m(\O_1)\subseteq \Mat_n(\o_1)$ and has characteristic polynomial $\uri{\ol{f}}^m$.
Then $\MM$ is conjugate to a matrix in $\Mat_m(\O_\l)$ whose centraliser is also contained in $\Mat_m(\O_\l)$. 
Moreover, if it is conjugate to two such matrices $\bb$ and $\cc$, then $\bb$ and $\cc$ are also conjugate under the action of the smaller group $\GL_m(\O_\l)$.
\end{proposition}
\begin{proof}
We consider the $\O_\l$-subalgebra $\cA=\Mat_m(\O_\l) \subseteq \Mat_n(\o_\l)$ and the $\O_\l$-submodule $\cZ=\cA^{\perp}$,
the submodule perpendicular to $\cA$ with respect to the trace pairing into $\O_\l$. Notice that in case $\l=1$ the submodules $\cA$ and $\cZ$ coincide with the vector spaces $\cA$ and $\cZ$ from the previous subsection. Since in the present setup $\cA$ is a free $\O_\l$-module we have a direct sum decomposition $\Mat_n(\o_\l) = \cA\oplus \cZ$ following the same argument as in Lemma~\ref{directsum} for the case $\l=1$. Writing $\cA^i = \cA\cap \pi^i \Mat_n(\o_\l)$ and $\cZ^i = \cZ\cap \pi^i\Mat_n(\o_\l)$, we also have $\pi^i \Mat_n(\o_\l) = \cA^i \oplus \cZ^i$.

We shall prove the first assertion of the lemma by induction. 
For $i=1,\ldots, \l$ we prove that the matrix $\MM_i:=\MM \modd \pi^i$ is conjugate to a matrix in $\cA/\cA^i$.
For $i=1$ we know that this is true by the assumption of the lemma.
We assume by induction that the assertion holds for~$\MM_{i-1}$.
By conjugating $\MM_i$, we can assume that $\MM_i \in \cA/\cA^i \oplus \cZ^{i-1}/\cZ^i$.
For a matrix $\yy \in \Mat_n(\o_\l)$ we have
$$(1+\pi^{i-1}\yy)(\MM_i)(1+\pi^{i-1}\yy)^{-1} = \MM_i + \pi^{i-1}(\yy\MM_i-\MM_i\yy)\modd\pi^{i}.$$
Notice that the expression $\yy \MM_i-\MM_i \yy\modd\pi$ depends only on $\ol{\MM}$ and $\ol{\yy}$.
It follows from Lemma \ref{kernelQ} that the image of the $\ad_{\ol{\MM}}$ contains $\cZ^{i-1}/\cZ^i$.
This implies that by conjugating $\MM_i$ again we can assume that it is contained in~$\Mat_m(\cA/\cA^i)$, as desired.

We thus assume that $\MM \in \cA$.
We next prove that its centraliser is contained in $\Mat_m(\O_\l)$.
So let $\yy \in \Mat_\nn(\o_\l)$ be in the centraliser $\Mat_\mm(\o_\l)(\MM)$. We write $\yy=\aa + \zz$ where $\aa \in \cA$ and $\zz \in \cZ$.
Since we assume that $\MM \in \cA$ the direct sum decomposition $\Mat_\nn(\o_\l)=\cA \oplus \cZ$ is stable under $\ad_\MM$ and it follows that $\aa$ and $\zz$ also commute with~$\MM$.
We therefore need to prove that if $\zz \in \cZ$ commutes with $\MM$ then $\zz=0$.
By reducing the equation~$\zz \MM = \MM \zz$ modulo $\pi$ we get that $\ol{\zz} \in \ol{\cA}  \cap \ol{\cZ} = 0$ and so $\ol{\zz}=0$, again from Lemma~\ref{kernelQ}.
So $\zz=\pi \zz'$ for some matrix~$\zz' \in \cZ$. But then $\zz' \MM = \MM \zz'$ in $\Mat_\nn(\o_{\l-1})$,
so we conclude that $\zz' = 0 \modd \pi$. Continuing by induction, we get that~$\zz=0$. 

Finally, we prove the uniqueness statement. 
Assume therefore that $\bb, \cc\in \Mat_m(\O_\l)$ are conjugate to one another by an element $g\in \GL_n(\o_\l)$ and that 
their characteristic polynomial mod $\pi$ is $\ol{f}^m$.
By considering the rational canonical form of their reductions $\ol{b}$ and $\ol{\cc}$ we conclude that $\ol{\bb}$ and $\ol{\cc}$ 
are conjugate under the action of $\GL_m(\O_1)$. We can thus assume without loss of generality that $\ol{\bb}=\ol{\cc}$. 

We need to prove that $\bb$ and $\cc$ are also conjugate inside $\GL_m(\O_\l)$.
If we reduce the equation $g \bb g^{-1} = \cc$ modulo $\pi$, 
we get that $\ol{g}$ centralizes $\ol{\bb}$, and is therefore contained in $\GL_\nn(\o_1)(\ol{\bb})\subseteq \GL_\mm(\O_1)$.
We can thus write $g=\ti{g}(1+\pi y)$ where $\ti{g}$ is a lift of $\ol{g}$ in $\GL_\mm(\O_\l)$ 
 and $\yy \in \Mat_\nn(\o_\l)$.
We therefore need to prove that if~$\bb$ and~$\cc$ in~$\Mat_\mm(\O_\l)$ are conjugate by an element of the form $1+\pi \yy$,
then they are conjugate by an element in $\GL_\mm(\O_\l)$.
For this, we write $\yy=\aa+\zz$ where $\aa \in \cA$ and $\zz \in \cZ$.
We then get the equation
$$(1+\pi \aa + \pi \zz) \bb = \cc (1 + \pi \aa + \pi \zz).$$
Since $\cA$ is closed under multiplication, $\cZ$ is an $\cA$-bisubmodule and $\cA\oplus \cZ = \Mat_\nn(\o_\l)$, 
we look at the projection on $\cA$ of the above equation and we get that $(1+\pi \aa) \bb = \cc (1+\pi \aa )$.
The matrices $\bb$ and $\cc$ are thus conjugate by the matrix $(1+\pi \aa ) \in \GL_\mm(\O_\l)$, as desired.
\end{proof}
\begin{remark}
Notice that this implies that if $g\in \GL_n(\o_k)$ satisfies $gBg^{-1}=C$ then $g\in \GL_m(\O_k)$. 
This follows from the fact that the stabiliser is contained in $\GL_m(\O_k)$. 
\end{remark}

%%%%%%%%%%%%%%%%

\section{One ring to rule them all}\label{sec:GR}
%\footnote{TC: think of a better section title (The algebra $\GR$). UO: how about "One ring to rule them all"?}
We fix $\l\geq 1$ and a \uri{complete} discrete valuation ring $\o$ with a uniformiser $\pi$ and a finite residue field $\o_1=\o/\pi$.
\subsection{Basic structure} For each $n\geq 1$ we write $G_n$ for the finite group $\GL_n(\o_\l)$. Following Zelevinsky, who considered the case $\l=1$, we are going to assemble the representations of all of these groups, for varying $n$, into a single algebraic structure. 

For each composition $n=n_1+n_2$ we identify the product $G_{n_1} \times G_{n_2} $ with the subgroup $G_{n_1,n_2} \subset G_{n} $ of block-diagonal matrices:
\[
G_{n_1,n_2}  \coloneq \left\{ \begin{bmatrix} a_{n_1\times n_1} & 0_{n_1\times n_2} \\ 0_{n_2 \times n_1} & d_{n_2\times n_2} \end{bmatrix} \in G_{n}  \right\}.
\]
We also consider the subgroups $U_{n_1,n_2} ,\ V_{n_1,n_2}\subset G_{n}$ of block-triangular matrices
\[
U_{n_1,n_2}  \coloneq \left\{ \begin{bmatrix} 1_{n_1\times n_1} & b_{n_1\times n_2} \\ 0_{n_2\times n_1} & 1_{n_2\times n_2}\end{bmatrix} \in G_{n}   \right\}
\quad \textrm{and} \quad
V_{n_1,n_2}  \coloneq \left\{ \begin{bmatrix} 1_{n_1\times n_1} & 0_{n_1\times n_2} \\ c_{n_2\times n_1} & 1_{n_2\times n_2}\end{bmatrix} \in G_{n}  \right\}.
\]
The triple $(U_{n_1,n_2} , G_{n_1,n_2} , V_{n_1,n_2} )$ is a virtual Iwahori decomposition of $G_n $. We use the abbreviated notation $\pind_{n_1,n_2} \coloneq \pind_{U_{n_1,n_2} , V_{n_1,n_2} }$ and $\pres_{n_1,n_2}\coloneq \pres_{U_{n_1,n_2} , V_{n_1,n_2} }$ for the associated induction and restriction functors $\Rep(G_{n_1,n_2} ) \leftrightarrow \Rep(G_n )$. Similar notation will be used for compositions $n=n_1+ \cdots + n_r$ with more than two parts.

For the sake of notational convenience we also define $G_0 \coloneq\{1\}$, the trivial group;  $G_{n,0}  = G_{0,n}  \coloneq G_n $; and $U_{n,0}  = U_{0,n}  = V_{n,0}  = V_{0,n}  \coloneq \{1\in G_n \}$. Thus the functors $\pind_{0,n}$, $\pind_{n,0}$, $\pres_{0,n}$, and $\pres_{n,0}$ are all the identity functor on $\Rep(G_n )$.

\begin{definition}
For each $n\geq 0$ let $\GR_n $ (or $\GR_n^{\o,\l}$ if $\o$ or $\l$ are not clear from the context) denote the Grothendieck group of $\Rep(G_n^{\o,\l})$. Let $\GR $ (or $\GR^{\o,\l}$) denote the direct sum
$\GR \coloneq \bigoplus_{n\geq 0} \GR_n $.
\end{definition}

Since $\Rep(G_n )$ is semisimple, $\GR $ is a free abelian group with basis $\bigsqcup_{n\geq 0} \Irr(G_n )$, \uri{where $\Irr(G)$ stands for isomorphism classes of irreducible representations of $G$}. This basis will be called the \emph{standard basis} for $\GR $. In the next few paragraphs, leading up to Proposition \ref{prop:GR-basics}, we shall equip $\GR $ with several additional structures.

For each pair of \uri{non-negative} integers $n,m$ the functor
\[
\Rep(G_n )\times \Rep(G_m ) \xrightarrow[\cong]{(\rho,\sigma)\mapsto \rho\boxtimes_{\C} \sigma} \Rep(G_{n,m} ) \xrightarrow{\pind_{n,m}} \Rep(G_{n+m} )
\]
commutes with direct sums, and hence defines a bilinear map $\GR_n  \otimes_{\Z} \GR_m  \to \GR_{n+m} $. These maps assemble into a bilinear map  $\GR \otimes_{\Z} \GR  \to \GR $, which we shall denote by $\rho\otimes \sigma \mapsto \rho\nmult \sigma$.

Similarly, for each $n\geq 0$ and each composition $n=n_1+n_2$ (with $n_i \geq 0$), the restriction functor  $\pres_{n_1,n_2}:\Rep(G_{n} ) \to \Rep(G_{n_1,n_2} )$ induces a linear map $\GR_{n}  \to \GR_{n_1}  \otimes_{\Z} \GR_{n_2} $. Taking the sum of these maps over all $n$ and all compositions $n=n_1+n_2$ gives a linear map $\Delta:\GR  \to \GR   \otimes_{\Z} \GR$.

Let $\langle \argument,\argument \rangle: \GR  \times \GR  \to \Z$ be the bilinear form induced on Grothendieck groups by the pairing
\begin{equation}\label{def.inner.prod}
\langle \rho, \sigma \rangle \coloneq \begin{cases} \dim \Hom_{G_n }(\rho,\sigma) & \textrm{if } \rho,\sigma\in \Rep(G_n ), \\ 0 & \textrm{if }\rho\in \Rep(G_n ),\ \sigma\in \Rep(G_m ),\ n\neq m.
\end{cases}
\end{equation}
Extend this form to a bilinear form on $\GR \otimes_{\Z} \GR $ by setting 
\begin{equation}\label{def.inner.prod.tensor}
\langle \rho_1\otimes \rho_2, \sigma_1\otimes \sigma_2 \rangle \coloneq \langle \rho_1,\sigma_1\rangle \langle \rho_2,\sigma_2\rangle.
\end{equation}

\begin{proposition}\label{prop:GR-basics}
\begin{enumerate}[\rm(1)]
\item The product $\nmult$ makes $\GR $ into an  associative, commutative, unital $\Z$-algebra.
\item The coproduct $\Delta$ makes $\GR $ into a  coassociative, cocommutative, counital $\Z$-coalgebra.
\item For all $\rho,\sigma,\tau\in \GR $ one has $\langle \rho , \sigma \nmult \tau \rangle = \langle \Delta\rho, \sigma\otimes \tau \rangle$.
\end{enumerate}
\end{proposition}

\begin{proof}
The third assertion follows from \cite[Theorem 2.18]{CMO1}. 
Indeed, it is proved there that for every virtual Iwahori decomposition $(U,L,V)$ the functors $\pres_{U,V}$ and $\pind_{U,V}$ are biadjoints. Since $\langle \argument,\argument\rangle$ is defined using Hom-spaces we get $\langle \rho , \sigma \nmult \tau \rangle = \langle \Delta\rho, \sigma\otimes \tau \rangle$. This also implies that the first assertion is equivalent to the second one. 
The commutativity of $\GR$ follows from the first assertion of \cite[Theorem 2.18]{CMO1}. The definition of~$\nmult$ and the assertions about $\pind_{0,n}$ and $\pind_{n,0}$ above makes it clear that the trivial representation of the trivial group $G^{\o,\l}_0$ is the unit element for $\nmult$. To prove associativity, we proceed as follows: let $n=n_1+n_2+n_3$ and let $W_i\in \Rep(G_{n_i})$ for $i=1,2,3$. Writing $\nmult$ explicitly with the idempotents corresponding to the different subgroups gives
\[
\begin{split}
(W_1\nmult &W_2)\nmult W_3 = \Big(\C G_{n_1+n_2}e_{U_{n_1,n_2}}e_{V_{n_1,n_2}}\ot_{G_{n_1,n_2}}(W_1\boxtimes W_2)\Big)\nmult W_3 \\
& \cong \C G_{n_1+n_2+n_3}e_{U_{n_1+n_2,n_3}}e_{V_{n_1+n_2,n_3}}\ot_{G_{n_1+n_2,n_3}}\big(
\C G_{n_1+n_2}e_{U_{n_1,n_2}}e_{V_{n_1,n_2}}\ot_{G_{n_1,n_2}}(W_1\boxtimes W_2)\big)\boxtimes W_3 \\
& \cong \C G_{n_1+n_2+n_3}e_{U_{n_1,n_2}}e_{U_{n_1+n_2,n_3}}e_{V_{n_1,n_2}}e_{V_{n_1+n_2,n_3}}\ot_{G_{n_1,n_2,n_3}}(W_1\boxtimes W_2\boxtimes W_3)
\end{split}
\]
where $G_{n_1,n_2,n_3}\cong G_{n_1}\times G_{n_2}\times G_{n_3}$ is defined in the obvious way. 
We have used here the fact that $G_{n_1+n_2,n_3}$ normalizes $V_{n_1+n_2,n_3}$ \uri{and that $e_{U_{n_1,n_2}}$ commutes with $e_{V_{n_1+n_2,n_3}}$ because $V_{n_1+n_2,n_3}$ is normalized by $U_{n_1,n_2}$}. 
A similar calculation shows that 
\[
W_1\nmult (W_2\nmult W_3)\cong 
\C G_{n_1+n_2+n_3}e_{U_{n_2,n_3}}e_{U_{n_1,n_2+n_3}}e_{V_{n_2,n_3}}e_{V_{n_1,n_2+n_3}}\ot_{G_{n_1,n_2,n_3}}(W_1\boxtimes W_2\boxtimes W_3).
\]
But since
\[
U_{n_1,n_2}U_{n_1+n_2,n_3} = U_{n_2,n_3}U_{n_1,n_2+n_3}\text{ and } V_{n_1,n_2}V_{n_1+n_2,n_3} = V_{n_2,n_3}V_{n_1,n_2+n_3}\]
the two representations are isomorphic, as required.
\end{proof}

\subsection{Gradings}\label{sec:GR-gradings}

The algebra $\GR$ is $\N$-graded, by definition. We shall consider refinements of this grading. For each $i,n\geq 1$ let $[\Mat^i_n]$ denote the set of similarity classes of $n\times n$ matrices over $\o_i$, that is, $\GL_n(\o_i)$-orbits under the conjugation action.  Letting $[\Mat_0^i]\coloneq\{\ast\}$, we define $[\Mat^i]\coloneq \bigsqcup_{n\geq 0} [\Mat_n^i]$. The block-sum operation $$[\MM_1]\oplus [\MM_2] \coloneq  \left[\begin{smallmatrix} \MM_1 & 0 \\ 0 & \MM_2 \end{smallmatrix} \right]$$ makes $[\Mat^i]$ into a commutative monoid, with neutral element~{$\ast\in [\Mat_0^i]$}. In the sequel we will also consider a similar monoid for an unramified extension $\O/\o$, and we will denote it by $[\Mat^{\O,i}]$.

For each $n\geq 1$ and $1\leq i\leq \lfloor \l/2 \rfloor$ Clifford theory (see Section 6 of \cite{Isaacs}) gives a decomposition
\[
\Rep(G_n) \cong \prod_{[\MM]\in [\Mat_n^i]} \Rep(G_n)_{[\MM]},
\]
where by definition a representation $\rho$ of $G_n$ lies in $\Rep(G_n)_{[\MM]}$ if the restriction $\rho\restrict_{K^{\l-i}}$ is a sum of $G_n$-conjugates of the character $\phi^\star_{\MM}$ defined in \eqref{eq:phi.star}. These decompositions, together with the trivial decomposition for $n=0$, give a decomposition at the level of Grothendieck groups
\[
\GR = \bigoplus_{[\MM]\in [\Mat^i]} \GR_{[\MM]}.
\]
Note that we have natural surjective maps $[\Mat^{i+1}]\to [\Mat^i]$ and the grading by $[\Mat^{i+1}]$ is a refinement of the grading by $[\Mat^i]$. 
The distinct summands in the decomposition are orthogonal to one another with respect to the form $\langle \argument,\argument \rangle$. 
Notice that this grading is indeed a refinement of the $\N$-grading because the $\N$-grading can be recovered as the size of the matrix $[\MM]$. 

\begin{proposition}\label{prop:GR-Mat-grading}
\begin{enumerate}[\rm(1)]
\item For all $[\MM_1],[\MM_2]\in [\Mat^i]$ \uri{with $i \le \ell/2$} one has
\[
\GR_{[\MM_1]}\nmult \GR_{[\MM_2]} \subseteq \GR_{[\MM_1\oplus \MM_2]}.
\]
\item For all $[\MM]\in [\Mat^i]$ one has
\[
\Delta(\GR_{[\MM]}) \subseteq \bigoplus_{[\MM]=[\MM_1]\oplus [\MM_2]} \GR_{[\MM_1]}\otimes \GR_{[\MM_2]}.
\]
\end{enumerate}
\end{proposition}

\begin{proof}
By Proposition \ref{prop:GR-basics} the multiplication is dual to the comultiplication. It will thus be enough to prove the first assertion.
For $k=1,2$ write $n_k$ for the size of the matrix $\MM_k$, that is $[\MM_k]\in [\Mat^i_{n_k}]$. 
We follow \cite[Section 3.2]{CMO1}. We consider the normal subgroup $K^{l-i}$ (we avoid the terminology $G_0$ used in \cite[Section 3.2]{CMO1} to avoid confusion with the group $G_0$ we have here). 
It is easy to show that $(U_{n_1,n_2}\cap K^{l-i},G_{n_1,n_2}\cap K^{l-i},V_{n_1,n_2}\cap K^{l-i})$ is an Iwahori decomposition of $K^{l-i}$. 
The functor $\pind_0:\Irr(G_{n_1,n_2}\cap K^{l-i})\to \Irr(G_{n_1+n_2}\cap K^{l-i})$ of \cite[Lemma 3.3]{CMO1} sends the irreducible representation $[(\MM_1,\MM_2)]$ to $[\MM_1\oplus\MM_2]$\uri{; indeed, since the relevant group is abelian, the Iwahori decomposition is in fact a product of these groups and the functor $\pind_0$ is extension by \lq zeros\rq, namely the trivial representation}. Theorem 3.4 of \cite{CMO1} now gives us that 
\[\pind(\GR(G_{n_1,n_2})_{[(\MM_1,\MM_2)]})\subseteq \GR(G_{n_1+n_2})_{[\MM_1\oplus \MM_2]}
\]
where for a finite group $G$ we denote by $\GR(G)$ the Grothendieck group of $\Rep(G)$. 
But this just means that 
\[
\GR_{[\MM_1]}\nmult \GR_{[\MM_2]}\subseteq \GR_{[\MM_1\oplus \MM_2]}
\] as desired.
\end{proof}

We use now the $[\Mat]=[\Mat^1]$-grading on $\GR$ to construct a collection of rings and corings which are indexed by the irreducible polynomials in $\o_1[t]$. For each irreducible $\ol{f}\in \o_1[t]$ we define 
\[
\GR (\ol{f})\coloneq \bigoplus_{m\geq 0}\bigoplus_{\mathrm{char}[\uri{\ol{\MM}}]=\ol{f}^m} \GR _{[\ol{\MM}]}.
\]
Proposition \ref{prop:GR-Mat-grading} immediately implies that for every irreducible $\ol{f}$, $\GR(\ol{f})$ is a subring and a subcoring of $\GR$. 
The polynomial $\ol{f}$ also defines a submonoid $[\Mat^i(\ol{f})]\subseteq [\Mat^i]$ for every $i\geq 1$ in the following way:
$$[\Mat^i(\ol{f})] = \{[\MM]\in [\Mat^i] \mid \mathrm{char}(\ol{\MM}) = \ol{f}^m\text{ for some } m\in \N\}$$ where $\ol{\MM}$ is the reduction of $\MM$ mod $\pi$. 
For $1\leq i\leq \l/2$ we then get the following grading on $\GR(\ol{f})$:
$$\GR(\ol{f}) = \bigoplus_{[\MM]\in [\Mat^i(\ol{f})]}\GR_{[\MM]}.$$
This grading will be used in Section \ref{sec:base.change}.

\subsection{Changing $\l$}

The construction of the (co)ring $\GR^{\o_\l}$ is compatible with changing the level $\l$, in a very transparent way:

\begin{proposition}\label{prop:GR-changing-l}
For each $\l\geq 1$, \uri{t}he map $\GR^{\o,\l} \to \GR^{\o,\l+1}$ induced by the pullback functors $\Rep(\GL_n(\o_\l)) \to \Rep(\GL_n(\o_{\l+1}))$ restricts to an isomorphism of rings and corings $$\GR^{\o,\l} \xrightarrow{\cong} \bigoplus_{n\geq 0}\GR^{\o,{\l+1}}_{[\uri{0}_{n\times n}]}.$$
\end{proposition}

\begin{proof}
Since $\GL_n(\o_{\l+1})/K^\l\cong \GL_n(\o_\l)$ inflation provides an isomorphism between the Grothendieck group of $\GL_n(\o_\l)$ and the subgroup of the Grothendieck group of $\GL_n(\o_{\l+1})$ of all representation with trivial $K^l$ action. We need to show that this is an isomorphism of rings and corings. Again, since multiplication is dual to comultiplication (Proposition \ref{prop:GR-basics}) it is enough to prove this for the multiplication. 

Write $U=U_{n_1,n_2},V=V_{n_1,n_2}$ for the $U$ and $V$ subgroups of $\GL_{n_1+n_2}(\o_{\l})$ and write $\ol{U} = U/U\cap K^\l$ and $\ol{V} = V/V\cap K^\l$. By describing $\nmult$ using the functors $\pind_{U,V}$ and $\pind_{\ol{U},\ol{V}}$ we reduce to proving that the diagram 
\[
\xymatrix@C=60pt{
\Rep(\GL_{n_1}(\o_{\l+1})\times \GL_{n_2}(\o_{\l+1})) \ar[r]^-{\pind_{U,V}} & \Rep(\GL_{n_1+n_2}(\o_{\l+1})) \\
\Rep(\GL_{n_1}(\o_{\l})\times \GL_{n_2}(\o_{\l})) \ar[u]^-{\infl} \ar[r]^-{\pind_{\ol{U},\ol{V}}} & \Rep(\GL_{n_1+n_2}(\o_{\l})) \ar[u]_-{\infl} 
}
\] is commutative. But this is known to be true by \cite[Theorem 2.18 (5)]{CMO1}, where $K=K^\l$ and $L=\GL_{n_1}(\o_{\l+1})\times \GL_{n_2}(\o_{\l+1})$ so we are done. 
\end{proof}

\begin{remark}
We can also define a ring and coring $\GR^{\o}$, built from smooth representations of the profinite groups $\GL_n(\o)$ using the appropriate smooth versions of the functors $\pind$ and $\pres$. Using \cite[Theorem 2.18 (5)]{CMO1}, as in the proof of Proposition \ref{prop:GR-changing-l}, one can show that $\GR^{\o} \cong \varinjlim_{\l} \GR^{\o_\l}$ as rings and corings. Hence, in order to understand $\GR^{\o}$ it is enough to consider the finite quotients $\o_\l$, as we shall do for the rest of this paper.
\end{remark}

%%%%%%%%%%%%%%%%

\section{Decomposition to primary components}\label{sec:primary.decomposition}

Fix $\o$ and $\l\geq 2$, and write $G_n\coloneq \GL_n(\o_\l)$, $\GR = \GR^{\o,\l}$, etc., as explained in Section \ref{sec:GR}. The main result in this section is the following. 

\begin{theorem}\label{thm:primary.decomposition.of.reps}
Suppose that $[\ol{\MM}_1]\in [\Mat_{n_1}]$ and $[\ol{\MM}_2]\in [\Mat_{n_2}]$ are coprime. Then the functor $\pind_{n_1,n_2}$ restricts to an equivalence of categories
\[
\pind_{n_1,n_2}:\Rep(G_{n_1} )_{[\ol{\MM}_1]} \times \Rep(G_{n_2} )_{[\ol{\MM}_2]} \xrightarrow{\cong} \Rep(G_{n_1+n_2} )_{[\ol{\MM}_1\oplus \ol{\MM}_2]}.
\]
\end{theorem}

Before we prove the theorem we derive its consequences for the Grothendieck group $\GR$.

\begin{proof}[Proof of Theorem~\ref{thm:primary}] We need to prove that the multiplication $\nmult$ induces an isomorphism of rings and corings 
\[
\Phi:\bigotimes_{\substack{\ol{f}  \in \o_1[t]\\ \text{irreducible}}} \GR(\ol{f})\longrightarrow \GR
\]
that preserves the $\Z$-valued bilinear forms and the standard bases on either side. The tensor product of infinitely many unital rings means here the direct limit of the finite tensor products, see \cite[Chapter 2]{Zelevinsky}. 
Since we are considering commutative rings, the multiplication from a tensor product of subrings to the ring is always a ring homomorphism. Once we prove that this is a unitary isomorphism (i.e. it preserves the inner product) the duality between the multiplication and comultiplication will imply immediately that this is a homomorphism of corings as well. 
Let $\{\ol{f_0},\ol{f_1},\ldots\}$ be an enumeration of all the monic irreducible polynomials in $\o_1[t]$. 
From linear algebra we know that if the characteristic polynomial $f$ of a matrix $\MM$ over $\o_1$ splits as $f=\prod_i\ol{f_i}^{m_i}$, where almost all $m_i$ are zero, then $[\MM]$ splits uniquely as 
\[
[\MM]=\bigoplus_i [\MM_i]
\]
where for each $i$ the matrix $\MM_i$ has characteristic polynomial $\ol{f_i}^{m_i}$ (so almost all of the matrices $\MM_i$ are the $0\times 0$ matrix).
The grading of $\GR$ by $[\Mat]$ gives us then a grading on $\GR$ by the set 
\[
N:=\{(m_i)\in \N^{\N} \mid m_i=0\text{ for almost all } i\}
\]
We write this grading as
\[
\GR = \bigoplus_{(m_i)\in N}\GR_{(m_i)}.
\]
Similarly, $\GR(\ol{f_i})$ is graded by $\N$, where $\GR(\ol{f_i})_n$ contains all representations for which the associated matrix has characteristic polynomial $\ol{f_i}^n$. 

Proposition \ref{prop:GR-Mat-grading} implies that for every $(m_i)\in N$, 
\[
\Phi\big(\bigotimes_i \GR(\ol{f_i})_{m_i}\big)\subseteq \GR_{(m_i)}.
\]
By applying Theorem \ref{thm:primary.decomposition.of.reps} iteratively and using the fact that the Grothendieck group of the product of categories is the tensor product of the Grothendieck groups of the categories we get that 
\[
\Phi_{(m_i)}:\bigotimes_i \GR(\ol{f_i})_{m_i}\to \GR_{(m_i)}
\] is an isomorphism. 
Since $\Phi=\bigoplus_{(m_i)\in N}\Phi_{(m_i)},$ $\Phi$ is an isomorphism as well. Since $\Phi$ sends irreducible elements to irreducible elements it is also unitary as required. 
\end{proof}

To prove Theorem~\ref{thm:primary.decomposition.of.reps} we need the following lemma.  
 
 \begin{lemma}\label{reduction} Let $G$ be a finite group and let $L, V' \subseteq V, U' \subseteq U$ be subgroups of $G$ such that $(U,L,V)$ is a virtual Iwahori decomposition and $L$ normalises $U'$. 
Let $Y$ be an irreducible representation of $G$. Assume that the following conditions hold.
\begin{enumerate}

\item $L':=[V',U] \subseteq L$ and acts on $Y$ by a linear character $\psi$.

\item $V'$ and $U$ centralise $L'$.

\item $U'= \bigcap _{v \in V'} \mathrm{Ker}(\psi_v)$, where $\psi_v(u)=\psi(uvu^{-1}v^{-1})$.

\item $U$ is abelian and the map $V' \to (U/U')^\vee$ given by $v \mapsto \psi_v$ is surjective.

\end{enumerate}  

Then $e_Ve_UY=e_{V}e_{U'} Y$. Dually, if $X$ is an irreducible $L$-representation on which $L'$ acts via a linear character then $\C G e_U e_V \ot_L X \simeq \C G e_{U'} e_V \ot_L X$.  

\end{lemma}

\begin{proof} For every $v \in V'$ let $\gamma_v:U \to L'$ be the map $\gamma_v(u)=uvu^{-1}v^{-1}$. We note that the assumption that~$U$ centralises~$L'$ implies that~$\gamma_v$ is a group homomorphism and therefore $\psi_v=\psi\circ \gamma_v$ is a linear character of~$U$ for every~$v \in V'$. The assumption that $V'$ centralises $L'$ implies that the map $v \mapsto \psi_v$ defines a homomorphism~$V' \to U^\vee$. For $v \in V'$ let $e_{U}^{\psi_v}$ denote the idempotent in $\mathbb{C}U$ which corresponds to $\psi_v$. 

We claim that for every $v \in V'$ the operators $ve_Ue_V$ and $e_{U}^{\psi_v}$ coincide on~$Y$. Indeed, for every $y \in Y$ we have 
\[
\begin{split}
ve_Ue_V(y)&=ve_Uv^{-1}e_V(y)=\frac{1}{|U|}\sum_{u \in U}vuv^{-1}e_V(y)\\
&=\frac{1}{|U|}\sum_{u \in U}u\gamma_v(u^{-1})e_V(y)=\frac{1}{|U|}\sum_{u \in U}ue_V\gamma_v(u^{-1})(y)\\
&=\frac{1}{|U|}\sum_{u \in U}\psi_v(u^{-1})ue_V(y)=e_U^{\psi_v}e_V(y).
\end{split}
\]
It follows that as operators on $Y$ we have
\[
\begin{split}
e_{V'}e_Ue_V&=\frac{1}{|V'|}\sum_{v \in V'}ve_Ue_V=\frac{1}{|V'|}\sum_{v \in V'}e_U^{\psi_v}e_V\\
&=\frac{1}{|V'||U|}\sum_{u\in U}\sum_{v \in V'}\psi_v(u^{-1})ue_V\\
&=\frac{1}{|V'||U|}\Biggl(\sum_{u\in U'}\sum_{v \in V'}\underbrace{\psi_v(u^{-1})}_{=1}ue_V+\sum_{u\in U \smallsetminus U'}\underbrace{\Bigl(\sum_{v \in V'}\psi_v(u^{-1})\Bigr)}_{=0}ue_V\Biggr)\\
&=\frac{|U'|}{|U|}e_{U'}e_V.
\end{split}
\]
Multiplying the last equality by $e_V$ on the left and applying both sides to $Y$ gives
\begin{equation}\label{eqn:isom.of.spaces.induced.by.idemp1}
e_Ve_{U'}Y \simeq e_V\left(e_{U'}e_V Y\right)=e_V\left(e_{V'}e_Ue_VY\right)=e_Ve_{U}e_VY \simeq e_Ve_{U}Y.
\end{equation}
The first and last isomorphisms follow from~\eqref{commutativity}, the second equality arises from the calculation we have done here, and the third equality follows from the fact that $e_Ve_{V'}=e_V$ because $V'\subseteq V$. Since $e_Ve_U Y=e_Ve_{U'}(e_U Y) \subseteq e_Ve_{U'}Y$ we must have equality between the two subspaces.
\end{proof}

\begin{proof}[Proof of Theorem~\ref{thm:primary.decomposition.of.reps}] To simplify the notation we write $n\coloneq n_1+n_2$ and $\ol{\MM}\coloneq \ol{\MM}_1\oplus \ol{\MM}_2$. By \cite[Theorem 3.6]{CMO1}, the functor $\pind_{n_1,n_2}$ fits into a commuting diagram
\begin{equation}\label{eq:primary-diag}
\xymatrix@C=80pt{
\Rep(G_{n_1})_{[\ol{\MM}_1]} \times \Rep(G_{n_2})_{[\ol{\MM}_2]} \ar[d]_-{\cong} \ar[r]^-{\pind_{n_1,n_2}} & \Rep(G_{n })_{[\ol{\MM}]} \ar[d]^-{\cong} \\
\Rep(G_{n_1}(\ol{\MM}_1))_{[\ol{\MM}_1]} \times \Rep(G_{n_2}(\ol{\MM}_2))_{[\ol{\MM}_2]} \ar[r]^-{\pind_{U_{n_1,n_2}(\ol{\MM}),V_{n_1,n_2}(\ol{\MM})}} & \Rep(G_n(\ol{\MM}))_{[\ol{\MM}]}
}
\end{equation}
in which the vertical arrows are equivalences. So it will suffice to prove that the functor $\pind\coloneq \pind_{U_{n_1,n_2}(\MM), V_{n_1,n_2}(\MM)}$ appearing in the bottom row of \eqref{eq:primary-diag}  is an equivalence.

For $1 \le i < \l$, let $U^i = U_{n_1,n_2} \cap K^i_{n}$ and similarly $V^i=V_{n_1,n_2} \cap K^i_{n}$.
By Lemma~\ref{lem:coprime}, the assumption that $\ol{\MM}_1$ and $\ol{\MM}_2$ are coprime implies that $U_{n_1,n_2}(\ol{\MM})=U^1$,  $V_{n_1,n_2}(\ol{\MM})=V^1$, and  \uri{that}
\[
G_n(\ol{\MM})=U^1(G_{n_1}(\ol{\MM}_1)\times G_{n_2}(\ol{\MM}_2))V^1
\] 
\uri {is} an Iwahori decomposition. Now \cite[Theorem 2.23]{CMO1} implies that the functor $\pind$ is fully faithful; and also that this functor is essentially surjective if and only if $\pres_{U^1,V^1} X \neq 0$ for each  irreducible $X\in \Rep(G_n(\ol{\MM}))_{[\ol{\MM}]}$. We shall prove that the latter holds by showing that:
\begin{enumerate} 
\item[(i)] $e_{V^{i}} e_{U^{j}} {X} = e_{V^{i}} e_{U^{j+1}} {X}$ for all $i,j \ge 1$ with $i+j=\l-1$. In particular 
\[
e_{V^1} e_{U^1} {X}=e_{V^1} e_{V^2} \cdots e_{V^{\l-2}} e_{U^1} {X} = e_{V^1} {X},
\] 
as $U^{\l-1}\subset \ker \xi$ acts trivially on $X$.
\item[(ii)] $e_{V^1}{X} \ne 0$, that is, ${X}$ has a nontrivial $V^1$-invariant vector.
\end{enumerate}

\smallskip

\noindent {\em Proof of} (i): We shall use Lemma \ref{reduction}, with $G\coloneq G_n(\ol{\MM})$, $U'\coloneq U^{j+1}$, $U\coloneq U^{j}$, $V' \coloneq V \coloneq V^{i}$, and $L\coloneq G_{n_1}(\ol{\MM}_1)\times G_{n_2}(\ol{\MM}_2)$. We verify the hypotheses (1)--(4) of Lemma \ref{reduction}:  

For (1), we have $L'\coloneq [V^{i},U^{j}] = K^{\l-1}_n \cap L$, so clearly $L'\subseteq L$. On the other hand, since by assumption $K^{\l-1}_n$ acts on $X$ through the character $\psi\coloneq \psi_{\ol{\MM}}$, the same is true of $L'$. 

For (2), since $V^i$ and $U^j$ are contained in $K^1_n$, and $K^1_n$ centralizes $K^{\l-1}_n$, it is certainly true that $V^{i}$ and $U^j$ centralize $L'$. 

For (3) and (4), Lemma~\ref{lem:duality.between.U.and.V.over.F} ensures that $U^j/U^{j+1}$ and $V^i/V^{i+1}$ are Pontryagin duals of one another, with the duality realized by the pairing 
$(u,v)\mapsto \psi_v(u)$. This implies both (3) and (4).  

Thus Lemma \ref{reduction} applies, and establishes claim (i).

\smallskip

\noindent {\em Proof of} (ii): We will prove that $X$ contains a nonzero $V^i$-invariant vector,  for each $1\leq i \leq \l-1$, by induction on $-i$. For $i=\l-1$ it is in fact true that $V^{\l-1}$ acts trivially on $X$, since we assumed that $K_n^{\l-1}$ acts on $X$ through the character $\psi=\psi_{\ol{\MM}}$, which is trivial on the subgroup $V^{\l-1}$. 

Now assume that ${X}^{{V}^{i+1}}\neq 0$ and take a nonzero $w\in {X}^{{V}^{i+1}}$ upon which ${V}^i$ acts by a linear character. By the above-noted duality between $V^i/V^{i+1}$ and $U^j/U^{j+1}$ (where $i+j=\l-1$), there exists an element $u\in U^j$ (not unique) such that $v(w) = \psi_v(u)\cdot w = [u,v](w)$ for every $v\in V^i$. Now for each $v\in V^i$ we have
\[
vu(w) = [v,u] uv(w) = \psi([v,u])u\psi([u,v])(w) = u(w),
\]
showing that the nonzero vector $u(w)\in W$ is fixed by $V^i$. This proves (ii)  and completes the proof of the Theorem.
\end{proof} 
 
\begin{remark}\label{rem:CMO-vs-Hill1}
The bijection between the sets of isomorphism classes of irreducible representations induced by the equivalence
\[
\pind_{n_1,n_2}:\Rep(G_{n_1} )_{[\ol{\MM}_1]} \times \Rep(G_{n_2} )_{[\ol{\MM}_2]} \xrightarrow{\cong} \Rep(G_{n_1+n_2} )_{[\ol{\MM}_1\oplus \ol{\MM}_2]}
\]
of Theorem \ref{thm:primary.decomposition.of.reps} is the same as the bijection established by Hill in \cite[Corollary 2.6]{Hill_Jord}. This is most easily seen by considering the inverse of the bottom horizontal arrow in the diagram \eqref{eq:primary-diag}, which we showed in the course of proving Theorem \ref{thm:primary.decomposition.of.reps} to be given on irreducibles by $X\mapsto e_{U^1} e_{V^1} X$. On the other hand, examining the Hecke algebra isomorphism from \cite[Chapter 2, Theorem 4.1]{Howe-Moy} that is used by Hill in \cite[Corollary 2.6]{Hill_Jord}, one finds that the corresponding map for Hill's bijection is given by
\[
X\mapsto \begin{cases} e_{U^j}e_{V^j} X & \textrm{if }\l=2j, \\ 
e_{U^{j}}e_{V^{j+1}} X &\textrm{if } \l=2j+1. \end{cases}
\]

Using the claim (i) established in the proof of Theorem \ref{thm:primary.decomposition.of.reps} it is easy to see that the two maps coincide up to isomorphism. For example, in the case $\l=2j$ we can first use the equality from claim (i)  to obtain for all $1<  k \leq j$ that
\[
e_{U^k}e_{V^k} X = e_{U^k} e_{U^{\l-k}} e_{V^{k}} X = e_{U^k} e_{U^{\l-k}} e_{V^{k-1}} X = e_{U^k} e_{V^{k-1}} X \cong e_{V^{k-1}} e_{U^k} X,
\]
where the last isomorphism is \eqref{commutativity}. Then we can use the analogue of claim (i) with the roles of $U$ and $V$ interchanged to get
\[
e_{V^{k-1}} e_{U^k} X = e_{V^{k-1}} e_{V^{\l-k}}e_{U^k} X = e_{V^{k-1}} e_{V^{\l-k}} e_{U^{\uri{k}-1}} X = e_{V^{k-1}} e_{U^{k-1}} X \cong e_{U^{k-1}} e_{V^{k-1}} X,
\]
showing that indeed $e_{U^j} e_{V^j}X \cong e_{U^1} e_{V^1} X$. \uri{The case $\l=2j+1$ follows as well because $e_{U^j}e_{V^{j+1}} X = e_{U^j} e_{V^{j}} X$  using the case $i=j$ in claim (i).}

\end{remark}

\section{Base change}\label{sec:base.change}

In this section we study in detail the ring and coring $\GR^{\o,\l}(\ol{f})$, where $\ol{f}\in \o_1[t]$ is a monic irreducible polynomial of degree $d$, and prove Theorem~\ref{thm:base.change.iso}. Recall that $\GR^{\o,\l}(\ol{f})$ is the subgroup of $\GR^{\o,\l}$ generated (as an abelian group) by all irreducible representations that lie over orbits of characters of $K^{\l-1}$ represented by~$\ol{\MM} \in \Mat_n(\o_1)$ where $\ol{\MM}$ is $\ol{f}$-primary and $\nn=md$ with $m \in \N$. When $d=1$, that is $\ol{f}=t-\ol{x}$ with $\ol{x} \in \o_1$, the character $g \mapsto \phi\left(\pi^{-\ell}\ol{x}(g-1)\right)$ of~$K^{\l-1}_\nn$ extends to a linear character of $\GL_n(\o_\l)$ for all $\nn$, and tensoring with this character gives an isomorphism~$\GR^{\o,\l}(t-\ol{x}) \cong \GR^{\o,\l}(t)$ of rings and of corings. This enables us to reduce the study to representations with associated nilpotent matrices.

Our goal is to prove that a similar result holds also when $d>1$. 
We recall the setting mentioned in the introduction: we assume that $f\in \o[t]$ is a monic lifting of $\ol{f}$, that $x\in \Mat_d(\o)$ is a root of $f$ and that $\O=\o[x]$. This is an unramified extension of $\o$, and $\O_1 = \O/(\pi)$ is a field extension of $\o_1$ of degree $d$. 
We then consider the representation ring $\GR^{\O,\l}(t-\ol{x})$ which corresponds to $\O$. 

Our main result in this section is the following base change result (Theorem~\ref{thm:base.change.iso} of the introduction). 
\begin{theorem}\label{thm:base.change} Let $\o$ and $\O=\o[x]$ be as above. Then there exists an isomorphism $\GR^{\O,\l}(t-\ol{x}) \cong \GR^{\o,\l}(\ol{f})$ of rings and corings.   
\end{theorem}
Combining the theorem with the discussion above we obtain  
\begin{corollary}\label{cor.reduction.nilpotent} The ring $\GR^{\o,\l}(\ol{f})$ is isomorphic to $\GR^{\O,\l}(t)$.
\end{corollary}

The rest of this section is devoted to the proof of Theorem~\ref{thm:base.change}.  
Throughout we write $\l=2k+\epsilon$ with $\epsilon \in \{0,1\}$. We will treat the two cases separately. The proof for $\epsilon=0$ will be considerably simpler than the proof for $\epsilon=1$. The group $K_\nn^{k+\epsilon} \cong \left(\Mat_n(\o_k),+\right)$ is the maximal abelian congruence kernel of $\GL_n(\o_\l)$ and we can stratify $\Rep(\GL_n(\o_\l))$ according to their $K_\nn^{i}$-isotypic components for $k+\epsilon \le i \le \l$.
It follows from Clifford Theory (see \S\ref{sec:GR-gradings}) that we have the following splitting of the Grothendieck group \begin{equation}\label{eq:local.decomp.o}
\GR^{\o,\l}(\ol{f})=\bigoplus_{[\MM]\in [\Mat^k(\ol{f})]} \GR^{\o,\l}_{[\MM]},
\end{equation}
Note that 
over $\O_1$ the polynomial splits as 
$\ol{f}(t)=\prod_{\sigma \in \Gal(\O_1/\o_1)} (t-\sigma(\ol{x}))$, and we get a similar decomposition  
\begin{equation}\label{eq:local.decomp.O}
\GR^{\O,\l}(t-\ol{\xx})=\bigoplus_{[\MM]\in [\Mat^{\O,k}(t-\ol{x})]} \GR^{\O,\l}_{[\MM]},
\end{equation}
By \uri{Proposition}~\ref{intoT} we know that the inclusion $\Mat_m(\O_k) \subseteq \Mat_{\mm \dd}(\o_k)$ induces a bijection of orbits 
\[
\GL_{\mm}(\O_k)\backslash \Mat_{\mm}(\O_k)_{t-\ol{x}}  \simeq \GL_{\mm\dd}(\o_k)\backslash \Mat_{\mm \dd}(\o_k)_{\ol{f}},
\]
between orbits on the left whose reduction modulo $\pi$ is $(t-\ol{x})$-primary and orbits on the right whose reduction modulo $\pi$ is $\ol{f}(t)$-primary. Therefore, it is enough to prove that for every such $\MM$ we have $$\GR^{\o,\l}_{[\MM]}\cong \GR^{\O,\l}_{[\MM]}.$$ We separate now the proof for the case where $\l$ is even and where $\l$ is odd.

\begin{subsection}{Even level $\l=2k$.}\label{subsec:even}
Let $[\MM]\in [\Mat^k_{md}(\ol{f})]$. We first establish an equivalence 
\[
\cF: \Rep(\GL_{\mm\dd}(\o_\l))_{[\MM]}\to \Rep(\GL_{\mm}(\O_\l))_{[\MM]}  
\]
and then prove that it induces an isomorphism of rings and corings $\GR^{\o,\l}_{[\MM]} \to \GR^{\O,\l}_{[\MM]}$.

Let $W$ be an irreducible representation in $\Rep(\GL_{\mm\dd}(\o_\l))_{[\MM]}$. By \uri{Proposition}~\ref{intoT} $\GL_{\mm\dd}(\o_k)(\MM) = \GL_m(\O_k)(\MM)$ and we thus have a short exact sequence 
\[
1\to K^k \to \GL_{md}(\o_\l)(\MM)\to \GL_m(\T_k)(\MM)\to 1.
\]
By Clifford Theory the representation $W$ is induced from a representation $X$ of $\GL_{md}(\o_\l)(\MM)$ whose restriction to $K^k$ is represented by~$\MM$. Let $\cA=\Mat_\mm(\O_\l)$ and let $\cZ=\cA^\perp$ with respect to the trace pairing; see \S\ref{fieldext}. The direct sum decomposition $\Mat_{\mm\dd}(\o_\l) = \cZ \oplus \cA$ translates to a direct product  $K^k =\gZ^k \cdot \gA^k$, with $\gA^i=1+\cA^i$ and similarly $\gZ^i=1+\cZ^i$\uri{,  where  $\cA^i = \cA\cap \pi^i \Mat_n(\o_\l)$ and $\cZ^i = \cZ\cap \pi^i\Mat_n(\o_\l)$}. It follows that we can write $\GL_{md}(\o_\l)(\MM)= \gZ^k \cdot \GL_m(\O_\l)(\MM)$, where the intersection of the two groups is trivial.
Moreover, $\gZ^k$ is a normal subgroup of this group, and it acts trivially on $X$.
The representation $X$ can thus be considered as an irreducible representation of $\GL_\mm(\O_\l)(\MM)$
whose restriction to $\gA^k=K^k \cap \GL_\mm(\O_\l)$ is given by the character~$\MM$.
By Clifford Theory again, such irreducible representations are in one to one correspondence with
irreducible representations of $\GL_\mm(\O_\l)$ lying over the orbit of $\MM$.

In other words, we have established the following categorical equivalence:
\[
\cF: \Rep(\GL_{md}(\o_\l))_{[\MM]} \overset{e_\MM}{\longrightarrow} \Rep(\GL_{md}(\o_\l)(\MM))_\MM \overset{\inf^\star}{\longrightarrow} \Rep(\GL_\mm(\O_\l)(\MM))_\MM  \overset{\ind}{\longrightarrow}  \Rep(\GL_\mm(\O_\l))_{[\MM]}
\]
which induces an isomorphism of abelian groups $\GR^{\o,\l}_{md}(\ol{x})_{\uri{[\MM]}} \to \GR^{\O,\l}_m(\ol{x})_{\uri{[\MM]}}$. 
\uri{The middle equivalence is the quasi-inverse of the inflation functor  $\inf:\Rep(\GL_m(\O_l)(\MM))_{\MM}\to \Rep(\GL_{md}(\o_l)(\MM))_{\MM}$}. 

\begin{lemma} The functor $\cF$ induces an isomorphism of rings and corings.
\end{lemma}
\begin{proof}
It will be enough to prove that $\cF$ preserves multiplication, the proof for the comultiplication will follow by duality. 
Let $\MM=\MM_1\oplus \MM_2$, where $\MM_i\in [\Mat^k_{n_i}(\ol{f})]$.
Let $\nn=\nn_1+\nn_2$ and $\mm=\mm_1+\mm_2$ with  $\nn_i=\mm_i\dd$. Let $G=\GL_n(\o_\l)$ and let $L = G_1 \times G_2$ be the block diagonal subgroup in $G$ with $G_i=\GL_{\nn_i}(\o_\l)$. Let $U=U_{\nn_1,\nn_2}$ and $V=V_{\nn_1,\nn_2}$ be the corresponding upper and lower uni-triangular block matrices. Finally,  let $G'=\GL_\mm(\O_\l)$ and for every subgroup $H \leq G$ set $H'=H \cap G'$.

We need to show that the diagram
$$
\xymatrix{
\Rep(G_1)_{[\MM_1]}\times \Rep(G_2)_{[\MM_2]}\ar[rrr]^{ \qquad \qquad \pind_{U,V}}\ar[d]^{(e_{\MM_1},e_{\MM_2})} & & &  
\Rep(G)_{[\MM]}\ar[d]^{e_{\MM}} \\
\Rep((G_1(\MM_1))_{\MM_1}\times \Rep(G_2(\MM_2))_{\MM_2}\ar[rrr]^{\qquad \qquad \pind_{U(\MM),V(\MM)}}\ar[d]^{(\inf^{\star},\inf^{\star})} & & &
\Rep(G(\MM))_{\MM}\ar[d]^{\inf^{\star}} \\
\Rep(G_1'(\MM_1))_{\MM_1}\times \Rep(G_2'(\MM_2))_{\MM_2}\ar[rrr]^{\qquad \qquad \pind_{U'(\MM),V'(\MM)}}\ar[d]^{(\ind,\ind)} & & &
\Rep(G'(\MM))_{\MM}\ar[d]^{\ind}\\
\Rep(G_1')_{[\MM_1]}\times \Rep(G_2')_{[\MM_2]}\ar[rrr]^{\qquad \qquad \pind_{U',V'}}  & & &
\Rep(G')_{[\MM]}
}
$$
commutes. The upper and lower square commute due to the compatibility of Clifford Theory with the induction operations; \uri{see \S\ref{subsec:clifford}}. 
The middle square is an application of Theorem \ref{thm:inflation.coherence}, \uri{now with $G=\GL_n(\o_\l)(\xi)$ and $L=(G_1\times G_2)(\MM)$, $U=U(\MM)$, $V=V(\MM)$ and $\gZ=\gZ^k$.}
\end{proof}

\end{subsection}

%%%%%%%%%
%    Odd            %
%%%%%%%%%

\begin{subsection}{Odd level $\l=2k+1$}\label{subsec:odd}
The difficulty in the odd case arises from the fact that $K^k$ is not abelian. Let $W$ be an irreducible representation in $\Rep(\GL_{\mm\dd}(\o_\l))$ whose restriction to $K^{k+1}$ contains the character represented by $\MM \in \Mat_\mm(\O_k)$, and assume that the reduction $\ol{\MM}$ is $\ol{f}$-primary.  We can write $\GL_{md}(\o_\l)(\MM) = \GL_m(\O_\l)(\MM) \cdot K^k$, however, unlike in the even case, 
we cannot write this group as the product of $\GL_\mm(\O_\l)(\MM)$ and $\gZ^k=1+\cZ^k$, because $\gZ^k$ is not a subgroup anymore. Indeed, if $\aa, \bb \in \cZ$ then
\begin{equation}\label{calculationsZ}
(1+\pi^k \aa)(1+\pi^k \bb) = (1+\pi^k(\aa+\bb))(1+\pi^{2k}(\aa\bb)),
\end{equation}
and it may happen that $ab\notin \cZ$. Nevertheless, the image of the subset $\gZ^k$ modulo $K^{k+1}$ is a subgroup of $\GL_{md}(\o_{k+1})$ which is isomorphic to $(\ol{\cZ},+)\subseteq\Mat_{md}(\o_1)$, and the subset $\gZ^k$ is stable under conjugation by $\GL_{\mm d} (\o_l)(\MM)$.

We write
\begin{equation}\label{notation.G.N}
\begin{split}
G=&\GL_{m}(\O_{k+1})(\MM) \\
N=&\gZ^kK^{k+1}/K^{k+1} \cong  (\ol{\cZ},+).
\end{split}
\end{equation}
We have 
\[
\GL_{md}(\o_{k+1})(\MM) = \uri{ \GL_{m}(\O_{k+1})(\MM)K^k/K^{k+1} = \GL_{m}(\O_{k+1})(\MM)A^kZ^k/K^{k+1} =  }\GL_{m}(\O_{k+1})(\MM)N \cong G \ltimes \ol{\cZ}. 
\]
Note that $G$ depends on $\MM$ while $\cZ$ depends only on $\xx$, the semisimple part of $\MM$. We thus have the following commutative diagram, where the rows are short exact sequences, and the vertical arrows are inclusions:
\begin{equation}\label{troubles}
\xymatrix{
1 \ar[r] & K^{k+1} \ar[r]\ar[d]^{=} & \GL_m(\O_\l)(\MM)\ar[r]\ar@{^{(}->}[d] & G\ar[r]\ar[r]\ar@{^{(}->}[d]  & 1 \\ 
1 \ar[r] & K^{k+1} \ar[r]           & \GL_{md}(\o_\l)(\MM)\ar[r]   & GN\ar[r] & 1 
}
\end{equation}

Applying Clifford Theory (Subsection \ref{subsec:clifford}) to these exact sequences, there exist equivalences 
\begin{equation}\label{equiv.clifford.and.proj}
\begin{split}
\Rep(\GL_m(\O_\l)(\MM))_\MM &\cong \Rep^{\alpha}(G), \\
\Rep(\GL_{md}(\o_\l)(\MM))_\MM &\cong \Rep^{\beta}(GN), 
\end{split}
\end{equation}
for some two-cocycles $\alpha$ and $\beta$.  The fact that $\alpha$ is the restriction of $\beta$ follows easily from the commutativity of the above diagram.
So our goal in this section is to establish an equivalence $\Rep^{\beta}(G)\cong \Rep^{\beta}(GN)$. We first recall how exactly the cocycle $\beta$ arises from the data we have at hand. 
The character $\phi_{\MM}^{\star}:K^{k+1}\to \C^{\times}$ 
\uri{induces a morphism of short exact sequences 
\begin{equation}
\xymatrix{
1 \ar[r] & K^{k+1} \ar[r]\ar[d]^{\phi^{\star}_{\MM}} & \GL_{md}(\o_\l)(\MM)\ar[r]\ar@{..>}[d]^{\psi}   & GN\ar[r]\ar[d]^{=} & 1 \\  
1 \ar[r] & \C^{\times} \ar@{..>}[r]           & \Gamma\ar@{..>}[r]  & GN\ar[r] & 1 }
\end{equation}
which arises from applying the functor $H^2(GN,-)$ to the morphism of $GN$-modules $\phi_{\MM}^{\star}$; here $GN$ acts  on~$K^{k+1}$ by conjugation and trivially on $\C^\times$. The lower extension corresponds to the image of the cocyle which corresponds to the upper extension under the map $H^2(GN,K^{k+1}) \to H^2(GN,\C^\times)$.} The two-cocycle $\beta$ then arises by choosing a lifting $\wt{s}:GN\to\Gamma$ of the projection by the formula 
\[
\beta(g_1,g_2) = \wt{s}(g_1)\wt{s}(g_2)\wt{s}(g_1g_2)^{-1}\in \C^{\times}
\]
for $g_1,g_2 \in GN$. This is equivalent to choosing a lifting $s:GN\to \GL_{\mm d}(\o_l)(\MM)$ and defining 
\begin{equation}\label{def.of.beta}
\beta(g_1,g_2) = \phi_{\MM}^{\star}\left(s(g_1)s(g_2)s(g_1g_2)^{-1}\right).
\end{equation}
We have a certain freedom in choosing the two-cocycle $\beta$, namely, we may replace the lifting $\wt{s}$ (respectively the lifting $s$). We first choose a specific lifting, and prove that the resulting two-cocycle satisfies certain properties. As the multiplication in the twisted group algebra differs from the the group multiplication we write 
\[
\C^\beta G N = \Span_\C \left\{T_g \mid g \in  GN \right\}.
\]
We recall that the multiplication in $\C^{\beta}G N$ is given by the formula 
\[
T_{g_1}T_{g_2} = \beta(g_1,g_2)T_{g_1g_2}.
\]
The group $\Gamma$ can be realised as a subgroup of the multiplicative group of $\C^{\beta} GN$ by identifying the subgroup~$\C^{\times}$ with its natural embedding in $(\C^{\beta}GN)^{\times}$ and identifying $\wt{s}(g)$ with $T_g$.

\begin{lemma}\label{lem:choosebeta} Let $s: GN \to \GL_{\mm d}(\o_\l)(\MM)$ be a set-theoretic section such that $s(N) \subset \gZ^k$ and $s(gh)=s(g)s(h)$ for $g\in G$ and $h\in N$. Then 
% recall that \ff is a macro for the characters that can be change
\[
\beta((1+\pi^k\aa)K^{k+1} ,(1+\pi^k\bb)K^{k+1} ) = \phi_{\ol{\MM}}(\ol{\aa}\ol{\bb}), \quad \aa,\bb \in \cZ,
\]
where $(1+\pi^k\aa)K^{k+1} , (1+\pi^k\bb)K^{k+1}$ denote the images of the respective elements in $N$. In particular $\beta$ depends only on $\ol{\MM}$ and not on the particular choice of $s$. Moreover, in $\C^{\beta}G N$ the equality $T_gT_hT_g^{-1} = T_{ghg^{-1}}$ holds for all $g\in G$ and $h\in N$. This implies that the restriction of such $\beta$ to $N$ is $G$-invariant with respect to the diagonal conjugation action. 
\end{lemma}
\begin{proof} \uri{Note that such section exists because $G \cap N =\{1\}$, therefore we can choose values on $N$ and on $G$ independently and define $s(gh)=s(g)s(h)$ for $g\in G$ and $h\in N$.} The cocycle~$\beta$ is given by the formula $$\beta(h_1,h_2) = \phi^\star_\MM(s(h_1)s(h_2)s(h_1h_2)^{-1}), \quad \uri{h_1,h_2 \in GN,}$$
where $\phi^\star_\MM$ is the character of $K^{k+1} \cong \Mat_n(\o_{k})$ defined in Equation \eqref{eq:phi.star}.

For $(1+\pi^k\aa)K^{k+1}\in N$ we write $\wt{\aa}\in \cZ$ for an element which satisfies the equation $$s((1+\pi^k\aa)K^{k+1})= 1+\pi^k\wt{\aa}.$$ We use here the assumption that $s(N)\subseteq \gZ^k$. Since $s$ is a well defined function on $N$ we get that if $\aa=\bb\mod\pi$ then $\wt{\aa}=\wt{\bb}$. 

To prove the statement about the restriction of $\beta$ to $N$, let $h_i= (1+\pi^k \aa_i)K^{k+1} \in N$. 
Then $h_1h_2 = (1+\pi^k(\aa_1+\uri{\aa_2}))K^{k+1} $ in $N$.
There is a $c\in \cZ$ such that $\wt{\aa_1} + \wt{\aa_2} = \wt{\aa_1+\aa_2} + \pi c$.  
We calculate 
\[
\begin{split}
s(h_1)s(h_2) &= (1+\pi^k\wt{\aa_1})(1+\pi^k\wt{\aa_1}) =  (1+\pi^k(\wt{\aa_1}+\wt{\aa_2}) + \pi^{2k}\wt{\aa_1}\wt{\aa_2}) \\
& = (1+\pi^k(\widetilde{\aa_1+\aa_2}) + \pi^{k+1}c + \pi^{2k}\wt{\aa_1}\wt{\aa_2})\\
&= (1+\pi^{k+1}c)(1+\pi^{2k}\wt{\aa_1}\wt{\aa_2})(1+\pi^k(\wt{\aa_1+\aa_2})) \\
& = (1+\pi^{k+1}c)(1+\pi^{2k}\wt{\aa_1}\wt{\aa_2})s(h_1h_2).
\end{split}
\]
We then have 
\begin{equation}\label{eqn:beta.explicit}
\begin{split}
\beta(h_1,h_2) &= \phi_{\MM}^{\star}(s(h_1)s(h_2)s(h_1h_2)^{-1}) = \phi_{\MM}^{\star}((1+\pi^{k+1}c)(1+\pi^{2k}\wt{\aa_1}\wt{\aa_2})) \\
&= \phi_{\MM}^{\star}(1+\pi^{2k}\wt{\aa_1}\wt{\aa_2}) = \phi_{\ol{\MM}}(\ol{\aa_1}\ol{\aa_2}),
\end{split}
\end{equation}
where in the last part we have used the fact that $\phi_{\MM}^{\star}$ vanishes on $\gZ^{k+1}$ and that the restriction of $\phi_{\MM}^{\star}$ to $K^{2k}$ depends only on the reduction of $\MM$ and $\aa_i$  mod $\pi$.  
That proves the first part of the lemma.

For the second part of the lemma we need to prove the equality 
$$\psi(s(g)s(h)s(g)^{-1}) = \psi(s(ghg^{-1}))\text{  in } \Gamma$$ for $g\in G$ and $h\in N$. 
For this, we use the fact that conjugation by $s(g)$ stabilises the subset $\gZ^k$. \uri{By the choice of the section $s$, the elements $s(g)s(h)s(g)^{-1}$ and $s(ghg^{-1})$ belong to $\gZ^k$ and differ by an element of $K^{k+1}$}. Since $\phi_{\MM}^{\star}$ vanishes on $K^{k+1}\cap \gZ^k=\gZ^{k+1}$, the two elements $s(g)s(h)s(g)^{-1}$ and $s(ghg^{-1})$ have the same image in $\Gamma$, and the \uri{equality $T_gT_hT_g^{-1}=T_{ghg^{-1}}$ holds in $C^{\beta}G N$} as desired.
Finally the $G$-invariance follows from the equation
\[
\beta(gh_1g^{-1},gh_2g^{-1})T_{(gh_1g^{-1})(gh_2g^{-1})}=T_gT_{h_1}T_g^{-1}T_gT_{h_2}T_g^{-1}=T_{g}T_{h_1}T_{h_2}T_g^{-1}=\beta(h_1,h_2)T_{gh_1h_2g^{-1}},
\]
for all $h_1,h_2 \in N$.
\end{proof}
From now on we assume that $\beta=\beta(\MM,s)$ arises from a splitting $s$ which satisfies the conditions of the last lemma. \uri{This means, in particular, that $\beta$ is bimultiplicative}. 
Next, we have the following:
\begin{lemma}
The restriction of $\beta$ to the subgroup $N$ is non-degenerate. In other words, the twisted group algebra $\C^{\beta}N$ is isomorphic to a matrix algebra.
\end{lemma}
\begin{proof}
Since the twisted group algebra is semisimple, it is enough to prove that its centre is one-dimensional.
Let us write 
\[
\langle h_1,h_2 \rangle_\beta = \beta(h_1,h_2)/\beta(h_2,h_1), \quad h_i \in N.
\] 
Since the group $N$ is abelian, this is an alternating form which depends only on the cohomology class of $\beta$.
Moreover, in the group algebra $\C^{\beta} N$ we have
$$T_{h_1}T_{h_2}T_{h_1}^{-1} = \langle h_1,h_2\rangle_\beta T_{h_2}.$$
This implies that the centre of $\C^{\beta}N$ is spanned by all the $T_h$'s for which $h$ is contained in the radical of~$\langle \cdot,\cdot \rangle_{\beta}$.
The centre is thus one-dimensional if and only if the form is non-degenerate.
We can calculate the form explicitly here. Indeed, by the previous lemma we get that  
$$\langle 1+\pi^k \aa,1+\pi^k \bb \rangle_{\beta}=\ffbar{\ol{\MM}\ol{\aa}\ol{\bb}-\ol{\MM}\ol{\bb}\ol{\aa}}, \uri{\quad a, b \in \cZ,}$$
where $\ol{\phi}$ is the additive character
\begin{equation}
\ol{\phi}(\ol{\aa}) = \phi(\pi^{-1}\uri{\tr} (\aa)),
\end{equation} where $\aa\in \Mat_{\mm d}(\o)$ and $\ol{\aa}$ is its reduction mod $\pi$ (see \S \ref{subsec:Pontryagin.Dual}).

Assume that $(1+\pi^k\aa)$ is in the radical, that is, $\langle 1+\pi^ka,1+\pi^kb \rangle_{\uri{\beta}}=1$ for all $b\in \cZ$.
This is equivalent to $$\tr(\MM ab -\MM ba)=\tr((\MM \aa-\aa\MM)\bb)=0 \mod \pi, \quad \forall b \in \cZ.$$ 
It follows that $\ad_{\ol{\MM}}(\ol{\aa})=[\ol{\MM},\ol{\aa}] \in \ol{\cZ}^{\perp}=\uri{\ol{\cA}},$
and using the fact that $\ol{\cZ}$ is an $\ol{\cA}$-bimodule we get that $\ad_{\ol{\MM}}(\ol{a}) \in \ol{\cZ}$. This implies that $\ad_{\ol{\MM}}(\ol{a})=0$ and therefore $\ol{\aa} \in \ol{\cA} \cap \ol{\cZ} = 0$,
so the image of $1+\pi^k\aa$ in $N$ is trivial, as desired.
\end{proof}

For every $g\in G$, conjugation by $T_g$ defines an algebra automorphism of $\C^{\beta}N$. By the Skolem-Noether Theorem 
for every $g\in G$ there is an element $\LL_g\in \C^{\beta}N$  such that 
\begin{equation}\label{Mgeq} 
\LL_gT_h\LL_g^{-1}=T_gT_hT_g^{-1} = T_{ghg^{-1}}
\end{equation} for every $h\in N$.
The element $\LL_g$ is defined up to a non-zero scalar which will be fixed in Lemma \ref{lem:Sg.formula}. As a result there exists a two-cocycle $\gamma$ on $G$ such that 
$$\LL_{g_1}\LL_{g_2} = \gamma(g_1,g_2)\LL_{g_1g_2}.$$
We claim the following:
\begin{lemma}\label{isomorphismphi}
The map $\C^{\beta}G N\to \C^{\beta\gamma^{-1}}G\ot \C^{\beta}N$ given by $T_gT_h\mapsto T_g\ot \LL_gT_h$ is an algebra isomorphism.
Its inverse is given by $T_g\ot T_h\mapsto T_g\LL_g^{-1}T_h$. 
\end{lemma}
\begin{proof}
It is easy to see that the two maps above are linear and inverse to each other. 
The fact that they are algebra morphisms is a direct verification, using the properties of $\LL_g$ and $\gamma$.
\end{proof}
This lemma has the following corollary, which is already very close to what we want:
\begin{corollary}\label{cor:eq.categories}
We have an equivalence of categories $\Rep^{\beta}(G N)\cong \Rep^{\beta\gamma^{-1}}(G)$. 
\end{corollary}
\begin{proof}
Since the algebra $\C^{\beta} G N$ is isomorphic to $\C^{\beta\gamma^{-1}}G\ot \C^{\beta}N$, their representation categories are equivalent.
On the other hand, $\C^{\beta} N$ has a unique irreducible representation, and we therefore see that the representation category of the second algebra 
is equivalent to the representation category of $\C^{\beta\gamma^{-1}}G$, where the equivalence $\Rep^{\beta\gamma^{-1}}(G)\to \Rep(\C^{\beta\gamma^{-1}}G\ot \C^{\beta}N)$
is realised by mapping a representation $W$ to the representation $W\ot I$, where $I$ is the unique irreducible representation of $\C^{\beta} N$.
\end{proof}

The desired equivalence $\Rep^{\beta}(GN)\cong \Rep^{\beta}G$ would follow once we show that $\gamma$ is cohomologous to the trivial cocycle. 

\begin{lemma-definition}\label{lem:Sg.formula} For every $g\in G$ the element $\LL_g$ given by  
\begin{equation}\label{Mgformula} \LL_g = \sum_{a\in N}T_{gag^{-1}}T_a^{-1}=\sum_{a \in N} \beta([g,a],a^{-1})T_{[g,a]}.
\end{equation}
is an invertible element which satisfies Equation \eqref{Mgeq}.
\end{lemma-definition}
\begin{proof}
First notice that since $T_gT_hT_g^{-1}=T_{ghg^{-1}}$ the coefficient of $T_1$ in $S_g$ is the cardinality of the fixed point subspace of $G$ in $N$.  This implies that $\LL_g\neq 0$. We will first prove that $\LL_gT_h = T_{ghg^{-1}}\LL_g$ for all $h\in N$. This will imply that $\LL_g$ is invertible for all $g\in G$. Indeed, the Skolem-Noether Theorem implies that there are invertible elements $M_g$ such that $M_gT_hM_g^{-1}=T_{ghg^{-1}}$. The elements $M_g^{-1}\LL_g$ are then non-zero and central and therefore invertible, since $\C^{\beta}N$ is isomorphic to a matrix algebra. This implies that the elements $\LL_g$ are invertible as well. 

We continue with proving the formula $S_gT_{h} = T_{ghg^{-1}}S_g$
for $g\in G$ and $h\in N$. 
We multiply the expression $\sum_{a\in N}T_{gag^{-1}}T_{a}^{-1}$ on the right by $T_h$ and on the left by $T_{ghg^{-1}}$ and compute 
\[
\begin{split}
\Big(\sum_{a\in N}T_{gag^{-1}}T_{a}^{-1}\Big)T_h&=\sum_{a\in N}\underbrace{\beta(a,a^{-1})^{-1}\beta(a^{-1},h)}_{\beta(a^{-1},a^{-1}h)}\beta(gag^{-1},a^{-1}h)T_{gag^{-1}a^{-1}h} \\
&=\sum_{a\in N}\beta([g,a],a^{-1}h)T_{[g,a]h}, \\ \\
T_{ghg^{-1}}\Big(\sum_{a\in N}T_{gag^{-1}}T_{a}^{-1}\big)&=\sum_{a\in N}\underbrace{\beta(ghg^{-1},gag^{-1})}_{\beta(h,a)}\underbrace{\beta(a,a^{-1})^{-1}\beta(ghag^{-1},a^{-1})}_{\beta(ghag^{-1}a^{-1},a^{-1})}T_{ghag^{-1}a^{-1}} \\
&=\sum_{a\in N}\beta([g,ha],a^{-1})T_{[g,ha]h}.
\end{split}
\]
We have used the fact that $\beta$ is bimultiplicative on $N$ and $G$-invariant. This follows from the explicit formula for $\beta$ in Lemma \ref{lem:choosebeta}. The two sums are equal and~\eqref{Mgeq} follows.  The second form of $S_g$ follows directly from the first one.
\end{proof}

This enables us, in principle, also to calculate $\gamma$. Indeed, a direct calculation based on comparison of the coefficients of $T_1$ reveals the fact that 
$$\gamma(g_1,g_2) = \frac{1}{|C_N(g_1g_2)|}\sum_{\substack{a,b\in N ~\text{such that} \\ [g_1,a][g_2,b]=1}}\beta([g_1,a],a^{-1})\beta([g_2,b],b^{-1})$$
where for any $g\in G$ we write $C_N(g)=\{h\in N|ghg^{-1}=h\}$.
Unfortunately, it is not so clear from this formula why this cocycle is trivial. Instead, we shall prove that $\gamma$ is trivial by using tools from group cohomology.
We first claim the following:
\begin{lemma}
The order of $\gamma$ is a power of $p$.
\end{lemma}
\begin{proof}
By taking determinants we get the equations $$\det(\LL_{g_1})\det(\LL_{g_2}) = \gamma^{\sqrt{|N|}}(g_1,g_2)\det(\LL_{g_1g_2}),$$
for every $g_1,g_2\in G$. This is a trivialisation of the cocycle $\gamma^{\sqrt{|N|}}$. 
But $\sqrt{|N|}$ is a power of $p$, and so the order of $\gamma$ is a power of $p$ as well.
\end{proof}
Next, we have the following claim:
\begin{lemma}
The cohomology class of $\gamma$ is trivial.
\end{lemma}
\begin{proof} We prove that the restriction of $\gamma$ to some $p$-Sylow subgroup $P$ of $G$ is cohomologically trivial.
By using the fact that $\mathrm{cores}^G_P\mathrm{res}^G_P = |G:P|$ \uri{(cf. \cite[Chapter III, \S9]{Brown1994})} and the fact that the order of $\gamma$ is a power of~$p$, the lemma would follow.

We first notice that $\gamma$ is, in fact, inflated from the first congruence quotient $$\ol{G}:=G/(G\cap (K^1/K^{k+1})).$$ 
Indeed, this follows from the triviality of the action of $G\cap K^1/K^{k+1}$ on $\C^{\beta}N$ which implies that $S_g$ is a scalar matrix for all $g\in G\cap K^1/K^{k+1}$.
Let us write $\ol{\MM}=\ol{\s}+\ol{\n}$ for the semisimple-nilpotent decomposition of the reduction of $\MM$ modulo $\pi$.
Consider the subgroup $\langle 1+\ol{\nu}\rangle\subseteq \ol{G}$. 
This is a $p$-subgroup, and is therefore contained in a $p$-Sylow subgroup \uri{$\ol{P}$} of $\ol{G}$.
The group $\ol{G}$ acts naturally on the $\O_1$ vector space $\O_1^m$.
Since \uri{$\ol{P}$}  is a $p$-group we may assume, by choosing a suitable basis of $\O_1^m$, that \uri{$\ol{P}$}  is contained in the group of unipotent upper triangular matrices\uri{, which is a maximal $p$-Sylow subgroup of $\GL_m(\O_1)$}. \uri{The $p$-Sylow subgroup $P$ that we look at is the inverse image of $\ol{P}$, but as its action on $N$ factors through $\ol{P}$ and the two-cocycle $\gamma$ is inflated from $\ol{P}$, from now on we work with $\ol{P}$ and continue to denote the two-cocycle by $\gamma$.}  

\uri{Our strategy for proving that $\gamma|_{\ol{P}}$ is trivial is to} show that there exists a primitive idempotent $e \in \C^{\beta}N$ which commutes with $\LL_g$ for all $g\in \ol{P}$. Since $e$ is primitive, the image of $e$ in the unique irreducible representation $I$ of $\C^{\beta}N$ is one-dimensional. If we denote by $u$ a basis vector for the image of $e$, we  get that for every $g\in \ol{P}$, $\LL_g u= \lambda(g)u$ for some non-zero scalar~$\lambda(g)$. This will imply that $\partial\lambda = \gamma$ on $\ol{P}$, and therefore that $\gamma|_{\ol{P}}$ is trivial, as desired.

%$\ol{P} \ltimes \ol{\cZ} \supset \ol{P} \ltimes \ol{\cL}$ 

To do so, we find a \textit{Lagrangian} subgroup $L$ inside $N$ which is stable under the action of $\uri{\ol{P}}$. 
\uri{By definition, and more conveniently phrased in terms of $\ol{\cZ} \cong N$, this is a subgroup $\ol{\cL}$ of $\ol{\cZ}$ such that $|\ol{\cL}|^2=|\ol{\cZ}|$ and $\ol{\beta}|_{\ol{\cL}}$ is trivial, where $\ol{\beta}$ is the pullback of $\beta$ along the isomorphism between $\ol{\cZ} \times \ol{\cZ}$ and $N \times N$.   
We write $\ol{\cZ}=\oplus_{i,j=1}^m \ol{\cZ}_{i,j} \subset \Mat_{md}(\o_1)$,} where $\ol{\cZ}_{i,j}$ stands for a copy of $\O_1^\perp \subset \Mat_d(\o_1)$ in the $(i,j)$-block of the matrix.
The restriction of $\uri{\ol{\beta}}$ to $\ol{\cZ}_{i,i}$ is non-degenerate (and is actually independent of $i$).
Indeed, the restriction of $\uri{\ol{\beta}}$ to $\ol{\cZ}_{i,i}\cong \O_1^{\perp}$ is given by 
$$(\ol{\aa},\ol{\bb})\mapsto \ol{\phi}(\ol{x}\ol{\aa}\ol{\bb})$$ 
and this gives a non-degenerate two-cocycle.  
This follows from the fact that $1+\ol{\n}$ is in $\uri{\ol{P}}$ and $\ol{\n}$ is therefore an upper triangular matrix and therefore does not affect the value of $\uri{\ol{\beta}}|_{\cZ_{i,i}}$. For every $i$ we choose a Lagrangian $\ol{\cL}_i\subseteq \ol{\cZ}_{i,i}$.
We then define 
\begin{equation}\label{eq:decomposition.L}
\ol{\cL} = \Big(\bigoplus_{i<j}\ol{\cZ}_{i,j}\Big)\bigoplus \Big(\bigoplus_i \ol{\cL}_i\Big).
\end{equation}
 This is a Lagrangian inside $\ol{\cZ}$. Indeed, a direct calculation using the fact that the matrix $\ol{\MM}$ is upper triangular reveals the fact that for $(z_{i,j}),(w_{i,j})\in \ol{\cL}$ we have 
\begin{equation}\label{eq:formofbeta} \uri{\ol{\beta}}((z_{i,j}),(w_{i,j})) = \prod_i\uri{\ol{\beta}}(z_{i,i},w_{i,i}),\end{equation}
and as $z_{i,i}$ and $w_{i,i}$ are contained in the Lagrangian $\ol{\cL}_i$ we get the desired result.
Since the group $\uri{\ol{P}}$ is contained in the subgroup of block  upper triangular matrices, it follows that \uri{$\ol{\cL}$} is stable under conjugation by~$\uri{\ol{P}}$.

\uri{For $i < j$, let $E_{i,j} = \frac{1}{|\cZ_{\uri{i,j}}|}\sum_{z\in \ol{\cZ}_{i,j}}T_z$, and note that it generates the trivial $\C^{\ol{\beta}}\cZ_{i,j}$-module because $\ol{\beta}|_{\ol{\cZ}_{i,j}} \equiv 1$. For every $i$, let $E_{i,i}\in \C^{\uri{\ol{\beta}}}\ol{\cL}_i \uri{\cong \C \ol{\cL}_i}$ be a \uri{primitive idempotent corresponding to a one-dimensional representation}.}
Consider $e=\prod_{i\leq j} E_{i,j}$. 
We claim that $e$ is a primitive idempotent which is stable under conjugation by~$\uri{\ol{P}}$. The fact that it is a primitive idempotent follows directly from the fact that $\ol{\cL}$ is a Lagrangian subgroup and that it has the decomposition \eqref{eq:decomposition.L}, \uri{therefore $\C e$ is a one-dimensional $\C^{\ol{\beta}}\cL$-module}. To prove stability under conjugation by $\uri{\ol{P}}$ we proceed as follows:
Since for every $g\in \uri{\ol{P}}$ and $z\in \ol{\cZ}$ we have that $T_gT_zT_g^{-1}=T_{gzg^{-1}}$ and since $\uri{\ol{P}}$ is contained in the group of block upper triangular matrices we get that $g$ stabilises the idempotent $\prod_{i<j}E_{i,j}$. 
For $z\in \ol{\cL}_i$ we have $gzg^{-1} = z + \sum_{j<k} z_{j,k}$. Since \uri{$\ol{\beta}(z,w)=1$ when both $z$ and $w$} are block upper triangular matrices and at least one of them is strictly block upper triangular we get 
$$T_gT_zT_g^{-1} = T_{z+\sum_{j<k}z_{j,k}} = T_z\prod_{j<k}T_{z_{j,k}}.$$
This implies that $T_gT_z\prod_{j<k}E_{j,k}T_{g^{-1}} = T_z\prod_{j<k}E_{j,k}$. Since $E_{i,i}$ is a linear combination of $T_z$ for $z\in \ol{\cL}_i$ for every $i$ the idempotent $\prod_{i\leq j}E_{i,j}$ is stable under conjugation by $g$, as desired.
\end{proof}

%%%%%%%%%

\begin{remark} In case the degree $d$ of the Galois extension $\O_1/\o_1$ is odd we can find a Lagrangian in $\C^{\beta}N$ invariant under the action of the group $G$ directly, and not only for the $p$-Sylow subgroup. This will give a more direct proof of the fact that $\gamma$ is trivial, without using restriction and corestriction argument.
The Lagrangian can be found in the following way: The algebra of $d\times d$ matrices is known to be isomorphic to a crossed product: $\Mat_d(\o_1)\cong \uri{\O_1*\Gal(\O_1/\o_1)}$. The crossed product algebra has an $\O_1$ basis given by $(U_{\sigma})_{\sigma\in\Gal(\O_1/\o_1)}$, and the multiplication is given by 
$$y_1U_{\sigma}y_2U_{\tau} =  y_1\sigma(y_2)U_{\sigma\tau}.$$
Under the identification $\Mat_{md}(\o_1)\cong \Mat_m(\Mat_d(\o_1))$ the algebra $\cA$ corresponds to $\O_1U_1$ and the $\cA$-bimodule $\cZ$ corresponds to $\oplus_{\sigma\neq 1}\O_1U_{\sigma}$. 
Moreover, the trace pairing in $\Mat_d(\o_1)$ is trivial on $\O_1U_{\sigma}\times \O_1U_{\tau}$ if and only if $\sigma\tau\neq 1$. 
If $d$ is odd, write $\sigma$ for a generator of the Galois group of $\O_1/\o_1$, and define $$\ti{\cL}= \bigoplus_{i=1}^{(d-1)/2}\O_1U_{\sigma^i}.$$ 
Let now $\cL$ be the image of $\Mat_m(\ti{\cL})$ in $N$. Then the above discussion shows that $\cL$ is a Lagrangian in $N$ which is stable under the action of the group $\GL_m(\O_1)$ by conjugation, and therefore also under the action of $G$. In case the degree $d$ is even, however, there are examples where a $G$-stable Lagrangian in $N$ does not exist.
\end{remark}

Since $\gamma$ is trivial, we can choose scalars $\lambda(g)$ for every $g\in G$ such that the elements $\TL_g:=\lambda(g)\LL_g$ satisfy 
$\TL_{g_1}\TL_{g_2} = \TL_{g_1g_2}$ for every $g_1,g_2\in G$.
We have some liberty in the choice of $\lambda:G\to \C^{\times}$. Indeed, we can translate $\lambda$ by a one-dimensional character of $G$. The possible choices of $\lambda$ thus form a torsor over $\Hom_{Grp}(G,\C^{\times})$. 
For a specific $\lambda$ we will denote by $\Phi_\MM^{\lambda}$ the resulting isomorphism $\C^{\beta}G N\to \C^{\beta} G \ot \C^{\beta} N$ of Lemma \ref{isomorphismphi}, suitably adjusted. Translating $\lambda$ by a one-dimensional character of $G$ will alter the above categorical equivalence by tensoring with the one-dimensional representation. We will call $(\TL_g)$ a \textit{trivialisation} of $\gamma$. In the next pages we would like to emphasize that $\TL_g$ and $\gamma$ arise from the matrix $\MM$ and we will write $(\TL_g)=(\TL^{\MM}_g)$ and $\gamma=\gamma_{\MM}$. 

Summing everything up, we have the following proposition:
\begin{proposition}\label{prop:iso.abelian}
We have an isomorphism $\GR^{\o,\l}(\ol{f}) \cong \GR^{\O,\l}(t-\ol{x})$ of abelian groups.
\end{proposition}
\begin{proof}
The last lemma together with \eqref{equiv.clifford.and.proj} and Corollary \ref{cor:eq.categories} enables us to construct such an isomorphism. 
Indeed, we have the following composition of categorical equivalences 
$$\Rep(\GL_{md}(\o_\l))_{[\MM]}\stackrel{e_{\MM}}{\to} \Rep(\GL_{md}(\o_\l)(\MM))_\MM \to \Rep^{\beta}(G N) \to \Rep^{\beta\gamma^{-1}}(G)\to $$ $$\Rep^{\beta}(G)\to \Rep(\GL_m(\O_\l)(\MM))_\MM\stackrel{\ind}{\to} \Rep(\GL_m(\O_{\l}))_{[\MM]}$$
where the first, second, fifth and sixth equivalences arise from Clifford Theory, the third equivalence is Corollary \ref{cor:eq.categories}, and the fourth equivalence uses the trivialisation of $\gamma$, and depends on the particular trivialisation. 
By considering this equivalence on the level of Grothendieck groups we get $\GR^{\o,\l}_{[\MM]} \cong \GR^{\O,\l}_{[\MM]}$.
Taking the direct sum of all these isomorphisms together, we get the desired isomorphism $\GR^{\o,\l}(\ol{f})\cong \GR^{\O,\l}(t-\ol{x})$.
\end{proof}
The isomorphism $\GR^{\o,\l}(\ol{f})\cong \GR^{\O,\l}(t-\ol{x})$ depends on a choice of trivialisations $(\TL^{\MM}_g)$ of $\gamma_{\MM}$ for every $[\MM]\in [\Mat^k(\ol{f})]$. We claim the following: 

\begin{theorem}\label{oddhomomorphism}
The trivialisations $(\TL^{\MM}_g)$ can be chosen in such a way that the resulting map $$\GR^{\o,\l}(\ol{f}) \to \GR^{\O,\l}(t-\ol{x})$$ is an isomorphism of rings and corings.
\end{theorem}
The rest of the section is devoted to proving Theorem~\ref{oddhomomorphism}. We start with a few lemmas.
Assume that we have two matrices $\MM_i\in \Mat_{\mm_i}(\O_k)$ ($i=1,2$) whose reductions modulo $\pi$ have $\ol{f}$-primary characteristic polynomials.
Let $\MM=\MM_1\oplus \MM_2$ and let $\nn_i=\mm_i d$. We keep the same notation as in~\eqref{notation.G.N}, adding subscript $i$ for the groups related to $\MM_i$. We decompose the subgroup~$N$ of $\GL_{md}(\o_{k+1})(\MM)$ where $n=md = m_1d+m_2d$ to the direct product $N=N_{1}\times N_{12}\times N_{21}\times N_{2}$
with respect to the block decomposition $\nn=\nn_1+\nn_2$\uri{; the subgroups $N_1$ and $N_2$ stand for the diagonal blocks, and the subgroups $N_{12}$ and $N_{21}$ for the off-diagonal blocks}. 

\begin{lemma} \label{lem:splittingN}
Let $N$ be a finite abelian group, and let $\beta$ be a non-degenerate two-cocycle on $N$. Assume that $N=N_E\times N_O$,
and that $\langle N_E,N_O\rangle_{\beta} = 1$. Then $\C^{\beta}N\cong \C^{\beta}N_E\ot \C^{\beta}N_O$. If $\Theta: \C^{\beta}N\to \C^{\beta}N$ 
is an algebra automorphism such that $\Theta(\C^{\beta}N_i)\subseteq \C^{\beta}N_i$ for $i\in\{O,E\}$
then $\Theta$ is given by conjugation by an invertible element of the form $M_E M_O$ where $M_i\in \C^{\beta}N_i$ for $i\in\{E,O\}$. 
The elements $M_E$ and $M_O$ are unique up to multiplication by a non-zero scalar.
\end{lemma}
\begin{proof}
The fact that $\langle N_E,N_O\rangle_{\beta}=1$ implies that the multiplication map $\C^{\beta} N_E\ot \C^{\beta}N_O\to \C^{\beta}N$ is an algebra homomorphism. It is also surjective, so by dimension considerations it is an isomorphism of algebras. If $\Theta$ is an automorphism of $\C^{\beta}N$ which satisfies the above condition then Skolem-Noether Theorem applied to $\Theta|_{\C^{\beta} N_i}$ for $i\in\{O,E\}$ gives us the invertible elements $M_E$ and $M_O$.
\end{proof}

The following lemma is crucial in proving the compatibility of the isomorphism we have with the product $\nmult$. 

%%%%%%%%%%%%

\begin{lemma} \label{lem:reducingelements}
Let $u\in U'(\MM) = U_{\nn_1,\nn_2}(\MM)\cap \GL_{m}(\O_{k+1})$ and $e_{12} = \frac{1}{|N_{12}|}\sum_{h\in N_{12}}T_h\in \C^{\beta}N_{12}$. Then $\TL^{\MM}_u e_{12} = e_{12}$. A similar claim holds for the $V$-subgroup and the idempotent $e_{21}$.  
\end{lemma}
\begin{proof}
Since the restriction of $\beta$ to $N_{12}$ is trivial as a function (this follows easily from general properties of the trace), we see that $e_{12}$ is an idempotent. 
We will first prove that $\TL^{\MM}_u e_{12} = c(u)e_{12}$ for some non-zero scalar $c(u)$ and for every $u\in U'(\MM)$. 
Since $\TL^{\MM}_u$ is proportional to $$\sum_{h\in N}\beta([u,h],h^{-1})T_{[u,h]}$$ it will be enough to prove that statement for the last element.
We write a general element $h$ of $N$ as a $2\times 2$ block matrix of the form $h=(1+\pi^k z)K^{k+1}$ where $$z=\begin{pmatrix} z_{11} & z_{12} \\ z_{21} & z_{22}\end{pmatrix}\in \cZ.$$
We have 
\begin{equation}\label{eqn:explicit.commutator}
\begin{split}
[u,h] & = \begin{pmatrix} 1 & \aa \\ 0 & 1 \end{pmatrix}\left(1+\pi^k\begin{pmatrix} z_{11} & z_{12} \\ z_{21} & z_{22}\end{pmatrix}\right)\begin{pmatrix} 1 & -a \\ 0 & 1\end{pmatrix}\left(1-\pi^k\begin{pmatrix}z_{11} & z_{12} \\ z_{21} & z_{22}\end{pmatrix}\right) \\
& = 1+\pi^k\begin{pmatrix}az_{21} & -z_{11}a - az_{21}a + az_{22} \\ 0 & -z_{21}a\end{pmatrix} \mod K^{k+1}.
\end{split}
\end{equation}
\uri{Substituting \eqref{eqn:explicit.commutator} in the explicit formula for $\beta$ given in \uri{\eqref{eqn:beta.explicit}}, we get}
$$\beta([u,h],h^{-1}) = \ol{\phi}(\ol{\MM}\ol{b}),$$
where 
%$$b=\begin{pmatrix}az_{21} & -z_{11}a - az_{21}a + az_{22} \\ 0 & -z_{21}a\end{pmatrix}\cdot \begin{pmatrix}-z_{11}&  -z_{12} \\ -z_{21} & -z_{22}\end{pmatrix}=$$
$$b=\begin{pmatrix} -az_{21}z_{11} + z_{11}az_{21} + az_{21}az_{21} - az_{22}z_{21} & -az_{21}z_{12} + z_{11}az_{22} + az_{21}az_{22} - az_{22}^2 \\ 
z_{21}az_{21} & z_{21}az_{22}\end{pmatrix}.$$
Using the fact that $\MM=\MM_1\oplus \MM_2$ we get 
$$\beta([u,h],h^{-1}) = \phi_{\ol{\MM_1}}(\ol{-az_{21}z_{11} + z_{11}az_{21} + az_{21}az_{21} - az_{22}z_{21}})\phi_{\ol{\MM_2}}(\ol{z_{21}az_{22}}).$$
Using the definition of $\phi_{\ol{\MM_i}}$, the last expression can be written as 
$$\ol{\phi}(\ol{z_{11}(az_{21}\MM_1-\MM_1az_{21})})\ol{\phi}(\ol{z_{22}(\MM_2z_{21}a - z_{21}\MM_1a)})\uri{\ol{\phi}(\ol{\MM_1az_{21}az_{21}})}.$$

\uri{Denoting $N_E:=N_1 \times N_2$ and $N_O:=N_{12} \times N_{21}$ we now note that if $\beta(h_1,h_2)=1$ if $h_1\in N_E$ and $h_2\in N_O$ or vice versa. This follows from the explicit form \uri{\eqref{eqn:beta.explicit}} combined with fact that $\xi N_E \subset N_E$ and  $N_E N_O \subset N_O$. That implies that $N_E$ is orthogonal to $N_O$ with respect to~$\langle \cdot, \cdot \rangle_\beta$ and therefore the conditions of Lemma~\ref{lem:splittingN} are satisfied. It also implies that }
$$T_{[u,h]}e_{12} = T_{f(h)}e_{12}$$ 
where $$f(h) = \big(1+\pi^k\begin{pmatrix} az_{21} & 0 \\ 0 & -z_{21}a\end{pmatrix}\big)K^{k+1}.$$
We calculate 
\begin{equation}\label{eq:reducing.beta}
\begin{split}  
\sum_{h\in N}\beta([u,h],h^{-1})&T_{[u,h]}e_{12}  \\ 
=\sum_{\substack{z_{11},z_{12},\\z_{21},z_{22}}}&\ol{\phi}(\ol{z_{11}(az_{21}\MM_1-\MM_1az_{21})})\ol{\phi}(\ol{z_{22}(\MM_2z_{21}a - z_{21}\MM_1a)})\uri{\ol{\phi}(\ol{\MM_1az_{21}az_{21}})}T_{f(h)}e_{12}.
\end{split}
\end{equation}
Notice that only $z_{21}$ appears in $f(h)$. We consider now the above sum where $z_{21}$ is fixed and we sum over $z_{11}$ and $z_{22}$. 
Since $\cZ$ is an $\cA$-bimodule the element $az_{21}\MM_1-\MM_1az_{21}$ lies in $\cZ$.
If the reduction of this element mod~$\pi$ is not trivial, the summation  
in \eqref{eq:reducing.beta} will vanish. This follows from the fact that the pairing $\ol{\cZ}\times \ol{\cZ}\to \C^{\times}$ \uri{defined by} $(\ol{\aa},\ol{\bb})\mapsto \ol{\phi}(\ol{\MM \aa \bb})$ is non-degenerate \uri{since} if $\ol{\aa}$ is in the radical of the form, then $\ol{\MM\aa}$ is in $\ol{\cZ}^{\perp}=\ol{\cA}$. But since $\ol{\MM}\in \cA$ is invertible and $\cA$ is an algebra, this implies that $\aa\in \ol{\cA}\cap \ol{\cZ}=0$. A similar argument holds for $\bb$. It follows that $\ol{az_{21}}\in \ol{\cZ}$ commutes with $\ol{\MM_1}$. But this is possible only if $\ol{\aa z_{21}} = 0$.

Considering now the other instance of $\ol{\phi}$ in \eqref{eq:reducing.beta} we similarly get that $\ol{\MM_2z_{21}\aa} = \ol{z_{21}\MM_1\aa}$. Using the fact that $\left(\begin{smallmatrix} 1 & \aa \\ 0 & 1\end{smallmatrix}\right)\in U'(\MM)$ we get that $\ol{\MM_1\aa} = \ol{\aa\MM_2}$. This implies as before that $\ol{z_{21}\aa}$ commutes with $\ol{\MM_2}$, and that $\ol{z_{21}\aa}=0$. 

\smallskip

Summing up this all we get that $f(h)=1$, and 
$$\TL^{\MM}_ue_{12} = c(u)e_{12}$$ for some non-zero scalar $c(u)$.

The fact that $(\TL^{\MM}_g)$ is a trivialisation of $\gamma$ implies that $c$ is a group homomorphism from $U'(\MM)$ to $\C^{\times}$.
Assume now that $g\in (\GL_{m_1}(\O_{k+1})\times \GL_{m_2}(\O_{k+1}))(\MM)$. \uri{We have 
$$c(gug^{-1})e_{12} = \TL^{{\MM}}_{gug^{-1}}e_{12} = \TL^{{\MM}}_g\TL^{{\MM}}_u(\TL^{{\MM}}_g)^{-1}e_{12} =  \TL^{{\MM}}_g \TL^{{\MM}}_ue_{12}(\TL^{{\MM}}_{g})^{-1} = c(u)e_{12}$$
so $c(u) = c(gug^{-1})$. The penultimate equality holds because such $g$ normalises $N_{12}$ and therefore commutes with $e_{12}$. Observing that $[(\GL_{m_1}(\O_{k+1})\times \GL_{m_2}(\O_{k+1}))(\MM),U'(\MM)] = U'(\MM)$, for example by considering diagonal matrices in the first group,
we get that $c$ must be the trivial homomorphism.}
\end{proof}

\begin{lemma}\label{lem:coherent.Q}
Let $N=N_1\times N_2\times N_{12}\times N_{21}$ be a decomposition of $N$ as before and let $(\TL_g^{\MM})$ be a trivialisation of $\gamma_{\MM}$. 
Then for $g_i\in \GL_{\mm_i}(\O_{k+1})(\MM_i)$ ($i=1,2$) we get unique elements $\TL^{\MM_i}_{g_i}$ such that the following equation holds:
$$\TL^{\MM}_{\diag(g_1,g_2)}e_{12}e_{21} = \TL^{\MM_1}_{g_1}\TL^{\MM_2}_{g_2}e_{12}e_{21}.$$
Moreover the elements $\TL^{\MM_i}_{g_i}$ provide a trivialisation of $\gamma_{\MM_i}$ for $i=1,2$. 
\end{lemma}
\begin{proof}
Write $g=diag(g_1,g_2)$, $N_E:=N_1\times N_2$ and $N_O:= N_{12}\times N_{21}$. 
We begin by considering the equation for $S_g$ which we write here as $\LL^{\MM}_g$ to distinguish the different elements $S_g$ arising from the different groups.  
$$\LL^{\MM}_g=\sum_{h\in N} T_{ghg^{-1}}T_h^{-1}.$$
Since $\beta(N_E,N_O)=\beta(N_O,N_E)=1$ and since $N\cong N_E\times N_O$ and $gN_ig^{-1}=N_i$ for $i \in \{O,E\}$ 
we get $$\LL^{\MM}_g = \sum_{h\in N_E}T_{ghg^{-1}}T_h^{-1}\sum_{h\in N_O}T_{ghg^{-1}}T_h^{-1}.$$
The fact that $\beta(N_1,N_2)=\beta(N_2,N_1)=1$ implies that the first sum is just the product $\LL^{\MM_1}_{g_1}\LL^{\MM_2}_{g_2}$. We write $\LL^O_g$ for the second sum.

Next, we use the fact that $N_O=N_{12}\times N_{21}$. We rewrite
\[
\begin{split}
\LL^O_g &= \sum_{h_{12}\in N_{12},h_{21}\in N_{21}}T_{gh_{12}h_{21}g^{-1}}T_{h_{12}h_{21}}^{-1} \\
&= \sum_{h_{12}\in N_{12},h_{21}\in N_{21}}T_{gh_{12}g^{-1}} \beta(gh_{12}g^{-1},gh_{21}g^{-1})^{-1}\beta(h_{12},h_{21}) T_{gh_{21}g^{-1}}T_{h_{21}}^{-1}T_{h_{12}}^{-1} \\
&=\sum_{h_{12}\in N_{12},h_{21}\in N_{21}}T_{gh_{12}g^{-1}}  T_{gh_{21}g^{-1}}T_{h_{21}}^{-1}T_{h_{12}}^{-1}
\end{split}
\]

where we have used the fact that $\beta$ is $G$-invariant. 
Since the action of the group $G$ stabilises $N_{12}$ it holds that $\LL^O_g$ commutes with $e_{12}$. We compute:
\[
\begin{split}
\LL^O_ge_{12}e_{21} &= e_{12}\LL^O_ge_{21} = e_{12}\sum_{h_{12}\in N_{12},h_{21}\in N_{21}}T_{gh_{12}g^{-1}}  T_{gh_{21}g^{-1}}T_{h_{21}}^{-1}T_{h_{12}}^{-1}e_{21} \\
&= e_{12}\sum_{h_{12}\in N_{12},h_{21}\in N_{21}}  T_{gh_{21}g^{-1}}T_{h_{21}}^{-1}T_{h_{12}}^{-1}e_{21}=  e_{12}\sum_{h_{21}\in N_{21}}  T_{gh_{21}g^{-1}}T_{h_{21}}^{-1}e_{12}e_{21}. 
\end{split}
\]
For every $h\in N_{21}$ \uri{conjugation by $T_h$ maps $e_{12}$ to another primitive idempotent} 
$$T_he_{12}T_h^{-1}=e_{12}^{\psi_h}$$
\uri{which corresponds to a character} $\psi_h\in N_{12}^{\vee}$. \uri{By the non-degeneracy of the} pairing $\langle\cdot,\cdot\rangle_{\beta}:N_{21}\times N_{12}\to \C^{\times}$
\uri{we get in this way all of the primitive idempotents} in $\C^{\beta}N_{12}=\C N_{12}$.
This means that $$e_{12}\sum_{h_{21}\in N_{21}}T_{gh_{21}g^{-1}}T_{h_{21}}^{-1}e_{12}e_{21}$$ will be non-zero exactly when $gh_{21}g^{-1}=h_{21}$ (this follows from the non-degeneracy of $\langle\cdot,\cdot\rangle_{\beta}$), and it will then be equal to $e_{12}e_{21}$. The above calculation now implies that 
$\LL^O_ge_{12}e_{21} = |C_{N_{21}}(g)|e_{12}e_{21}$.
This implies that 
\begin{equation}\label{eq:Sg.seperation}\LL^{\MM}_ge_{12}e_{21} = \LL^{\MM_1}_{g_1}\LL^{\MM_2}_{g_2}\LL^O_ge_{12}e_{21} = 
|C_{N_{21}}(g)|\LL^{\MM_1}_{g_1}\LL^{\MM_2}_{g_2}e_{12}e_{21}\end{equation}

Assume now that $(\TL^{\MM}_g)$ is a trivialisation of $\gamma$. 
By considering also the elements $diag(g_1,1)$ and $diag(1,g_2)$ and by using Equation \ref{eq:Sg.seperation} above we get unique invertible elements $\TL^{\MM_i}_{g_i}\in \C^{\beta}N_i$ which satisfy the equation:
\begin{equation}
\TL^{\MM}_ge_{12}e_{21} = 
\TL^{\MM_1}_{g_1}\TL^{\MM_2}_{g_2}e_{12}e_{21}\end{equation}

We claim that $(\TL^{\MM_i}_g)$ is a trivialisation for $\gamma_{\MM_i}$ for $i=1,2$. Indeed, equation \eqref{eq:Sg.seperation} already implies that $\TL^{\MM_i}_g$ is proportional to $\LL^{\MM_i}_g$. The fact that this is a trivialisation of $G_i$ follows from the fact that $(\TL^{\MM}_g)$ is a trivialisation of $\gamma_{\MM}$ and the following computation: if $g_1,g_1'\in G_1$ we write $g = diag(g_1,1)$ and $g'=diag(g_1',1)$ and calculate
\[
\begin{split}
\TL^{\MM}_{gg'}e_{12}e_{21} &= \TL^{\MM}_g\TL^{\MM}_{g'}e_{12}e_{21} = \TL^{\MM}_ge_{12}e_{21}\TL^{\MM}_{g'} \\
&=\TL^{\MM_1}_{g_1}e_{12}e_{21}\TL^{\MM}_{g'}=\TL^{\MM_1}_{g_1}\TL^{\MM}_{g'}e_{12}e_{21} \\ 
&= \TL^{\MM_1}_{g_1}\TL^{\MM_1}_{g'_1}e_{12}e_{21}.
\end{split}
\]

On the other hand 
$$\TL^{\MM}_{gg'}e_{12}e_{21} =\TL^{\MM_1}_{g_1g_1'}e_{12}e_{21}.$$ 
By considering again the decomposition of $\C^{\beta}N$ arising from Lemma \ref{lem:splittingN} this implies that
$$\TL^{\MM_1}_{g_1g_1'} = \TL^{\MM_1}_{g_1}\TL^{\MM_1}_{g'_1},$$
which means exactly that $(\TL^{\MM_1}_g)$ is a trivialisation for $\gamma$ on $G_1$. The proof that $(\TL^{\MM_2}_g)$ is a trivialisation for $G_2$ is the same. 
\end{proof}

We are now ready to prove that the isomorphisms of abelian groups $\GR^{\o,\l}(\ol{f})\to \GR^{\O,\l}(t-\ol{x})$ are compatible with the multiplication. We do this in two steps. In the first step we consider only the orbits of $\MM_1,\MM_2$ and $\MM=\MM_1\oplus\MM_2$. In the second step we will consider all orbits in $[\Mat^k(\ol{f})]$. 
\begin{proposition}\label{prop:local.odd.mult}
Assume that $\MM=\MM_1\oplus\MM_2$ and all the groups are as before. 
Assume also that the trivializations $(\TL^{(\mm_1+\mm_2)}_g), (\TL^{\MM_1}_g), (\TL^{\MM_2}_g)$ arising from the categorical equivalences
$$\Rep^{\beta}(G)\cong \Rep^{\beta\gamma^{-1}}(GN), \Rep^{\beta}(G_1)\cong \Rep^{\beta\gamma^{-1}}(G_1N_1),\text{ and } \Rep^{\beta}(G_2)\cong \Rep^{\beta\gamma^{-1}}(G_2N_2)$$
satisfy the conclusion of Lemma \ref{lem:coherent.Q}. Then the equivalences $$\Rep(\GL_{md}(\o_l))_{\MM}\cong \Rep(\GL_{m}(\O_l)_{\MM}), 
\Rep(\GL_{\mm_1d}(\o_\l))_{\MM_1}\cong \Rep(\GL_{\mm_1}(\O_\l)_{\MM_1}) \text{ and } $$ 
$$\Rep(\GL_{\mm_2d}(\o_\l))_{\MM_2}\cong \Rep(\GL_{\mm_2}(\O_\l)_{\MM_2})$$
are compatible with the product $\nmult$ in the sense that the diagram 
$$\xymatrix{
\GR^{\o,\l}_{\MM_1}\ot  \GR^{\o,\l}_{\MM_2}\ar[rr]^{\qquad\nmult}\ar[d] & & \GR^{\o,\l}_{\MM}\ar[d] \\ 
\GR^{\O,\l}_{\MM_1}\ot  \GR^{\O,\l}_{\MM_2}\ar[rr]^{\qquad \nmult}& & \GR^{\O,\l}_{\MM}
}
$$
commutes.
\end{proposition}
\begin{proof}
The equivalences between the categories is given as the composition of the equivalences 
$$\Rep(\GL_{md}(\o_\l))_{\MM}\cong \Rep(\GL_{md}(\o_\l)(\MM))_{\MM}\cong \Rep^{\beta}(\GL_{md}(\o_{k+1})(\MM))_{\MM}\cong $$
$$\Rep^{\beta}(\GL_{m}(\O_{k+1})(\MM))_{\MM}\cong 
\Rep(\GL_{m}(\O_\l)(\MM))_{\MM}\cong 
\Rep(\GL_{m}(\O_\l))_{\MM} $$
and similarly for the equivalence for the groups $G_1$ and $G_2$.
All the equivalences except for the middle arise from Clifford Theory and are compatible with $\nmult$ \uri{by \cite[Theorem 3.6]{CMO1}}. We will prove here that the middle equivalence, for projective representations, is compatible with $\nmult$ as well; \uri{see \cite[Theorem 3.14]{CMO1} for a comprehensive treatment of the induction functors in the context of projective representations}.

Let then $W_i\in \Rep^{\beta}(\GL_{\mm_i}(\O_{k+1})(\MM_i))_{\MM_i}$ be a projective representation for $i=1,2$. We have
$$W_1\nmult W_2 = \C^{\beta} \GL_{m}(\O_{k+1})(\MM)e_{U'(\MM)}e_{V'(\MM)}\ot_{G_1\times G_2}(W_1\boxtimes W_2).$$
The corresponding representation of 
$\C^{\beta} \GL_{md}(\o_{k+1})(\MM)$ is then given by 
\begin{equation}\big(\C^{\beta} \GL_{m}(\O_{k+1})(\MM)e_{U'(\MM)}e_{V'(\MM)}\ot_{G_1\times G_2}(W_1\boxtimes W_2)\big)\ot I_{\MM}\end{equation}
%\footnote{We have some issues with the terminology of $U(\MM)$, $\wt{U(\MM)}$ and $U'(\MM)$ on the level of choosing consistent notations which will not be too ugly} 
where as before $I_{\MM}$ is the simple representation of $\C^{\beta}N$. The action of $g\in G$ is given by the action with $T_g\ot \TL^{\MM}_g$ and the action of $h\in N$ is given by $1\ot T_h$. 

On the other hand, if we first apply the categorical equivalence to $W_1$ and $W_2$ we get the representations 
$$W_1\ot I_{\MM_1}\text{ and } W_2\ot I_{\MM_2}$$
where the action of $g\in G_i\subseteq G_iN_i$ on $W_i\ot I_{\MM_i}$ is given by $T_g\ot \TL^{\MM_i}_g$ and the action of $h\in N_i\subseteq G_iN_i$ is given by $1\ot T_h$. Applying $\nmult$ we get 
$$(W_1\ot I_{\MM_1})\nmult (W_2\ot I_{\MM_2}) = $$
\begin{equation}\label{eq:secondrep}\C^{\beta}GN e_{U'(\MM)N_{12}}e_{V'(\MM)N_{21}}\ot_{(G_1\times G_2)\cdot (N_1\times N_2)}(W_1\ot I_{\MM_1})\boxtimes (W_2\ot I_{\MM_2}).\end{equation}
Where we have used the fact that $U(\MM)=U'(\MM)N_{12}$ and $V(\MM)=V'(\MM)N_{21}$. 
We next use the isomorphism of rings $\C^{\beta}GN\cong \C^{\beta}G\ot \C^{\beta}N$. Under this isomorphism the element 
$$e_{U'(\MM)N_{12}}e_{V'(\MM)N_{21}} = e_{U'(\MM)}e_{12}e_{V'(\MM)}e_{21}$$ 
where $e_{12}$ is the idempotent which corresponds to $N_{12}$ and similarly for $e_{21}$ becomes
$$\sum_{u\in U'(\MM), v\in V'(\MM)}
T_uT_v\ot \TL^{\MM}_u e_{12}\TL^{\MM}_ve_{21} = e_{U'(\MM)}e_{V'(\MM)}\ot e_{12}e_{21}$$ by Lemma \ref{lem:reducingelements}.
The representation \ref{eq:secondrep} then gives the $\C^{\beta}G\ot \C^{\beta}N$ representation 
\begin{equation}\label{eq:thirdrepresentation}\big(\C^{\beta}Ge_{U'(\MM)}e_{V'(\MM)}\ot \C^{\beta}Ne_{12}e_{21}\big)\ot_{(G_1\times G_2)\cdot (N_1\times N_2)} (W_1\ot I_{\MM_1})\boxtimes ( W_2\ot I_{\MM_2}).\end{equation}

We fix an isomorphism $$\Psi:\C^{\beta}Ne_{12}e_{21}\ot_{(N_1\times N_2)}(I_{\MM_1}\boxtimes I_{\MM_2})\cong I_{\MM}$$
of $\C^{\beta}N$ representations. \uri{The existence of such an isomorphism follows, for example, from dimension count. Indeed, $\dim(I_{\MM_i})=\sqrt{|N_i|}$ and the dimension of the representation $\C^{\beta}Ne_{21}\ot_{(N_1\times N_2)}(I_{\MM_1}\boxtimes I_{\MM_2})$ is therefore $|N_{12}|\sqrt{|N_1N_2|}=\sqrt{|N|}=\dim(I_\MM)$, as $I_\MM$ is the unique irreducible representation of $\C^{\beta}N$. Finally, $\C^{\beta}Ne_{12}e_{21}\ot_{(N_1\times N_2)}(I_{\MM_1}\boxtimes I_{\MM_2})$ is then a non-zero subrepresentation, which must be equal to it. }

We rewrite Equation \eqref{eq:thirdrepresentation} in the following way:
we will write $(-\ot - )_{(G_1\times G_2)\cdot (N_1\times N_2)}$, the coinvariants with respect to the action of the group $(G_1\times G_2)\cdot (N_1\times N_2)$, instead of $- \ot_{(G_1\times G_2)\cdot (N_1\times N_2)}- $. This will make it easier to analyse the representation. Using the isomorphism $\Psi$ we get
\[
\begin{split}
\big(\C^{\beta}Ge_{U'(\MM)}e_{V'(\MM)}&\ot \C^{\beta}Ne_{12}e_{21}\big)\ot_{(G_1\times G_2)\cdot (N_1\times N_2)}(W_1\ot I_{\MM_1})\boxtimes ( W_2\ot I_{\MM_2}) \\
&\cong\big((\C^{\beta} Ge_{U'(\MM)}e_{V'(\MM)}\ot (W_1\boxtimes W_2))\ot (\C^{\beta}N e_{12}e_{21}\ot (I_{\MM_1}\boxtimes I_{\MM_2}))\big)_{(G_1\times G_2)\cdot (N_1\times N_2)}\\
&\cong \big((\C^{\beta} Ge_{U'(\MM)}e_{V'(\MM)}\ot (W_1\boxtimes W_2))\ot (\C^{\beta}N e_{12}e_{21}\ot_{(N_1\times N_2)} (I_{\MM_1}\boxtimes I_{\MM_2}))\big)_{(G_1\times G_2)}.
\end{split}
\]
The action of $g=diag(g_1,g_2)\in G_1\times G_2$ on the last vector space is given by 
\[
\begin{split}
g\cdot (T_{g'}&e_{U'(\MM)}e_{V'(\MM)}\ot w_1\ot w_2\ot T_h e_{12}e_{21}\ot v_1\ot v_2) \\
&=(T_{g'}T_g^{-1}e_{U'(\MM)}e_{V'(\MM)}\ot T_{g_1}\cdot w_1\ot T_{g_2}\cdot w_2\ot T_h e_{12}e_{21}(\TL^{\MM}_g)^{-1}\ot \TL^{\MM_1}_{g_1}v_1\ot \TL^{\MM_2}_{g_2}v_2) \\ 
&=(T_{g'}T_g^{-1}e_{U'(\MM)}e_{V'(\MM)}\ot T_{g_1}\cdot w_1\ot T_{g_2}\cdot w_2\ot T_h e_{12}e_{21}(\TL^{\MM_1}_{g_1})^{-1}(\TL^{\MM_2}_{g_2})^{-1}\ot \TL^{\MM_1}_{g_1}v_1\ot \TL^{\MM_2}_{g_2}v_2)\\ 
&=(T_{g'}T_g^{-1}e_{U'(\MM)}e_{V'(\MM)}\ot T_{g_1}\cdot w_1\ot T_{g_2}\cdot w_2\ot T_h e_{12}e_{21}\ot v_1\ot v_2),
\end{split}
\]
where we have used Lemma \ref{lem:coherent.Q} to relate $\TL_g^{\MM}$ and $\TL_{g_i}^{\MM_i}$, the fact that $\TL^{\MM_i}_{g_i}\in \C^{\beta}N_i$ and the  relevant tensor product is over $N_1\times N_2$. 
This shows that the above representation is in fact isomorphic with
$$\big(\C^{\beta} G e_{U'(\MM)}e_{V'(\MM)}\ot_{G_1\times G_2} (W_1\boxtimes W_2)\big)\ot \big(\C^{\beta}Ne_{12}e_{21}\ot_{N_1\times N_2}(I_{\MM_1}\boxtimes I_{\MM_2})\big).$$ 
Using the isomorphism $\Psi$ above, this representation is isomorphic to 
$$\big(\C^{\beta} G e_{U'(\MM)}e_{V'(\MM)}\ot_{G_1\times G_2} (W_1\boxtimes W_2)\big)\ot I_{\MM}$$ as desired.  
\end{proof}

The last thing we need to do is to show that we can choose all the isomorphism $\GR^{\O,l}_{[\MM]}\cong \GR^{\o,l}_{[\MM]}$ in a way which will give us a ring homomorphism. We know this is locally true, for a particular triple $(\MM_1,\MM_2,\MM=\MM_1\oplus\MM_2)$ and we want to show that we can also make a global choice of trivialisations which will give us a global ring isomorphism. For this we will use Tychonoff's Theorem.

For every $[\MM]\in[\Mat^k(\ol{f})]$ we write $$X_{\MM}=\{(\TL^{\MM}_g)|(\TL^{\MM}_g)\text{ is a trivialisation of }\gamma_{\MM}\}.$$ 
Since the trivialisations form a torsor over the finite group $Hom_{\Z}(G(\MM),\C^{\times})$ this set is finite. Let $$X=\prod_{\MM} X_{\MM}.$$ We will call an element $x=(Q^{\MM}_g)_{\MM}$ \textit{coherent} if whenever $\MM=\MM_1\oplus \MM_2$ and $g=diag(g_1,g_2)$ for $g_i\in G(\MM_i)$ the conclusion of Lemma \ref{lem:coherent.Q} holds, that is $$\TL^{\MM}_ge_{12}e_{21}= \TL^{\MM_1}_{g_1}\TL^{\MM_2}_{g_2}e_{12}e_{21}$$ in the twisted group algebra $\C^{\beta}N$. By Proposition \ref{prop:local.odd.mult} a  coherent element in $X$ will give an isomorphism of rings as desired.

We thus want to prove that $X$ contains a coherent element. For this, consider $X$ as a topological space with the product topology, where the $X_{\MM}$ are finite discrete topological spaces. The space $X$ is then compact by Tychonoff's Theorem, as a product of compact spaces. 

Since $[\Mat^k(\ol{f})]$ is a countable union of finite sets it is countable. Choose an enumeration $([\MM_i])$ of $[\Mat_k(\ol{f})]$, and define $$\mu_i = \bigoplus_{j=1}^i\MM_j.$$ For every $i$ choose a trivialisation $(\TL^{\mu_i}_g)$ of $\gamma_{\mu_i}$. By Lemma \ref{lem:coherent.Q} we get a trivialisation $(\TL^{i,j}_g)$ for $\gamma_{\MM_j}$ for every~{$j\leq i$}.

Choose now a sequence of elements $(x_i)$ in $X$ which satisfy the following condition:
If $j\leq i$ then $(x_i)_{\MM_j} = (\TL^{i,j}_g)$.
The compactness of $X$ implies that $(x_i)$ has a convergent subsequence $x_{i_k}$. Write $x=\lim_k x_{i_k}$. We claim that $x$ is a coherent element of $X$. 
Indeed, Assume that $\MM_i\oplus\MM_j=\MM_k$ for some $i,j,k\in \N$. Then since all the spaces $X_{\MM}$ are finite we get that 
$x_{\MM_i} = (x_{i_{r}})_{\MM_i}, x_{\MM_j} = (x_{i_{r}})_{\MM_j}$ and 
$x_{\MM_k} = (x_{i_{r}})_{\MM_k}$ for $r \uri{\gg}0$.
For $r$ \uri{large} enough we get that $i,j,k< i_r$. But for \uri{such $r$,  
$\mu_{i_r}$} will contain the two identical matrices $\MM_i\oplus \MM_j$ and $\MM_k$ as direct summands. By considering the action of the relevant Weyl group we get that the resulting trivialisations for $\gamma_{\MM_i\oplus \MM_j}$ and for $\gamma_{\MM_k}$ are equal. But the trivialisation we get for $\gamma_{\MM_i\oplus \MM_j}$ is compatible with the trivialisations we get for $\gamma_{\MM_i}$ and $\gamma_{\MM_j}$ in the sense of Lemma \ref{lem:coherent.Q}. This concludes the proof of Theorem~\ref{oddhomomorphism} and hence of Theorem~\ref{thm:base.change.iso}. 

\end{subsection}

%%%%%%%%%%%%%%%%%%%%%%%%%%%%%%%%%%%%%

\section{The principal series and proofs of theorems  \ref{thm:atom.psh.subalgebra}   and  \ref{thm:large.psh.subalgebra}}\label{sec:principal.series} 
 
\subsection{The principal series of $\GL_n(\o)$}  Essential for the multiplication and comultiplication on $\GR^\o$ to commute, thereby endowing $\GS^\o$ with a structure of a Hopf algebra, is the existence of a Mackey formula. \uri{A general Mackey formula should involve two \lq Levi subgroups\rq~$L_1$ and $L_2$ and a combinatorial formula of the form        
\begin{equation}\label{eq:Mackey_intro}
\pRes^G_{L_2} \circ \pInd_{L_1}^G \cong \bigoplus_{w\in W_{L_2}\backslash W_G / W_{L_1}} \pInd_{L_2\cap wL_2w^{-1}}^{L_2} \circ \Ad_w \circ \pRes^{L_1}_{w^{-1}L_1w\cap L_1}
\end{equation}}
keeping track of the commutation of induction and restriction functors (in a broad sense), \uri{$\pInd$ and $\pRes$},  through the combinatorics of \uri{appropriate \lq Weyl groups\rq~$W$}. See \cite[\S5]{DM} for the precise general formulation and proof for the case of Harish-Chandra functors associated with reductive groups over finite fields. The existence of such formula is a key for the development of the theory, in particular for the existence of a PSH-algebra structure; see \cite{Zelevinsky}, in particular Appendix~3. We remark that a naive generalisation of Harish-Chandra functors to the context of rings, namely taking only one idempotent rather than two, does not lead to any tractable relation in between induction and restriction. This can be seen by the growing complexity of the permutation representation of $\GL_n(\o)$ over the flag variety; see \cite{Casselman} for $n=2$ and \cite{OnnSingla} for $n=3$. 

While we do not know whether a Mackey formula of the sort \eqref{eq:Mackey_intro} holds in general for $\GL_n(\o)$, it does hold for the \lq maximal torus\rq~as we now explain.  Let $T \cong \GL_1(\o_\l)^n$ denote the group of diagonal matrices in $G=\GL_n(\o_\l)$. Let $U$ denote the subgroup of upper unitriangular matrices in $\GL_n(\o_\l)$ and let $V$ be its transpose. Let $\pind=\C G e_ue_V\otimes_T:\Rep(T) \to \Rep(G)$ and $\pres = e_Ue_V\C G \ot_G:\Rep(G)\to \Rep(T)$. 

\begin{theorem}[Decomposition of the principal series \cite{CMO2}]\quad
\begin{enumerate}
\item There is a natural isomorphism $\pres \pind \cong \sum_{\sigma \in S_n} \Ad(\sigma)$.
\item For each linear character \uri{$\chi=(\chi_1,\ldots,\chi_n)$} of $T$ there is an isomorphism of algebras $\End_G(\pind \chi)\simeq \C \mathrm{Stab}_{S_n}(\chi)$. In particular, if $\chi={\bf 1}$ is the trivial representation we have $\End_G\left(\pind({\bf 1})\right) \cong \C S_n$. 
\end{enumerate}
\end{theorem} 

The isomorphism in part (2) of the theorem coincides with the Hecke algebra isomorphism associated to the flag representation of $\GL_n(\F_q)$ and should be contrasted with the highly complicated decomposition of the flag representation of $\GL_3(\o_\l)$; see \cite[\S6.3]{OnnSingla}.

\subsection{Proof of Theorem~\ref{thm:atom.psh.subalgebra}} 
{Let $\rho$ be a strongly cuspidal representation of level $\l$. In case $\l=0$ the result follows from Sections 2 and 9 in \cite{Zelevinsky}. Assume that $\l>0$. The restriction of $\rho$ to $K^{\l}/K^{\l+1} \cong \Mat_n(\kk)$ consists of irreducible matrices with common irreducible characteristic polynomial $\ol{f}$ and therefore $\GS(\rho)\subseteq \GR^{\o,\l}(\ol{f})$.

By Theorem \ref{thm:base.change.iso} we have an isomorphism of rings and corings $$\GR^{\o,\l}(\ol{f})\cong \GR^{\O,\l}(t-\ol{x})$$ where $\O=\o[x]$, and the minimal polynomial of $x$ over $\o$ is $f$. Under this isomorphism $\rho$ corresponds to a representation of $\GL_1(\O_\l)$ whose restriction to $K^{\l}/K^{\l+1}\cong \Mat_1(\O_1)$ is the $1\times 1$ matrix $(\ol{x})$; see the discussion following the statement of Theorem~\ref{thm:atom.psh.subalgebra} in the introduction. Notice that \uri{now} $K^{\l}$ and $K^{\l+1}$ are the congruence subgroups inside \uri{$\GL_m(\O)$ rather than $\GL_n(\o)$}.

We now proceed by induction on $\l$. 
Translation by the one-dimensional character which corresponds to the image of $x$ in $\O_\l$ gives an isomorphism $\GR^{\O,\l}(t-\ol{x})\cong \GR^{\O,\l}(t)$. Under this isomorphism the image of~$\rho$ is a representation of $\GL_1(\O_\l)$ which is trivial on the $\l$-th congruence subgroup. We can thus consider it as a representation $\wt{\rho}$ of $\GL_1(\O_{\l-1})$. The original algebra generated by $\rho$ is a PSH-algebra if and only if the algebra generated by this representation of $\GL_1(\O_{\l-1})$ is a PSH-algebra because translation by one-dimensional character is compatible with multiplication and comultiplication. We can proceed inductively to translate by one-dimensional characters until we reach the level $\l=0$. But then the claim is clear, because all representations of $\GL_1(\kk)$ are cuspidal for any finite field $\kk$.  \hfill \qed
}
\subsection{Proof of Theorem~\ref{thm:large.psh.subalgebra}}
{
We know that if $\rho$ is a strongly cuspidal representation then $\GS(\rho)$ is a PSH algebra. We just need to make sure that the Hopf axiom is still valid when we multiply representations from different~$\GS(\rho)$s.
In other words, we need to show that if we have $\nu_i\in \GS(\rho_i)$ for $i=1,\ldots,n$ then $$\Delta(\nu_1\circ\cdots\circ \nu_n) = \Delta(\nu_1)\circ\cdots\circ\Delta(\nu_n).$$
By Corollary \ref{cor:hopf.axiom.primary} we can reduce to the case where all the representations $\rho_i$ are of the same level $\l$ and are all contained in $\GR^{\o,\l}(\ol{f})$ for some irreducible polynomial $\ol{f}\in\o_1[t]$. Using the base change isomorphism from Theorem \ref{thm:base.change.iso} we can further assume, by passing to the relevant extension, that this polynomial has degree 1. By translating by a one-dimensional character we can assume that this polynomial is $t$, and therefore that the restriction of all $\nu_i$ to the $\l$-th congruence subgroups is trivial. We can thus consider all the $\nu_i$ as (virtual) representations in $\GR^{\o,\l-1}$. We can now continue to reduce $\l$ inductively using the isomorphism of Theorem \ref{thm:base.change.iso} and translation by a one-dimensional character. For $\l=0$ we know that the theorem is true by Section 9 in~\cite{Zelevinsky}, so we are done.} \hfill \qed

%%%%%%%%%%%%%%%%%%%%%%%%%%%%%%%%%%%%%

\section{Some nilpotent representations}\label{sec:nilpotent}

%%%%%%%%%%%%%%%%%%%
%                                                       %
%   Section                                        %
%                                                       %   
%%%%%%%%%%%%%%%%%%%

Corollary \ref{cor.reduction.nilpotent} reduces the study of $\GR^{\o,\l}(\ol{f})$ to the nilpotent subalgebra $\GR^{\O,\l}(t)$.
In this section we consider the case $\l=2$, 
and focus on representations whose associated (nilpotent) matrix
has only one non-zero Jordan block. We thus fix an integer $r \ge 2$ and consider the $r \times r$ matrix $\eta$ given by \uri{
\[
\eta = E_{1,2} + E_{2,3} + \cdots + E_{r-1,r}
\]}
where $E_{i,j}$ is the matrix with $1$ in the $(i,j)$ entry and zero otherwise. 
We study the subgroup $\GN_\eta=\GN^{\o,2}_\eta$ of $\GR^{\o,2}(t)$ 
spanned by all irreducible representations whose associated matrix is of the form $\eta_n:=\eta \oplus 0_{n-r}$ for some~$n$. \uri{The inner products on the corresponding $\Z$-modules are the restrictions of \eqref{def.inner.prod} and \eqref{def.inner.prod.tensor}}.
The map that assigns to an irreducible representation the orbit of its associated matrix commutes with the multiplication~$\nmult$,
and we get that $\GN_\eta \nmult \GR^{\o,1} \subseteq \GN_\eta$ and that $\Delta(\GN_\eta)\subseteq R^{\o,1}\ot \GN_\eta + \GN_\eta \ot \GR^{\o,1}$.
The second statement follows from the fact that the matrix $\eta$ is indecomposable. 

We write $\kk$ for the residue field $\o_1$. We also omit the ring $\o$ and simply write $\GR^\ell=\GR^{\o,\l}$. First we determine the structure of the centraliser of $\eta_n$ inside $\GL_n(\kk)$.
We then use Clifford Theory to study its representations, and finally we will explain the relation with the multiplication $\nmult$.

\subsection{The centraliser of $\eta_n$}
The following lemma follows easily from linear algebra.

\begin{lemma}
	The centraliser of $\eta_n$ in $\Mat_n(\kk)$ is isomorphic to the endomorphism ring of the $\kk[t]$-module $\kk[t]/t^r \oplus \kk^{n-r}$. It consists of all matrices of the form
	\begin{equation}\label{eqone}
		M = 
		\begin{pmatrix}
			M_{11} & M_{12} \\
			M_{21} & M_{22} 
		\end{pmatrix}
		=
		\begin{pmatrix}
			 a_0 & a_1 & \ldots & a_{r-1} &  b_1 & \ldots & b_{n-r} \\
			0 & a_0 & \ldots & a_{r-2} &  0 & \ldots & 0 \\
			 \vdots & \vdots & \ddots &  \vdots & \vdots\ & \ddots & \vdots \\
			 0 & 0 & \ldots & a_{0} &  0 & \ldots & 0 \\
			 0 & 0 & \ldots & c_1 &  d_{11} & \ldots & d_{1,n-r}\\
			 \vdots & \vdots & \ddots & \vdots &  \vdots\ & \ddots & \vdots \\
			 0 & 0 & \ldots & c_{n-r} &  d_{n-r,1}& \ldots & d_{n-r,n-r}
		\end{pmatrix}
	\end{equation}
where $M_{11}\in \Mat_r(\kk)$ is an element of the subalgebra of $\Mat_r(\kk)$ generated by $\eta$, 
	$M_{12}\in \Mat_{r\times n-r}(\kk)$ consists of row vectors of the form $b_1\uri{E}_{1,r+1}+b_2\uri{E}_{1,r+2}\cdots +b_{n-r}\uri{E}_{1,n}$, $M_{21}\in \Mat_{\uri{(n-r) \times}r}(\kk)$ 
	consists of column vectors of the form $c_1\uri{E}_{r+1,r}+\cdots +c_{n-r}\uri{E}_{n,r}$ and $M_{22}\in \Mat_{n-r}(\kk)$ is arbitrary. Moreover, $M$ is invertible if and only if $M_{11}$ and $M_{22}$ are invertible.  

\uri{The centraliser of $\eta_n$ is isomorphic to the automorphism group of the $\kk[t]$-module $\kk[t]/t^r \oplus \kk^{n-r}$, which is given explicitly by}
 	\[
		C_{r,n}=\begin{pmatrix}
			A_r & V\\
			V^* & \GL_{n-r}(\kk) 
		\end{pmatrix}
		=
		\begin{pmatrix}
			\GL_1(\kk[t]/t^r) & t^{r-1}\Mat_{1 \uri{\times} (n-r)}(\kk)\\
			\Mat_{(n-r) \uri{\times}1}(\kk) & \GL_{n-r}(\kk) 
		\end{pmatrix}.
	\]
\end{lemma}

\uri{We write $A=A_r=\GL_1(\kk[t]/t^r)$ and let $A^i_j$ denote the $i$-th congruence subgroup modulo the $j$-th congruence subgroup of $A$. The notation of the $\kk$-vector spaces $V$ and $V^*$ reflects the fact that they are in natural duality.}  We shall frequently move between the centraliser $\GL_n(\kk)(\eta_n)$ and its isomorphic copy $C_{r,n}$. The action of $A_r \times \GL_{n-r}(\kk)$ on $V$ and on $V^*$ is given as follows: 
an element $\sum_i a_i t^i \in A$ acts on $V$ by multiplication with the scalar $a_0$ and on $V^*$ by multiplication with $a_0^{-1}$. 
The group $\GL_{n-r}(\kk)$ acts in the standard way on $V$ and on $V^*$ via the inverse.  

It is not difficult to see that we have an Iwahori decomposition of $C_{r,n}$ given by $(V^*, A_r \times \GL_{n-r}(\kk), V)$. Here and in the sequel we slightly abuse notation and identify the vector spaces $V$ and $V^*$ with the unipotent subgroups they define. 

By Clifford Theory we know that irreducible representations of $\GL_n(\o_2)$ that lie over the orbit $[\eta_n]$ are in bijection   
with representations of $\C ^{\beta}C_{r,n}$ for some cocycle $\beta$. A direct calculation \uri{similar to Lemma~\ref{lem:choosebeta}} shows that we can choose $\beta$ so that $\beta(c,d) = \ol{\phi}\left(\eta_n c d\right)$.
Since $\eta_n$ is nilpotent and commutes with~$c$ and with~$d$, we get that the matrix $\eta_ncd$ is nilpotent as well, and therefore $\beta$ is trivial.
This fact generalises easily to all nilpotent matrices.
Therefore, representations of $\GL_n(\o_2)$ lying over the character ${\eta_n}$
are in bijection with representations of $C_{r,n}$. 
%\end{subsection}

\subsection{Analysis of the irreducible representations of $C_{r,n}$ by Clifford Theory}
We will now write $C_{r,n}$ as an extension and use Clifford Theory repeatedly to classify its irreducible representations.
We then proceed to see how this interacts with the multiplication. The next lemma is straight forward. 

\begin{lemma}
	The group $C_{r,n}$ fits into the following short exact sequence of groups
	$$ 1\rightarrow H\rightarrow C_{r,n}\rightarrow A_{r-1} \times \GL_{n-r}(\kk) \rightarrow 1$$
	where $H$ is a generalised Heisenberg group given by $H=V \cdot V^* \cdot A_r^{r-1}$,
	and $A_r^{r-1} \simeq (\kk,+)$ is the subgroup of $A_r$ consisting of all elements of the form $1+a t^{r-1}$ for some $a\in \kk$.
\end{lemma}
%\footnote{UO: $A_l$ is now $A^{r-1}$, . }

We recall the construction of the irreducible representations of $H$.
The group $H$ is a Heisenberg group that fits into the central extension 
\[
1\rightarrow A_r^{r-1} \rightarrow H\rightarrow V\times V^*\rightarrow 1.
\]

Let $W$ be an irreducible representation of $H$. 
Assume first that the restriction of $W$ to $A_r^{r-1}$ is given by a nontrivial character $x$ 
(where $x\in \kk^\times$, and we identify $\kk$ and its dual group as in \S\ref{subsec:Pontryagin.Dual}).
That already determines $W$ uniquely up to an isomorphism.
Indeed, it is easy to see that $W$ has a basis $\{w_v\}_{v\in V}$, the action of $V$ is given by 
$v_1\cdot w_{v_2}= w_{v_1+v_2}$ and the action of $f \in V^*$ is given by $f\cdot w_v = \ol{\phi}({x\langle v,f \rangle)}w_v$.
Such a representation, which we denote by $W_x$, has dimension $|V|=q^{n-r}$, where $q$ is the cardinality of $\kk$. 
If the restriction of $W$ to $A_r^{r-1}$ is trivial, then $W$ is a pull-back of a representation of the abelian group $V\times V^*$.
The dual group of this group is $V^*\times V$, and therefore such a representation is one-dimensional and is given by a pair $(f,v)\in V^*\times V$.
We denote the corresponding representation of $H$ by $W_{(f,v)}$. Thus
\begin{equation}\label{irr.of.H}
\Irr(H)=\left\{W_x \mid x \in \kk^\times \right\} \coprod \left\{ W_{(f,v)} \mid (f,v) \in V^* \times V \right\}.
\end{equation}

We now determine the orbits of the quotient $C_{r,n}/H \simeq A_{r-1}\times \GL_{n-r}(\kk)$ on $\Irr(H)$. The action of $C_{r,n}/H$ on $\Irr(H)$ is as follows. On $W_x$ the action is trivial. The action of $(a,g) \in A_{r-1} \times \GL_{n-r}(\kk)$ on $W_{(f,v)}$ is given by
\begin{equation}\label{action.Crn.V}
(a,g)\cdot W_{(f,v)}= W_{(a_0^{-1} g\cdot f,  v \cdot a_0 g^{-1})},
\end{equation}
where $a_0$ is the constant term of $a$.

\uri{Before stating the next lemma we require some notation. We let $e_i=(0,\ldots,,0,1,0,\ldots,0) \in \Mat_{1 \times (n-r)}(\kk)$ be the all zero vector except $1$ in the $i$-th place, $ e^i = e_i^T \in \Mat_{(n-r) \times 1}(\kk)$, and let
	\begin{equation}\label{def.of.C12n-r}
	C^1_{2,n-r}=\begin{pmatrix}
		1 & * & * &\ldots & *\\
		0 & 1 & 0 &\ldots & 0\\
		0 & * & * &\ldots & *\\
		\vdots &\vdots &\vdots &\ddots &\vdots &\\
		0 & * & * &\ldots & *
	\end{pmatrix} \subset \GL_{n-r}(\kk),
\end{equation}
be the subgroup of $C_{2,n-r}$ consisting of all elements in $C_{2,n-r}$ in which $a_0=1$.}

\begin{lemma}\label{C.orbits.on.H}
	Using the parametrisation \eqref{irr.of.H}, the orbits of $C_{r,n}/H$  on $\Irr(H)$ are given by 
	\begin{enumerate}
	\item The singleton $\{W_{(0,0)}=\mathbf{1}_H\}$ with stabiliser $A_{r-1} \times \GL_{n-r}(\kk)$.
	
	\item Singletons $\{W_x \}$ for $x \in \kk^\times$ with stabiliser $A_{r-1} \times \GL_{n-r}(\kk)$.
 	
	\item $\{ W_{(f,v)} \mid f(v)=y\}$ for every $y \in \kk^\times$, represented by $W_{(e^1,ye_1)}$ whose stabiliser is  $A_{r-1} \times \GL_{n-r-1}(\kk)$.

	\item  $\{ W_{(f,0)} \mid 0 \ne f \in V^*\}$ represented by $W_{(e^1,0)}$, and $\{ W_{(0,v)} \mid 0 \ne v \in V\}$ represented by $W_{(0,e_1)}$. 

	\noindent The stabiliser in both cases is $A_{r-1} \times \left(\GL_{n-r-1}(\kk) \ltimes \kk^{n-r-1}\right)$.
	\item $\{ W_{(f,v)} \mid 0 \ne f \in V^*, 0 \ne v \in V, f(v)=0\}$ represented by $W_{(e^2,e_1)}$ whose stabiliser is $A_{r-1} \times C^1_{2,n-2}$. 
	 
	\end{enumerate}
\end{lemma}

\begin{proof} The classification of the orbits is immediate from \eqref{action.Crn.V}. For the stabilisers we use the explicit form~\eqref{eqone}. The action of the quotient $C_{r,n}/H$ factors through $\GL_1(\kk) \times \GL_{n-r}(\kk)$ and the relation between the components $a_0$ and $g$ is dictated by \eqref{action.Crn.V} and the stabiliser is isomorphic to $A_{r-1}\times C$, where $C$
is the stabiliser inside $\GL_{n-r}(\kk)$; we need to choose a splitting where $a_0 \in \GL_1(\kk)$ maps to the scalars in $\GL_n(\kk)$ to get this isomorphism. The group $C$ for orbits of type (3) consists of elements of the form 

\begin{equation}
	\begin{pmatrix}
		1 & 0 & 0 &\ldots & 0\\
		* & * & * &\ldots & *\\
		\vdots &\vdots &\vdots &\ddots &\vdots &\\
		* & * & * &\ldots & *
	\end{pmatrix} \subset \GL_{n-r}(\kk),
\end{equation}
for $W_{(e^1,0)}$ which is a semidirect product $\GL_{n-r-1}(\kk)\ltimes \kk^{n-r-1}$; similarly we get the transpose for  $W_{(0,e_1)}$. 

\smallskip

For orbit of type (4) the stabiliser within $\GL_{n-r}(\kk)$ takes the form  
\begin{equation}
	\begin{pmatrix}
		1 & 0 & 0 &\ldots & 0\\
		0 & * & * &\ldots & *\\
		\vdots &\vdots &\vdots &\ddots &\vdots &\\
		0 & * & * &\ldots & *
	\end{pmatrix} \subset \GL_{n-r}(\kk),
\end{equation}
hence isomorphic to $\GL_{n-r-1}(\kk)$. Finally, the stabiliser of $(e^2,e_1)$ is the \uri{subgroup $C^1_{2,n-2}$ defined in \eqref{def.of.C12n-r}}. 
\end{proof}

We now turn to the classification of irreducible representations of $C_{r,n}$. Consider the following directed graph in Figure~\ref{fig:graph}.
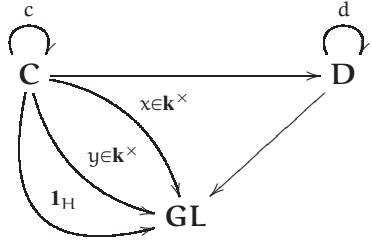
\begin{figure}[htb!]
 \centering
 \caption{Branching graph $\Gamma$}
 \label{fig:graph}
 \begin{displaymath}
   \xymatrix@+15pt{  &   \boldsymbol{C} \ar@(ul,ur)[]^c \ar[rr]              \ar@/^1.2pc/[rd]^{x \in \kk^\times}     \ar@/_1.2pc/[rd]^{y \in \kk^\times}    \ar@/_3pc/[rd]^{\mathbf{1}_H}     & & \boldsymbol{D} \ar@(ul,ur)^{d}   \ar[ld] \\
      & &  \boldsymbol{GL}     &}
   \end{displaymath}
\end{figure}
We claim that the irreducible representations of $C_{r,n}$ can be parametrised  by triples $(\gamma, \psi, Y)_{r,n}$, where $\gamma$ is a directed path in the graph $\Gamma$ from $\boldsymbol{C}$ to $\boldsymbol{GL}$ of constrained length (to be explained below), together with a character of $A_r$ (with constraints depending on the path) and a representation of $\GL_m(\kk)$, with $m$ depending on the path. The vertex $\boldsymbol{C}$ stands for the groups \uri{$C_{i,j}$ and $C^1_{i,j}$}, the vertex $\boldsymbol{D}$ stands for the affine groups $\GL_m(\kk) \ltimes \kk^m$ and the vertex~$\boldsymbol{GL}$ stands for $\GL_m(\kk)$. More precisely, we have the following. 

\begin{theorem} The irreducible representations of $C_{r,n}$ are parameterised by triples $\rho(\gamma,\psi,Y)_{r,n}$, where 
\begin{itemize}
\item $\gamma$ is one of the following paths
\[
(\mathbf{1}_H), \quad (x), \quad (y), \quad (\underbrace{c,\ldots,c}_{1 \le t \le n-r},\mathbf{1}_H), \quad (\underbrace{c,\ldots,c}_{1 \le t \le n-r},x),  \quad (\underbrace{c,\ldots,c}_{1 \le t \le n-r},y),\quad (\underbrace{\underbrace{c,\ldots,c}_{t},\underbrace{d,\ldots,d}_{s}}_{1 \leq2t+s \leq n-r}).
\] 
We call these paths {\bf \em permissible}.
\item $\psi \in \widehat{A}$ is {\bf \em compatible} with $\gamma$ (namely, its restriction to $A^{r-1}_r$ is as specified below), and $Y \in \Irr\left(\GL_{\delta(\gamma)}(\kk)\right)$. 
\[
\begin{array}{  | l    | l  | l |  l | l | l |}
\hline
 & \gamma & \delta(\gamma) & \psi_{|A^{r-1}_r} & \star & f(v)  \\
 \hline
 (i) &  (\mathbf{1}_H) & n-r & 0 & & \\
(ii) & (x)  & n-r & x  &  & \\
(iii) &  (y) & n-r-1 & 0 & & y  \\
\hline
(iv) &  ({c,\ldots,c},\mathbf{1}_H) & n-r-2t & 0 & 0 & \\
(v) &  ({c,\ldots,c},x) & n-r-2t& 0 & x &\\
 (vi) & ({c,\ldots,c},y) & n-r-2t-1& 0 &0 & y \\
\hline
 (vii)&  ({c,\ldots,c},{d,\ldots,d}) & n-r-2t-s & 0 & & \\
\hline
\end{array}
\]
Here the first column is the path $\gamma$, the second indicates the dimension, and the third is the restriction of $\psi$ to $A^{r-1}_r$.
\end{itemize}
 
\end{theorem}

\begin{proof} We split the irreducible representations of $C_{r,n}$ according to the orbits of their restrictions to $H$. The first three types of irreducible representations of $C_{r,n}$ lie above the first three orbits appearing in Lemma~\ref{C.orbits.on.H}.     

\begin{enumerate}
\item[(i)] The trivial representation $\mathbf{1}_H$ always extends and the representations above that orbit are determined by representations of the quotient.  

\item[(ii)] Representations that lie over $W_x$ for $x \in \kk^\times$ can be extended to representations of $C_{r,n}$ in the following way. The character of $A^{r-1}_r$ represented by $x$ can be extended to a (one-dimensional) character $\psi$ of the abelian group $A_r$.
We can define an action of $A_r\times \GL_{n-r}(\kk)$ on $W_x$ by setting 
$(a,g)\cdot w_v = \psi(a)w_{g(v)}$. A direct calculation shows that together with the action of $H$ on $W_x$ we get a representation of $C_{r,n}$.
Therefore, the relevant two-cocycle is trivial, and representations of $C_{r,n}$ lying over $W_x$ 
are in bijection with representations of the quotient $A_{r-1} \times \GL_{n-k}(\kk)$. The bijection is given by tensoring (an extension of) $W_x$ with (pullbacks to $C_{r,n}$ of) representations of the quotient.  

\item[(iii)]  Similar to the first two cases, with the extra choice of orbit parameterised by $y \in \kk^\times$.   

\end{enumerate}   

These types correspond to the two paths from $\boldsymbol{C}$ to $\boldsymbol{GL}$. Before we continue with the classification we digress to explain the (mathematical) inductive process. With orbits (1)-(3) excluded, we now look at the stabiliser of the orbit representatives marked in Lemma~\ref{C.orbits.on.H} for types (4)-(5). We have two types of possible stabilisers. One is $C^{1}_{2,n-r}$, associated again with the vertex $\boldsymbol{C}$ and reflecting the self-similar nature of the problem. The second is $A_{r-1} \times (\GL_{n-r-1}(\kk) \ltimes \kk^{n-r-1})$, associated with the vertex $\boldsymbol{D}$. Now, from $C^{1}_{2,n-r}$ one can similarly continue and get all five orbit types but with smaller dimensions. There is one difference though: note that from the second step the groups of type $\boldsymbol{C}$ that show up are always of the form $C^{1}_{2,j}$, but that does not change the process. The edge $c$ in the graph accounts for repetitions of that type of orbit. After exhausting these iterations $t$ times, and then  

\begin{enumerate}
\item[(iv)] branching out to vertex $\boldsymbol{GL}$ with the trivial representation of the (smaller) Heisenberg group that occurs in the (last) group of type $\boldsymbol{C}$ in the chain. 

\item[(v)] branching out to vertex $\boldsymbol{GL}$  with character $x \in \kk^\times$ of the centre of the (smaller) Heisenberg group that occurs in the (last) group of type $\boldsymbol{C}$ in the chain.

\item[(vi)] branching out to vertex $\boldsymbol{GL}$ with choice $y \in \kk^\times$ of and orbit of representations of the (smaller) Heisenberg group that occurs in the (last) group of type $\boldsymbol{C}$ in the chain.

\item[(vii)] branching out to vertex $\boldsymbol{D}$ then iterating $s$ times before branching out to vertex $\boldsymbol{GL}$. In more detail: consider the action of $\GL_m(\kk)$ on (the dual of) $\kk^m$.  We get that representations of the semidirect product $\GL_{m}(\kk) \ltimes \kk^{m}$ come in two families, those that correspond to the zero orbit or the rest. The zero orbit leads to a representation of $\GL_m(\kk)$ and the dense orbit leads to a representation of $\GL_{m-1}(\kk) \ltimes \kk^{m-1}$. We can thus brunch from a vertex of type $\boldsymbol{D}$ to a vertex of type $\boldsymbol{GL}$ or go back to $\boldsymbol{D}$  with a smaller dimension. 
\end{enumerate}
\end{proof}

\subsection{Interpretation of the representations via the multiplication}

Our next goal is to show that the product of elements from $\GN_\eta$ and $\GR^{\o,1}$ satisfies the Hopf formula.

\begin{proposition}\label{N.is.R.free} For every pair of natural numbers $n$ and $m$ with $n \ge r$, $\gamma$ a permissible path for $C_{r,n}$, $\psi \in \widehat{A}$ compatible with $\gamma$ and $W \in \Irr\left(\GL_{m}(\kk)\right)$ the following hold.  
\[
W \circ \rho(\gamma,\psi, \mathbf{1})_{r,n}=\rho(\gamma,\psi, W)_{r,n+m}.
\]
Here $\mathbf{1}$ stands for the trivial representation of the trivial group $\GL_0(\kk)$.
\smallskip

In particular, $\GN_\eta$ is a free module over $\GR^{\o,1}$ with basis 
\[
\mathcal{P}=\left\{\rho(\gamma,\psi, \mathbf{1})_{r,n} \mid n \ge r, \gamma~\text{permissible for $C_{r,n}$}, \psi~\text{compatible with $\gamma$} \right\}. 
\]
\end{proposition}

Before proving the proposition we shall see some corollaries. By identifying representations of $C_{r,n}$ with representations of $\GL_n(\Ow_2)$ by Clifford Theory, and using the multiplication $\nmult$, $\GN_\eta$ becomes a module and a comodule over $\GR^{\o,1}$.
Proposition~\ref{N.is.R.free} tells us that this is a free module with the prescribed basis. Since the comultiplication is dual to the multiplication, this also gives us a description of the comultiplication.
Thus within the algebra $\GR^{\o,2}$ we have $\Delta(\GN_\eta)\subseteq  \GR^{\o,1}\otimes \GN_\eta \oplus \GN_\eta\otimes \GR^{\o,1}$.

We have already proved in Proposition~\ref{prop:GR-basics} that the multiplication and comultiplication are commutative and cocommutative, respectively.
We will thus consider only the projection $\Delta_1$ of $\Delta$ to the first summand $\GR^{\o,1}\otimes \GN_\eta$ (which is orthogonal to the second and carries complete information), thus making $\GN_\eta$ into a left $\GR^{\o,1}$ comodule.
The map $\Delta_1:\GN_\eta\rightarrow \GR^{\o,1}\otimes \GN_\eta$ will be dual to the multiplication map $\GR^{\o,1}\otimes  \GN_\eta \rightarrow  \GN_\eta$.

Thus, we see that for any representation $\rho \in \mathcal{P}$ we have that $\Delta_1(\rho)=1\otimes \rho$.
Also, if $W$ is any representation in $\GR^{\o,1}$, then we can calculate $\Delta(W\nmult \rho)$.
First of all, it is clear that it will be supported on the subgroup $\GR^{\o,1} \otimes (\GR^{\o,1}\nmult\rho)$.
That is, it will contain only elements in $\GN_\eta$ spanned by $\rho$ and not by the other representations in $\mathcal{P}$.
\uri{This follows easily from the fact that if $\rho_1\ncong \rho_2$, then $W_1\nmult\rho_1$ and $W_2\nmult\rho_2$ will not have any common factor.}
Second, we have that
\[
\begin{split}
\langle\Delta(W \nmult \rho),W_1\otimes (W_2\nmult\rho)\rangle_{\GR^{\o,1}\otimes \GN_\eta} &= \langle W\nmult \rho,W_1\nmult W_2\circ\rho\rangle_{\GN_\eta} \\
&=\langle W,W_1 \nmult W_2\rangle_{\GR^{\o,1}} \\
&= \langle\Delta(W),W_1\otimes W_2\rangle_{\GR^{\o,1}\otimes \GR^{\o,1}}
\end{split}
\]
The second equality follows from the fact that we know that the map $\tau:\GR^{\o,1}\rightarrow \GN_\eta$ given by $W\mapsto W\nmult\rho$ preserves the inner product,
as it sends different irreducible representations to different irreducible representations.
We also have that 
\[
\langle \Delta_1(W)\uri{\nmult}\Delta_1(\rho),W_1\otimes (W_2\nmult\rho)\rangle_{\GR^{\o,1}\otimes \GN_\eta} = \langle \Delta_1(W),W_1\otimes W_2\rangle_{\GR^{\o,1}\otimes \GR^{\o,1}}.
\]
This follows from the explicit formula $\Delta_1(\rho)=1\otimes\rho$ for the comultiplication of $\rho$, together with the above consideration regarding the map $\tau$.
In particular, we see that $\Delta_1(W\nmult\rho)=\Delta_1(W) \uri{\nmult}\Delta_1(\rho)$, since the pairing of these two elements is the same with any basis element of $\GR^{\o,1}\otimes \GN_\eta$.
But this means that also $$\Delta_1(W_1\nmult (W_2 \nmult\rho))=\Delta_1(W_1\nmult W_2)\uri{\nmult}\Delta_1(\rho) = \Delta_1(W_1)\uri{\nmult}\Delta_1(W_2)\uri{\nmult}\Delta_1(\rho) = \Delta_1(W_1)\uri{\nmult}\Delta_1(W_2\nmult \rho),$$
so we have the following:
\begin{corollary}
	The Hopf formula is valid when we multiply a representation from $\GR^{\o,1}$ with a representation from $\GN_\eta$.
\end{corollary}
\begin{remark}
	Note that the fact that we used the projection $\Delta_1$ rather than $\Delta$ does not make any difference, and everything still holds exactly in the same way for $\Delta$ \uri{because $\Delta=\Delta_1+\sigma \Delta_1$ where $\sigma$ is the flip}. 
\end{remark}

\subsection{Proof of Proposition~\ref{N.is.R.free}}  
We now turn to prove Proposition~\ref{N.is.R.free}. We shall make a substantial use of the results described in \S\ref{subsec:clifford}, in particular the reduction of the multiplication to the level of centralisers. 
Moreover, a careful analysis reveals that all the representations in $\mathcal{P}$ are monomial, that is, induced from one-dimensional characters of certain subgroups.
The proof of the proposition is based on a careful case by case analysis (we will do only the first two cases. The rest of the cases are pretty much the same).

\subsubsection{Representations of type (ii)}
\begin{lemma}
Let $\psi$ be a character of $A_r$, with $\psi|_{A^{r-1}_r}\neq 1$, and let $W$ be an irreducible representation of $\GL_{m}(\kk)$. 
Then 
\begin{equation}\label{case.1}
W\nmult \rho((x),\psi,\mathbf{1})_{r,r} = \rho((x),\psi,W)_{r,r+m}.
\end{equation}
\end{lemma}
\begin{proof}
Notice that $\rho((x)),\psi,\mathbf{1})$ is just the one-dimensional representation $\psi$ of $A_r=C_{r,r}$. We also note that, using the compatibility with Clifford theory (see \S\ref{subsec:clifford}),  
the multiplication on the left hand side of \eqref{case.1} is isomorphic to $\Ind(\psi \circ W)$, where 
\[
\psi \circ W = \C C_{r,r+m}e_Ve_{V^*}\otimes_{A_r\times \GL_{m}(\kk)} \psi \otimes W.
\]

We use the fact that $V$ and $V^*$ satisfy the following commutation relation:
\[
[v,f] = 1+f(v)t^{r-1}
\]
(one should understand the right hand side as an element in the group, and not as a sum of group elements).
Since the restriction of $\psi$ to $A_r^{r-1}$ is nontrivial, we can use Lemma~\ref{reduction} (reduction of idempotents) here, to eliminate $e_V$.
So the multiplication is 
\[
\C C_{r,m+r}e_{V^*}\otimes_{A_r\times \GL_{m}(\kk)} \psi\otimes W.
\]
In other words, this representation is obtained by inflating the representation $\psi\otimes W$ to $A_r \times \GL_{m}(\kk) \cdot V^*$ and then inducing to $C_{r,m+r}$.
		
The representation $\rho((x),\psi,W)_{r,n}$ with $n=m+r$ is given in the following way:
As we have mentioned above, there exists an irreducible representation $W_{\psi}$ of $H$ whose restriction to $A^{r-1}_r$ is a multiple of $\psi$.
This representation has a basis $\{w_{v}\}_{v\in V}$. The subgroup $A_r$ acts on $W_{\psi}$ by the character $\psi$.
We have explained how to extend $W_{\psi}$ to a representation of $C_{r,n}$.
The representation $\rho((x),\psi,W)_{r,n}$ is given by $W_{\psi}\otimes W$, where $W$ is a representation of $C_{r,n}$ by inflation.
The isomorphism between $\rho((x),\psi,W)_{r,n}\cong W_{\psi}\otimes W$ and $W\nmult \rho((x),\psi,\mathbf{1})_{r,r}$ is then given by
\[
w_v\otimes w\mapsto ve_{V^*}\otimes w.
\]
A direct verification shows that this is an isomorphism.
\end{proof}

\subsubsection{Representations of type (v)}
The proof of Proposition~\ref{N.is.R.free} for all the other types is very similar.
To make the proof as simple as possible we concentrate on representations of type (v), but the proof will work verbatim to all other types as well.
We will need to consider a different character and induction from different subgroup, but the proof will be the same.

The idea is the following. By Clifford theory we know that the irreducible representations of $C_{r,n}$ are induced from some very specific subgroups.
In case of the representations in $\mathcal{P}$, they are induced from one-dimensional representations.
Specifically, consider the representation $\ti{W} = \rho((c,\ldots, c,(x), \psi, W)_{r,n}$ of type~(v) with $c$ occurring $t$ times and $m=\delta(\gamma)=n-r-2t$. 
After applying a permutation matrix, the representation $\ti{W}$ is induced from the subgroup $Z\leq \uri{C_{r,n} \leq}\GL_n(\kk)$ which we shall now describe:
\begin{equation} Z = \left\{
\begin{pmatrix}
M_{11} & M_{12} & M_{13} \\  
M_{21} & M_{22} & M_{23} \\
\uri{\underbrace{M_{31}}_{r}} & \uri{\underbrace{M_{32}}_{2t}} & \uri{\underbrace{M_{33}}_{m}} \\
\end{pmatrix} \right\},
\end{equation}
where: \\
- The matrix $M_{11}\in \Mat_{r\times r}(\kk)$ is in $A_r$.\\
- The matrix $M_{12}\in \Mat_{r\times 2t}(\kk)$ is of the form
\begin{equation}\label{eq:f}
 \begin{pmatrix}
  f \\ 0 \\ \vdots \\ 0
 \end{pmatrix}
\end{equation}
(only the first row is nonzero).\\
- The matrix $M_{13}\in \Mat_{r\times m}(\kk)$ is of the form 
\begin{equation}
 \begin{pmatrix}
  g \\ 0 \\ \vdots \\ 0
 \end{pmatrix}
\end{equation}(again, only the first row is nonzero).\\
- The matrix $M_{21}\in \Mat_{2t\times r}(\kk)$ 
is of the form
\begin{equation}\label{eq:v}
 \begin{pmatrix}
  0 & 0 \ldots 0 & v
 \end{pmatrix}
\end{equation}
- The matrix $M_{31}\in \Mat_{m\times r}(\kk)$  is of the form 
\begin{equation}
 \begin{pmatrix}
  0 & 0 \ldots 0 & w
 \end{pmatrix}
\end{equation}
- The matrix $M_{22}\in \Mat_{2t\times 2t}(\kk)$ is upper triangular, with ones on the diagonal.\\
- The matrix $M_{23}\in \Mat_{2t\times m}(\kk)$  is of the form: 
\begin{equation}
\begin{pmatrix}
g_1 \\
\vdots \\
g_t \\
0 \\
\vdots\\
0
\end{pmatrix}
\end{equation}
- The matrix $M_{32}\in \Mat_{m\times 2t}(\kk)$ is of the form:
\begin{equation}
\begin{pmatrix}
0 & 0 \cdots & 0 & w_{t+1} & w_{t+2} & \cdots & w_{2t}
\end{pmatrix}
\end{equation}
- The matrix $M_{33}\in \GL_m(\kk)$ is some invertible matrix. 

\smallskip
\noindent \uri{We refer to these blocks as the $(i,j)$-blocks of $Z$, with $1 \le i,j \le 3$, and denote them by $Z_{ij}$.} By projecting onto the $(3,3)$ part, the group $Z$ has $\GL_m(\kk)$ as a quotient.
We get the following split short exact sequence: 
$$1\rightarrow Z_0\rightarrow Z\rightarrow \GL_m(\kk)\rightarrow 1,$$
%If $1\leq i,j\leq 3$ we will denote by $Z_{ij}$ the subgroup of $Z$ which contains all matrices which have nontrivial part only in $M_{ij}$
%(in this context, trivial means zero in case $i\neq j$ and the identity matrix in case $i=j$).
\uri{with $Z_0=Z_{11}Z_{12}Z_{13}Z_{21}Z_{22}Z_{23}Z_{31}Z_{32}$ which is easily seen to be normal in $Z$}. 
We can now describe explicitly the representation $\ti{W}$. Let $Z_1=\uri{Z_{12}Z_{22}Z_{21}Z_{22}}$ denote the subgroup of all block matrices in $Z$ of the form
\begin{equation}
\begin{pmatrix}
* & * & 0 \\
* & * & 0  \\
0 & 0 & I 
\end{pmatrix}
\end{equation}
Let $a_1,\ldots,a_{2t-1}$ denote the elements on the secondary diagonal of an element in $Z_{22}$, that is 
\uri{\begin{equation}
\begin{pmatrix}
1 & a_1 & * & \cdots  & * \\
0 & 1 & a_2 & \cdots & * \\
\vdots & \vdots & \ddots  & \ddots &*\\
0 & 0 & 0 & \ddots & a_{2t-1}\\
0 & 0 & 0 & \cdots & 1
\end{pmatrix}.
\end{equation}}
	
Consider the following linear character $\ti{\psi}$ of the group $Z_1$. The restriction of $\ti{\psi}$ to $A$ is just $\psi$.
The restriction of $\ti{\psi}$ to $Z_{12}$ is given by $\ol{\phi}(f_1)$, \uri{where $f_1$ is the first component in \eqref{eq:f}}.
The restriction of $\ti{\psi}$ to $Z_{21}$ is given by $\ol{\phi}(v_{2t})$, \uri{where $v_{2t}$ is the last component in \eqref{eq:v}}.
The restriction of $\ti{\psi}$ to $Z_{22}$ is given by 
\[
\ol{\phi}(a_1 + a_2 + \cdots + a_{t-2} + xa_{\uri{t-1}} + a_{t}+\cdots + a_{2t-1}).
\]
We thus have the representation $\ti{\psi}\otimes W$ of $Z_1\times \GL_m(\kk)$.
By carefully checking the definition of $\ti{W}$, we see that it can be described as the representation of $C_{r,n}$ induced from the $Z$ representation
\[
\C Z e_{Z_{23}}e_{Z_{13}}e_{Z_{32}}e_{Z_{31}}\ot_{Z_1\times \GL_m(\kk)}\ti{\psi}\ot W,
\]
so we have that 
\[
\ti{W} = \C C_{r,n}e_{Z_{23}}e_{Z_{13}}e_{Z_{32}}e_{Z_{31}}\ot_{Z_1\times \GL_m(\kk)}\ti{\psi}\ot W.
\]

On the other hand, by using the formula for the comultiplication, we get that 
\begin{equation}\label{W.circ.rho}
W\nmult \rho((c,\ldots, c,(x), \psi, \mathbf{1})_{r,r+2t} = \Ind_{C_{r,n}}^{\GL_n(\o_2)}\left(\C C_{r,n} e_{\ol{Z}_{23}}e_{Z_{13}}e_{\ol{Z}_{32}}e_{Z_{31}}\ot_{Z_1\times \GL_m(\kk)}\ti{\psi}\ot W\right),
\end{equation}
\uri{where 
\[
\ol{Z}_{23}=\left\{\left(\begin{matrix} I_r & 0 & 0 \\ 0 & I_{2t} & M_{23} \\ 0 & 0 & I_m\end{matrix}\right) \mid M_{23} \in \Mat_{2t \times m}(\kk) \right\},
\]
}
and thus $\ol{Z}_{23}$ is strictly bigger than $Z_{23}$ unless $t=0$. We have that the representation of $C_{r,n}$ occurring on the right hand side \uri{of \eqref{W.circ.rho}} is a subrepresentation of 
\begin{equation}\label{1st.rep}
\C C_{r,n} e_{Z_{23}}e_{Z_{13}}e_{\ol{Z}_{32}}e_{Z_{31}}\ot_{Z_1\times \GL_m(\kk)}\ti{\psi}\ot W.
\end{equation}
By using the commutation of idempotents for the groups $Z_{23}Z_{13}$ and $\ol{Z}_{32}Z_{31}$ we get that this representation is isomorphic to 
\[
\C C_{r,n} e_{\ol{Z}_{32}}e_{Z_{31}}e_{Z_{23}}e_{Z_{13}}\ot_{Z_1\times \GL_m(\kk)}\ti{\psi}\ot W,
\]
which is a subrepresentation of 
\[
\C C_{r,n} e_{Z_{32}}e_{Z_{31}}e_{Z_{23}}e_{Z_{13}}\ot_{Z_1\times \GL_m(\kk)}\ti{\psi}\ot W.
\]
By using the commutation of idempotents again, the last representation is isomorphic with the representation~\eqref{1st.rep}.
So $W \nmult \rho(c,\ldots, c,(x), \psi, \mathbf{1})_{r,r+2t}$ is always a subrepresentation of $\ti{W} = \rho((c,\ldots, c,(x), \psi, W)_{r,n}$.
For an irreducible $W$, we have that $\ti{W}$ is irreducible and as $W \nmult \rho(c,\ldots, c,(x), \psi, \mathbf{1})_{r,r+2t}\neq 0$ by \cite[Theorem 2.23(1)]{CMO1}.
It then follows that for $W$ irreducible the two representations are isomorphic, and therefore the same holds for a general~$W$.
This finishes the proof for representations of type (v).

In all the other cases the computation works almost the same.
Two things change. The first is the restriction of the character $\ti{\psi}$ to $Z_{22}$.
\begin{itemize}
\item  In case (iv) the size of the (2,2) block of the matrix is again $2t \times 2t$, and the character is 
\[
\ol{\phi}(a_1+a_2+ \cdots + a_{t-1}+ a_{t+1} + \cdots + a_{2t-1})
\] 
(so it is the same as before, only $x$ becomes zero).\\
\item In case (vii) the size of (2,2) block of the matrix is $(2t+s) \times (2t+s)$. The characters are given by 
\[
\ol{\phi}(a_1+ a_2+ \cdots + \uri{a_{t-2} + a_{t}}+\cdots+a_{2t+s-1})
\]
 and 
\[
\ol{\phi}(a_1+ a_2+ \cdots a_{t+s-1} + a_{t+s+1}+\cdots+a_{2t+s-1}).
\] 
\item In case (vi) the size of the (2,2) block of the matrix is $(2t+1) \times (2t+1)$, and the character is given by 
\[
\ol{\phi}(a_1+\cdots +a_{t-1}+ya_t+a_{t+1}+\cdots +a_{2t}).
\]
\end{itemize}
The second thing that changes, besides the character, are the subgroups $Z_{23}$ and $Z_{32}$.
By using Clifford Theory again, it is possible to track these subgroups without too much effort.
Notice that all the representations in $\mathcal{P}$ are induced from characters of certain unipotent subgroups.

\uri{
\subsection*{Acknowledgments}  
%The first and second authors were partly supported by the Danish National Research Foundation through the Centre for Symmetry and Deformation (DNRF92). 
%The first author was also supported by fellowships from the Max Planck Institute for Mathematics in Bonn, and from the Radboud Excellence Initiative at Radboud University Nijmegen.
%The second author was also supported by the Research Training Group 1670 ``Mathematics Inspired by String Theory and Quantum Field Theory.'' 
The third author was supported by the Australian Research Council (ARC FT160100018). 

}


\begin{thebibliography}{10}

\bibitem{AOPS}
{\sc A.-M. Aubert, U.~Onn, A.~Prasad, and A.~Stasinski}, {\em On cuspidal
  representations of general linear groups over discrete valuation rings},
  Israel J. Math., 175 (2010), pp.~391--420.

\bibitem{BZ_reductive}
{\sc I.~N. Bernstein and A.~V. Zelevinsky}, {\em Induced representations of
  reductive {${\germ p}$}-adic groups. {I}}, Ann. Sci. \'{E}cole Norm. Sup.
  (4), 10 (1977), pp.~441--472.

\bibitem{Brown1994}
{\sc K.~S. Brown}, {\em Cohomology of groups}, vol.~87 of Graduate Texts in
  Mathematics, Springer-Verlag, New York, 1994.
\newblock Corrected reprint of the 1982 original.

\bibitem{Carayol}
{\sc H.~Carayol}, {\em Repr\'esentations cuspidales du groupe lin\'eaire}, Ann.
  Sci. \'Ecole Norm. Sup. (4), 17 (1984), pp.~191--225.

\bibitem{Casselman}
{\sc W.~Casselman}, {\em The restriction of a representation of {${\rm
  GL}\sb{2}(k)$} to {${\rm GL}\sb{2}({\mathfrak o})$}}, Math. Ann., 206 (1973),
  pp.~311--318.

\bibitem{CMO2}
{\sc T.~{Crisp}, E.~{Meir}, and U.~{Onn}}, {\em {Principal series for general
  linear groups over finite commutative rings}}, Apr. 2017.

\bibitem{CMO1}
{\sc T.~{Crisp}, E.~{Meir}, and U.~{Onn}}, {\em {A variant of Harish-Chandra
  functors.}}, {J. Inst. Math. Jussieu}, 18 (2019), pp.~993--1049.

\bibitem{DM}
{\sc F.~Digne and J.~Michel}, {\em Representations of finite groups of {L}ie
  type}, vol.~21 of London Mathematical Society Student Texts, Cambridge
  University Press, Cambridge, 1991.

\bibitem{Hill_Jord}
{\sc G.~Hill}, {\em A {J}ordan decomposition of representations for {${\rm
  GL}_n(O)$}}, Comm. Algebra, 21 (1993), pp.~3529--3543.

\bibitem{Howe-Moy}
{\sc R.~Howe}, {\em Harish-{C}handra homomorphisms for {${\germ p}$}-adic
  groups}, vol.~59 of CBMS Regional Conference Series in Mathematics, Published
  for the Conference Board of the Mathematical Sciences, Washington, DC; by the
  American Mathematical Society, Providence, RI, 1985.
\newblock With the collaboration of Allen Moy.

\bibitem{Howlett-Lehrer_HC}
{\sc R.~B. Howlett and G.~I. Lehrer}, {\em On {H}arish-{C}handra induction and
  restriction for modules of {L}evi subgroups}, J. Algebra, 165 (1994),
  pp.~172--183.

\bibitem{Isaacs}
{\sc I.~M. Isaacs}, {\em Character theory of finite groups}, Academic Press
  [Harcourt Brace Jovanovich, Publishers], New York-London, 1976.
\newblock Pure and Applied Mathematics, No. 69.

\bibitem{Macdonald}
{\sc I.~G. Macdonald}, {\em Zeta functions attached to finite general linear
  groups}, Math. Ann., 249 (1980), pp.~1--15.

\bibitem{OnnSingla}
{\sc U.~Onn and P.~Singla}, {\em {On the unramified principal series of
  $\mathrm{GL}(3)$ over non-Archimedean local fields.}}, {J. Algebra}, 397
  (2014), pp.~1--17.

\bibitem{Zelevinsky}
{\sc A.~V. Zelevinsky}, {\em Representations of finite classical groups},
  vol.~869 of Lecture Notes in Mathematics, Springer-Verlag, Berlin-New York,
  1981.
\newblock A Hopf algebra approach.

\end{thebibliography}
\end{document}